\documentclass[reqno,11pt]{amsart}

\usepackage{amscd, amssymb , amsmath,epsfig,subfig,fullpage}
\usepackage{epsfig}
\usepackage[all]{xy}
\usepackage{verbatim}

\newtheorem{Df}{Definition}
\newtheorem{definition}[Df]{Definition}
\newtheorem{theorem}[Df]{Theorem}

\newtheorem{prop}[Df]{Proposition}
\newtheorem{proposition}[Df]{Proposition}
\newtheorem{lemma}[Df]{Lemma}

\newtheorem{remark}[Df]{Remark}

\newtheorem{corollary}[Df]{Corollary}
\newtheorem{conjecture}[Df]{Conjecture}

\newcommand{\ev}{\ensuremath{d}}
\newcommand{\tev}{d'}
\newcommand{\diag}{\operatorname{diag}}
\newcommand{\pp}[1]{{\left(#1\right)}}
\newcommand{\ro}{{r}}
\newcommand{\rk}{{n}}
\newcommand{\Uqg}{\ensuremath{\mathcal U}}
\newcommand{\UqgH}{\ensuremath{\Uqg^{H}}}
\newcommand{\UK}{\ensuremath{\Uqg_{0}}}
\newcommand{\UH}{\ensuremath{\Uqg_{\h}}}
\newcommand{\Uqn}[1]{\Uqg_{#1}}
\newcommand{\ZU}{{\mathcal Z}}
\newcommand{\ZUo}{{\mathcal Z^{0}}}
\newcommand{\cat}{\mathcal{C}}
\newcommand{\catt}{{\mathcal{C}}}
\newcommand{\catU}{\mathcal{D}}
\newcommand{\catH}{\mathcal{D}^{H}}
\newcommand{\catHr}{\mathcal{D}^{\theta}}
\newcommand{\RH}{{\mathcal H}}
\newcommand{\Hh}{H\hspace{-1ex}\,H^h}
\newcommand{\qR}{\check{R}}
\newcommand{\oR}{{\mathcal R}}
\newcommand{\uH}{{u_{\mathcal H}}}
\newcommand{\qu}{{u}}
\newcommand{\ou}{{\nu}}
\newcommand{\La}{{L_W}}
\newcommand{\Proj}{{\mathsf{Proj}}}

\newcommand{\st}{\varphi} 
\newcommand{\Int}{\operatorname{Int}}
\newcommand{\tet}{\mathcal{T}}
\newcommand{\ang}[1]{{\left\langle{#1}\right\rangle}}
\newcommand{\angg}[1]{{\left({#1}\right)}}
\renewcommand{\SS}{\Sigma}
\newcommand{\QQ}{Q}
\newcommand{\slop}[1]{\Lop^{#1\frac12}}
\newcommand{\srop}[1]{\Rop^{#1\frac12}}
\newcommand{\scop}[1]{C^{#1\frac12}}
\newcommand{\Y}{{Y}}
\newcommand{\YY}{{\mathcal Y}}
\newcommand{\Tb}{\mathcal{T}^{\partial}}
\newcommand{\col}{{\Phi}}
\newcommand{\Cob}{{\mathcal Cob^\Delta}}
\newcommand{\Vect}{{\mathcal Vect}}
\newcommand{\os}{\eta}
\newcommand{\Cm}{{\mathcal Cob}}
\newcommand{\ma}{{\mathsf m}}
\newcommand{\Cy}{{\mathsf C}}
\newcommand{\Eq}{{\mathcal E}}
\newcommand{\bubble}{$H$-bubble}
\newcommand{\Pachner}{$H$-Pachner}
\newcommand{\lune}{$H$-lune}
\renewcommand{\d}[1]{{\widetilde #1}}
\newcommand{\FF}{\mathcal{F}}
\newcommand{\Ffun}{\ensuremath{\mathsf{F}}}
\newcommand{\Gfun}{\ensuremath{\mathsf{G}}}
\newcommand{\mt}{\operatorname{\mathsf{t}}}
\newcommand{\rr}{{\rho}}

\newcommand{\wta}{\lambda}
\newcommand{\wtb}{\mu}
\newcommand{\h}{\ensuremath{\mathfrak{h}}}
\newcommand{\roots}{\Delta}
\newcommand{\e}{\operatorname{e}}
\newcommand{\wb}{\overline}
\newcommand{\wt}{\widetilde}
\newcommand{\Graph}{\operatorname{Gr}}
\newcommand{\Rib}{\operatorname{Rib}}
\newcommand{\Gr}{\ensuremath{\mathcal{G}}}
\newcommand{\X}{\ensuremath{\mathcal{X}}}
\newcommand{\al}{{\operatorname{alg}}}
\newcommand{\C}{\ensuremath{\mathbb{C}}}
\newcommand{\Z}{\ensuremath{\mathbb{Z}}}
\newcommand{\R}{\ensuremath{\mathbb{R}}}
\newcommand{\N}{\ensuremath{\mathbb{N}}}
\newcommand{\Uhg}{\ensuremath{U_{h}(\g)}}
\newcommand{\g}{\ensuremath{\mathfrak{g}}}
\newcommand{\slt}{\ensuremath{\mathfrak{sl}(2)}}

\newcommand{\End}{\operatorname{End}}
\newcommand{\Hom}{\operatorname{Hom}}
\newcommand{\unit}{\ensuremath{\mathbb{I}}}
\newcommand{\Id}{\operatorname{Id}}
\newcommand{\qdim}{\operatorname{qdim}}
\newcommand{\FK}{{\Bbbk}}
\newcommand{\qn}[1]{{\left\{#1\right\}}}
\newcommand{\qN}[1]{{\left[#1\right]}}
\newcommand{\qd}{\operatorname{\mathsf{d}}}
\newcommand{\qe}{\operatorname{\mathsf{\delta}}}
\newcommand{\qee}{{\qd^{\frac12}}}
\newcommand{\A}{\ensuremath{\mathsf{A}}}
\newcommand{\At}{\ensuremath{{\mathsf{A}_{\Proj}}}}

\newcommand{\sll}{\mathfrak{sl}}
\newcommand{\ms}[1]{\mbox{\tiny$#1$}}

\newcommand{\LL}{\mathcal{L}}
\newcommand{\Lt}{\mathcal{L}_{\At}}
\newcommand{\bb}{\operatorname{\mathsf{b}}}
\renewcommand{\wp}{{\Phi}}
\newcommand{\states}{\operatorname{St}}
\newcommand{\vect}{\overrightarrow}
\newcommand{\cntr}{{\operatorname {cntr}}}
\newcommand{\MG}[1]{{\mathcal M}(#1,\Gr)}
\newcommand{\sqop}[2]{q^{#1}}
\newcommand{\Lop}{L}
\newcommand{\Rop}{R}
\renewcommand{\Re}{{\Rop}^{\frac12}} 
\newcommand{\Le}{{\Lop}^{\frac12}} 
\newcommand{\Ce}{{C}^{\frac12}} 
\newcommand{\T}{{\mathcal T}}
\newcommand{\Kas}{\mathsf{K}} 
\newcommand{\wsqop}{\widetilde q}
\newcommand{\act}[1]{\centerdot{#1}}

\newcommand{\sjtop}[6]{\left|\begin{array}{ccc}#1 & #2 & #3 \\#4 & #5 &
      #6\end{array}\right|}

\newcommand{\epsh}[2]
         {\begin{array}{c} \hspace{-1.3mm}
        \raisebox{-4pt}{\epsfig{figure=#1,height=#2}}
        \hspace{-1.9mm}\end{array}}

\newcommand{\pic}[2]{
  \setlength{\unitlength}{#1}
  {\begin{array}{c} \hspace{-1.3mm}
        \raisebox{-4pt}{#2}
        \hspace{-1.9mm}\end{array}}}

\begin{document}
\let\co=\verbatim \let\endco=\endverbatim

\title
{Topological invariants from non-restricted quantum
  groups}

\author{Nathan Geer}
\address{Mathematics \& Statistics\\
  Utah State University \\
  Logan, Utah 84322, USA}
 \email{nathan.geer@usu.edu}

\author{Bertrand Patureau-Mirand}
\address{LMAM, universit\'e de Bretagne-Sud, universit\'e europ\'eenne de
  Bretagne, BP 573, 56017 Vannes, France }
\email{bertrand.patureau@univ-ubs.fr}
                         ́                         ́

\date{\today}

\begin{abstract}
  We introduce the notion of a relative spherical category.  We prove that
  such a category gives rise to the generalized Kashaev and Turaev-Viro-type
  3-manifold invariants defined in \cite{GKT} and \cite{GPT2}, respectively.
  In this case we show that these invariants are equal and extend to what we
  call a relative Homotopy Quantum Field Theory which is a branch of the
  Topological Quantum Field Theory founded by E.~Witten and M.~Atiyah.  Our
  main examples of relative spherical categories are the categories of finite
  dimensional weight modules over non-restricted quantum groups considered by
  C.~De~Concini, V.~Kac, C.~Procesi, N.~Reshetikhin and M.~Rosso.  
These categories 
are not semi-simple and 
  have an infinite number of non-isomorphic
  irreducible modules all having vanishing quantum dimensions.  We also show
  that these categories have associated ribbon categories which gives rise to
  re-normalized link invariants.  In the case of $\sll_2$ these link
  invariants are the Alexander-type multivariable invariants defined by
  Y.~Akutsu, T.~Deguchi, and T.~Ohtsuki \cite{ADO}.
\end{abstract}

\maketitle
\setcounter{tocdepth}{3}

\section*{Introduction}
A principal feature of quantum topology is its interplay between tensor
categories and low-dimensional topology.  The fundamental example of such
interplay is the modular category formed from finite dimensional
representation of the restricted quantum group $U_q(\sll_2)$ at a root of
unity and the corresponding Reshetikhin--Turaev and Turaev--Viro 3-manifold
invariants.  Loosely speaking, this modular category is the category of all
representations over $U_q(\sll_2)$ quotient the ideal of representation with
zero quantum dimension.  This category is semi-simple with a finite number of
isomorphism classes of irreducible representations all having non-vanishing
quantum dimensions.  Generalizations to other semi-simple categories were
done by several authors including H.~Andersen, V.~Turaev, H.~Wenzl,
J.~Barrett, B.~Westbury, A.~Ocneanu, S.~Gelfand, D.~Kazhdan, and~others.
  
Less progress has been made in categories which 
have
simple
objects with vanishing quantum dimensions.  This is true even though the
Volume Conjecture naturally arises in such a setting (see \cite{MM}).  In the
case of $\sll_2$, 3-manifold invariants arising from representations with zero
quantum dimensions have been studied by two approaches: 1) R.~Kashaev's
geometric and physical construction in \cite{Kv,Kv94}, which was extended by
S.~Baseilhac and R.~Benedetti \cite{BB}
and 2) the topological and algebraic Turaev--Viro-type invariants of N.~Geer,
B.~Patureau-Mirand, V.~Turaev \cite{GPT2}.  In \cite{GKT}, the notion of a
$\hat \Psi$-system is used to generalize Kashaev's construction to a
categorical setting.

This paper has five main results: First, we show that the non-restricted
quantum groups associated to simple Lie algebras considered by C.~De~Concini
and V.~Kac, in \cite{DK}, admit Turaev--Viro-type invariants arising from
representations with zero quantum dimensions.  Second, we prove that such
quantum groups also lead to $\hat \Psi$-systems and so from \cite{GKT} admit
generalized Kashaev invariants.  Third, we show that in this situation the
generalized Kashaev and TV-type invariants are equal.  Fourth, we extend the
above mentioned TV-type invariants to a kind of Homotopy Quantum Field Theory.
Finally, we prove that certain representations (with zero quantum dimensions)
over non-restricted quantum groups give rise to re-normalized
Reshetikhin--Turaev link invariants.  All of these results are proved in a
general categorical language which also contains as examples the usual modular
categories formed from representations of restricted quantum groups.
    
This paper has two major components.  The aim of first component is to
introduce and study ``relative spherical categories.''  These categories are
generalizations of the usual modular categories associated to restricted
quantum groups (see Theorem \ref{T:UsualCatGSph}).  However, in general they
are not necessarily semi-simple 
(only ``generically" semi-simple)
 and can have an infinite number of
non-isomorphic simple objects all having vanishing quantum dimensions.  In
Subsection~\ref{TTVV}, we will show that a relative spherical category $\cat$
leads to a Turaev--Viro-type invariant $TV$ of triples (a closed oriented
3-manifold $M$, a link in $M$, a conjugacy class of homomorphisms $\pi_1(M)\to
\Gr$) where $\Gr$ is a group which is part of the data of the category.  Then
in Subsection \ref{SS:PsiHatRelG} we will also show that $\cat$ gives rise to
a $\hat \Psi$-system.  Thus, the construction of \cite{GKT} gives a Kashaev
type invariant $\Kas$ and in Subsection \ref{SS:PsiHatRelG} we prove that
$TV=\Kas$.  The final piece of the first component of this paper is to show
that $TV$ can be extended to what we call a ``relative Homotopy Quantum Field
Theory'' (relative HQFT), see Section \ref{S:HQFT}.  This relative HQFT is
similar in nature to Turaev's \cite{Tu2010} Homotopy Quantum Field Theory
which is a branch of Topological Quantum Field Theory founded by E.~Witten and
M.~Atiyah.  However, we have cobordism with graphs inside.  Our relative HQFT
is also similar to the Quantum Hyperbolic Field Theory coming from the Borel
subalgebra of quantum $\sll_2$ defined in \cite{BB1}.  One can hope to make a
precise relationship with these different field theories.

The aim of the second component of this paper is to construct invariants of
links and 3-manifolds from representations over non-restricted quantum groups
which
are 
only ``generically''  
semi-simple and have vanishing quantum dimensions.  In the
series of papers \cite{DK}--\cite{DPRR}, C.~De~Concini, V.~Kac, C.~Procesi,
N.~Reshetikhin and M.~Rosso establish a deep body of work pertaining to
quantum groups and their representation theory, at odd ordered roots of unity.
In Section \ref{S:RepQG} we recall some of their results and prove that the
representations studied in these papers give rise to three topological
invariants: $TV$, $\Kas$ and a Reshetikhin--Turaev or Akutsu--Deguchi--Ohtsuki
type link invariant (see Theorem \ref{T:RTLink}).  The representations
considered here have zero quantum dimensions and so the usual associated
topological invariants are trivial.  To overcome this difficultly we show that
these representations are part of an ambidextrous pair with modified dimension
$\qd$.  This modified dimension is used to renormalize the usual link
invariants and the quantum $6j$-symbols.  Let us be more precise.  For an odd
ordered root of unity, let $\Uqg$ and $\UqgH$ be the quantum groups associated
to a fixed simple Lie algebra, described in Subsection \ref{SS:Uqg(H)}.  We
prove that the finite dimensional weight modules over $\UqgH$ contain a ribbon
category which admits an ambidextrous pair.  Combining this with a direct
calculation of the open Hopf link we prove that the results of \cite{GPT1} can
be applied and lead to link invariants (see Theorem \ref{T:RTLink} and
Proposition \ref{P:S'}).  Then incorporating all the results of Section
\ref{S:RepQG} we show that the finite dimensional weight modules over $\Uqg$
form a relative spherical category and so the results described in the
previous paragraph can be applied.  In particular, this shows that the
category of finite dimensional
weight modules over $\Uqg$ has several $\widehat \Psi$-systems (giving a
positive answer to Conjecture 42 of \cite{GKT} for $\Uqg$) and proves that the
corresponding generalized Kashaev and Turaev-Viro-type 3-manifold invariants
are equal.

The work of N.\ Geer was partially supported by the NSF grants DMS-0706725 and
DMS-1007197.  B.\ Patureau thanks Utah State University and the members of its
mathematics department for their generous hospitality.


\section{Relative $\Gr$-spherical categories and $\hat\Psi$-systems}\label{S:SphPsi}
In this section we give the basic categorical definitions used in this paper. 
In Subsection \ref{SS:LinearCat} recall some  fairly well known notations involving tensor categories.  In Subsections \ref{SS:ambi} and \ref{SS:hatPsi} we discuss two less known notions: ambidextrous traces and $\hat\Psi$-systems.  These two notions will be used to construct non-trivial invariants when quantum dimensions are zero.  In Subsection \ref{SS:GSpher} we give the notion of relative $\Gr$-spherical categories.  Such categories are used to deal with the fact that our categories are not necessarily semi-simple but only ``generically'' semi-simple.

\subsection{Linear tensor categories}\label{SS:LinearCat}
A \emph{tensor category} $\cat$ is a category equipped with a covariant
bifunctor $\otimes :\cat \times \cat\rightarrow \cat$ called the tensor
product, an associativity constraint, a unit object $\unit$, and left and
right unit constraints such that the Triangle and Pentagon Axioms hold.  Let
$\FK$ be an integral domain.  A tensor category $\cat$ is said to be
\emph{$\FK$-linear} if its hom-sets are $\FK$-modules, the composition and
tensor product of morphisms are $\FK$-bilinear, and $\End_\cat(\unit)$ is a
free $\FK$-module of rank one.  Then we identify $\FK= \End_\cat(\unit)$, via
the ring isomorphism $\FK\to \End_\cat(\unit), \; k \mapsto k\Id_\unit$ and
call $\FK$ the \emph{ground ring} of $\cat$.  Let $\FK^\times$ be the set of
invertible elements of $\FK$.  An object $V$ of $\cat$ is \emph{simple} if
$\End_\cat(V)= \FK \Id_V$.  An object which is a direct sum of simple objects
is called {\em semi-simple}.  For any simple object $V$ and $f\in\End(V)$,
there is a unique $x\in \FK$ such that $f=x\Id_{V}$. This $x$ is denoted $\ang
f$.

If $\cat$ is an $\FK$-linear category and $\{V_i\}_{i\in I}$ is a set of
simple objects numerated by a set $I$, for $i,j,k\in I$ we consider the
following $\FK$-modules
\begin{align*}
  H^{ijk}=\Hom_\cat(\unit, V_i\otimes V_j \otimes V_k),
\end{align*}
\begin{align*}
  \label{}
  H^{ij}_k = \Hom_\cat(V_k,V_i \otimes V_j) \quad {\text{and}} \quad H_{ij}^k
  &= \Hom_\cat(V_i \otimes V_j,V_k).
\end{align*}

A pivotal category is a tensor category with duality morphisms $b_{V} : \:\:
\unit \rightarrow V\otimes V^{*} , \quad d_{V}: \:\: V^*\otimes V\rightarrow
\unit , \quad b'_V: \:\: \unit \rightarrow V^{*}\otimes V\quad {\rm {and}}
\quad d'_V: \:\: V\otimes V^*\rightarrow \unit$ which satisfy compatibility
conditions (see for example \cite{GPT2}).  

Let $\cat$ be a pivotal category.   
A morphism $f:V_1\otimes{\cdots}\otimes V_n \rightarrow
W_1\otimes{\cdots}\otimes W_m$ in $\cat$ can be represented by a box and
arrows:
$$ 
\xymatrix{ \ar@< 8pt>[d]^{W_m}_{... \hspace{1pt}}
  \ar@< -8pt>[d]_{W_1}\\
  *+[F]\txt{ \: \; f \; \;} \ar@< 8pt>[d]^{V_n}_{... \hspace{1pt}}
  \ar@< -8pt>[d]_{V_1}\\
  \: }
$$
Boxes as above are called {\it coupons}.    By a $\cat$-colored ribbon graph
in an oriented surface $\SS$, we mean a graph embedded in $\SS$ whose
edges are colored by objects of $\cat$ and whose vertices lying in $\Int
\SS=\SS\setminus\partial \SS $ are thickened to coupons colored by morphisms
of~$\cat$.   Let $ \Graph_\cat$ be the category of $\cat$-colored ribbon
graphs in $\R \times [0,1]$ and $\Gfun: \Graph_\cat \to \cat$ be the
Reshetikhin-Turaev  $\FK$-linear functor preserving the duality morphism of
$\cat$ (see \cite{GPT2}).

A ribbon category is a pivotal category with a braiding $c_{V,W}:V\otimes W\to
W\otimes V$ and twist $\theta_V:V\to V$ which are compatible with the pivotal
structure (see \cite[Chapter 1]{Tu}).  If $\cat$ is a ribbon category, the
Reshetikhin-Turaev functor $\Gfun$ extend as $\Gfun:\Rib_\cat \to \cat$ where
$\Rib_\cat$ is the category of $\cat$-colored ribbon graphs in $\R^2 \times
[0,1]$.

\subsection{Ambidextrous pairs and traces}\label{SS:ambi}
We recall the process of re-normalizing colored ribbon graphs first used in
\cite{GP1,GP2} then generalized using the notion of ambidextrous entities in
\cite{GKP1,GPT1, GPT2, GPV}.

Let $\cat$ be a $\FK$-linear pivotal (resp., ribbon) category and let
$T\subset S^2$ (resp., $T\subset S^3$) be a closed $\cat$-colored ribbon
graph.  Let $e$ be an edge of $T$ colored with a simple object $V$ of
$\cat$. Cutting $T$ at a point of $e$, we obtain a $\cat$-colored ribbon graph
$T_V $ in $\R\times [0,1]$ (resp., in $\R^2\times [0,1]$) where
$\Gfun(T_V)\in\End(V)=\FK \Id_V$.  We call $T_V $ a \emph{cutting
  presentation} of $T$ and let $\ang{T_V} \, \in \FK$ denote the isotopy
invariant of $T_V$ defined from the equality $\Gfun(T_V )= \, \ang{ T_V}\,
\Id_V$.

Let $\A$ be a class of simple objects of $\cat$ and $\qd:\A\to\FK^\times$ be a
mapping such that $\qd(V)=\qd(V^*)$ and $\qd(V)=\qd(V')$ if $V$ is isomorphic
to $V'$.  Let $T$ be a closed $\cat$-colored ribbon graph which admit two
cutting presentations $T_V$ and $T_{V'}$ with both $V$ and $V'$ in $\A$.  Then
we consider if the following equality holds:
\begin{equation}
  \label{eq:ambi}
  \qd(V)\ang{T_V}=\qd(V')\ang{T_V'}.
\end{equation}

First, suppose that $\cat$ is a $\FK$-linear ribbon category and let $\LL_\A$
be the class of $\cat$-colored ribbon graphs in $S^3$ with at least one edge
colored by an element of $\A$.  Suppose that Equation \eqref{eq:ambi} holds
for all $T\in\LL_\A$.  Then we say that $(\A,\qd)$ is an {\em ambidextrous
  pair} or \emph{ambi} for short.

On the other hand, suppose that $\cat$ is a $\FK$-linear pivotal category.  In
\cite{GPT2} a smaller class of graphs is considered: A ribbon graph is
\emph{trivalent} if all its coupons are adjacent to 3 half-edges.  Let
$\tet_{\A}$ be the class of $\cat$-colored connected trivalent ribbon graphs
in $S^2$ such that the colors of all edges belong to $\A$.  Suppose that
Equation \eqref{eq:ambi} holds for all $T\in\tet_\A$.  Then we say that
$(\A,\qd)$ is a \emph{trivalent-ambidextrous pair} or \emph{t-ambi} for short.

If $(\A,\qd)$ is ambi (resp.,  t-ambi)  we define a
function $\Gfun':\LL_{\A} \rightarrow \FK$ (resp., $\Gfun':\tet_{\A}
\rightarrow \FK$) by
\begin{equation}\label{El+}
  \Gfun'(T)={  \qd}(V)\ang{ T_{V}}\, ,
\end{equation}
where $T_{V} $ is any cutting presentation of 
$T$ with $V\in \A$.  The above
definitions imply that $\Gfun'$ is well defined.
With Equation \eqref{El+} in mind, we call $\qd$ a modified quantum dimension.  

Let $\Proj$ be the full subcategory of $\cat$ consisting of projective
objects.  A \emph{trace} on $\Proj$ is a family of $\FK$-linear functions $\mt
= \{\mt_{V}: \End_{\cat}(V) \to \FK \}_{V \in \Proj}$ which is suitably
compatible with the tensor product and composition of morphisms (see
\cite{GKP1,GPV}).  Let $\At$ be the set of simple projective objects of
$\cat$.  A trace $\mt = \{\mt_{V}\}_{V \in \Proj}$ defines a function
$\qd:\At\to \FK^\times$ given by $\qd(V)=\mt_{V}(\Id_V)$.  If
$\qd(V)=\qd(V^*)$ for all $V\in \Proj$ then $(\At,\qd)$ is (t-)ambi and the
corresponding isotopy invariant $\Gfun'$ is determined by the trace:
$$\Gfun'(T)=\qd(V)\ang{T_V}=\mt_V(\ang{T_V}\Id_V)=\mt_V(T_V)$$
  where $T_V$ is any cutting
presentation of $T$ (see \cite{GKP1,GPV}).   

\subsection{Relative $\Gr$-spherical categories}\label{SS:GSpher} 
We now fix a group $\Gr$.  
\begin{definition} A pivotal category is
{\em $\Gr$-graded} if for each $g\in \Gr$ we have a 
non-empty full subcategory $\cat_g$ of $\cat$ such that 
 \begin{enumerate}
  \item $\cat=\bigoplus_{g\in\Gr}\cat_g$, 
  \item  if $V\in\cat_g$,  then  $V^{*}\in\cat_{g^{-1}}$,
  \item if $V\in\cat_g$, $W\in\cat$ and $V$ is isomorphic to $ W$, then
    $W\in\cat_g$
  \item  if $V\in\cat_g$, $V'\in\cat_{g'}$ then $V\otimes
    V'\in\cat_{gg'}$,
  \item  if $V\in\cat_g$, $V'\in\cat_{g'}$ and $\Hom_\cat(V,V')\neq 0$, then
    $g=g'$.  
  \end{enumerate}
\end{definition}
Let  $\X\subset\Gr$ be a subset with the
following properties:
\begin{enumerate}
\item $\X$ is symmetric: $\X^{-1}=\X$,
\item $\Gr$ can not be covered by a finite number of translated copies of
  $\X$, in other words, for any $ g_1,\ldots ,g_n\in \Gr$, we have $
  \bigcup_{i=1}^n (g_i\X) \neq\Gr $.
\end{enumerate}
Let us write $\Gr'=\Gr\setminus\X$.

\begin{definition}\label{D:G-spherical}
 Let $\cat$ be a $\Gr$-graded $\FK$-linear pivotal category. 
 Let $\A$ be the class of all simple
  objects in $\bigcup_{g\in \Gr'}\cat_g$. We say that $\cat$ is
  \emph{$(\X,\qd)$-relative $\Gr$-spherical} if
  \begin{enumerate}
  \item \label{ID:G-sph5} for each $g\in\Gr'$, $\cat_g$ is semi-simple (i.e. all
    objects of $\cat_g$ are semi-simple) with finitely many isomorphism
    classes of simple objects, \label{ID:semi-simp}
  \item \label{ID:G-sph6} there exists a map $\qd:\A\rightarrow
    \FK^{\times}$ such that $(\A,\qd)$ is a t-ambi pair,
  \item \label{ID:G-sph7} there exists a map $\bb:\A\to \FK$ such that
    $\bb(V)=\bb(V^*)$, $\bb(V)=\bb(V')$ for any isomorphic objects $V, V'\in
    \A$ and for any $g_1,g_2,g_1g_2\in \Gr\setminus\X$ and $V\in \Gr_{g_1g_2}$
    we have
  \begin{equation*}\label{eq:bb}
    \bb(V)=\sum_{V_1\in irr(\cat_{g_1}),\, V_2\in irr(\cat_{g_2})}
    \bb({V_1})\bb({V_2})\dim_\FK(\Hom_\cat(V, V_1\otimes V_2))
  \end{equation*}
  where $irr(\cat_{g_i})$ denotes a representing set of the isomorphism classes
  of simple  objects of $\catU_{g_i}$.
  \end{enumerate}
  If $\cat$ is   a category with
   the above data, for brevity we say $\cat$ is a \emph{relative
    $\Gr$-spherical category}.
\end{definition}
The map $\bb$ always exists when $\FK$ is a field of characteristic $0$ and
$\cat$ is a category whose objects are finite dimensional $\FK$-vector spaces.
In particular, in \cite{GPT2} it is shown that, for any $g\in \Gr'$, the map
\begin{equation}\label{E:defbb}
   \bb(V)=\dim_\FK(V)/\left(\sum_{V'\in irr(\cat_g)}\dim_\FK(V')^2\right)
\end{equation}
is well defined and satisfies all the properties above.

A \emph{representative set} for $\A$ is a family $\{V_i\}_{i\in I}$ of simple
objects of $\A$ numerated by elements of a set $I$ such that any element of
$\A$ is isomorphic to a unique element of $\{V_i\}_{i\in I}$.  Let
$I\rightarrow I$, $i\mapsto i^*$ be the involution determined by
$V_{i^{*}}\simeq V_i^{*}$.  For each $g\in\Gr'$ let $I_g=\{i\in
I:V_i\in\cat_g\}$ then $I=\bigcup_{g\in\Gr'}I_g$.  
For $i\in I_g$, we call $\d i=g$ the \emph{degree} of $i$.

By {\it basic data} in $\cat $ we mean a representative set $\{V_i\}_{i\in I}$
for $\A$ and a family of isomorphisms $\{w_i:V_i \to V_{i^*}^*\}_{i \in I}$
such that
\begin{equation}\label{E:FamilyIso}
  d_{V_i}(w_{i^*} \otimes \Id_{V_i})=d'_{V_{i^*}}(\Id_{V_{i^*}}\otimes
  w_i)\colon V_{i^*} \otimes V_{i}\to \unit.
\end{equation}

\begin{lemma}\label{L:basicdata}
  If no object of $\A$ is isomorphic to its dual, then $\cat$ contains a basic
  data.  In particular, basic data exists if $\X$ contains the set $\{g\in \Gr
  : g=g^{-1}\}$.
\end{lemma}
\begin{proof}
  Choose a representative for each isomorphism class of simple objects of
  $\bigcup_{g\in\Gr'}\cat_g$.  Let these representatives $\{V_i\}_{i\in I}$ be
  numerated by a set $I$.  Then $I$ is a representative set for $\A$.  The
  hypothesis of the lemma imply that $i^{*}\neq i$ for all $i\in I$. Hence,
  for any unordered pair $(i,i^{*})$, we can take an arbitrary isomorphism
  $V_{i} \rightarrow (V_{{i}^*})^*$ for $w_{i}$ and choose $w_{{i}^*}$ so that
  it satisfies Equation \eqref{E:FamilyIso}.
  
  For the second statement, let $g\in\Gr'$ and let $V$ be a simple object of
  $\cat_g$.  Since $g^{-1}\neq g$, the $\Gr$-grading of $\cat$ implies that $V$
  and $V^*$ are not isomorphic.  Thus, we can apply the above argument.
\end{proof}
If $\cat$ is a $(\X,\qd)$-relative $\Gr$-spherical category with basic data
$\{V_i,w_i\}_{i\in I}$ then the functions $\qd:\A\to \FK$ and $\bb:\A\to \FK$
can and will be considered as functions from $I$ to $\FK$ given by
$\qd(i)=\qd(V_i)$ and $\bb(i)=\bb(V_i)$, respectively.

The following theorem shows that relative $\Gr$-spherical categories are
generalizations of the usual modular categories associated to restricted
quantum groups.
\begin{theorem}\label{T:UsualCatGSph}
  Let $\tilde{U}_q(\g)$ be the restricted quantum group associated to a simple
  Lie algebra $\g$ of type $A,B,C,D$, at a primitive root of unity $q$ of even
  order (see \cite[XI.6.3]{Tu}).  Let $\cat$ be the modular category formed
  from finite dimensional representations of $\tilde{U}_q(\g)$ modulo
  negligible morphisms (see \cite[XI]{Tu}).  Then $\cat$ is a
  $(\emptyset,\qdim)$-relative $\Gr$-spherical category with basic data where
  $\qdim$ is the usual quantum dimension of $\cat$ and $\Gr=\{1\}$ is the
  trivial group of one element.  Here the map $\bb$ can be chosen to be the
  map defined in \eqref{E:defbb} or the map defined by \eqref{E:defbb} where
  the dimension $\dim_\FK$ is replaced with the usual quantum dimension
  $\qdim$.
\end{theorem}
\begin{proof}
  First, $\cat$ is a $\C$-linear pivotal category (see \cite[XI]{Tu}).  We
  assign $\cat_{1}=\cat$ and then by definition $\cat$ is $\Gr$-graded.  Also,
  $\cat$ is semi-simple with finitely many isomorphism classes of simple
  objects and so Property \ref{ID:G-sph5} of Definition \ref{D:G-spherical}
  holds.  Since $\qdim$ is non zero for each simple object of $\cat$ then
  Equation \eqref{eq:ambi} holds by definition of $\Gfun$ and so $(\A,\qdim)$
  is a t-ambi pair.  Finally, it is shown in \cite{TW} that basic data exists.
\end{proof}
\begin{remark} In Subsection \ref{TTVV} we show that a relative $\Gr$-spherical
  category gives rise to a modified Turaev-Viro invariant.  Let $\cat$ be the
  modular category of Theorem \ref{T:UsualCatGSph}.  The original Turaev-Viro
  invariant is not equal to the modified invariant corresponding to $\cat$
  where $\bb$ is defined as in Equation \eqref{E:defbb}.  However, the
  original Turaev-Viro invariant is equal to the modified invariant arising
  from $\cat$ when $\bb$ is taken to be the map given by \eqref{E:defbb} where
  the dimension $\dim_\FK$ is replaced with the usual quantum dimension
  $\qdim$.  In this case, the modified invariant does not depend on the graph
  in the 3-manifold.
\end{remark}
\begin{remark}[About semi-simplicity:] The main example of relative
  $\Gr$-spherical categories we consider in this paper are the categories of modules over 
   the unrestricted quantum groups at roots of unity, see 
  Section \ref{S:RepQG}.  These categories
  are not semi-simple which as we explain in  this remark is an essential element of these categories.  The state sum
  invariants defined in this paper only use the algebraic data of
  these category corresponding to their semi-simple part
  $\sum_{g\in\Gr \setminus \X}\catU_g$.  Roughly speaking, we represent a
  3-manifolds with a representation of its fundamental group in $\Gr$
  by a $\Gr$-colored triangulation.  Then the trick consists in using
  gauge transformations to always avoid the colors in $\X$.  

  Nevertheless, even if the non-semi-simple modules are never
  considered, the construction relies on simple modules with vanishing
  dimension. 
   Such a module $V$ can not exists in a semi-simple monoidal
  category as $\qdim(V):=d_V\circ b'_V=0$ implies that the evaluation
  map $d_V:V^*\otimes V\to\unit$ has no section because
  $\Hom(\unit,V^*\otimes V)$ is generated by $b'_V$.  Hence the 
  non-semi-simplicity is essential in all these examples.
\end{remark}
\subsection{$\hat\Psi$-systems}\label{SS:hatPsi} 
Here we recall the notion of a $\hat\Psi$-system.  These systems are the algebraic notions underlying Kashaev's invariant define in \cite{Kv94}.  For more details see \cite{GKT}.

A
\emph{$\Psi$-system} in a $\FK$-linear category $\cat$ consists of
\begin{enumerate}
\item a distinguished set of simple objects $\{ V_i\}_{ i\in I}$ such that
  $\Hom(V_i,V_j)=0$ for all $i\ne j$;
\item an involution $I\to I$, $i\mapsto i^*$;
\item two families of morphisms \(\{ b_i\colon \unit\to V_i\otimes
  V_{i^*}\}_{i\in I}\) and \( \{d_i\colon V_{i}\otimes V_{i^*}\to\unit
  \}_{i\in I}\), such that for all $i\in I$,
  \[
  (\Id_{V_i}\otimes \,d_{i^*})(b_i\otimes \Id_{V_i})=\Id_{V_i} \quad {\rm
    {and}} \quad (d_i\otimes \Id_{V_i})(\Id_{V_i} \otimes
  \,b_{i^*})=\Id_{V_i};
  \]
\item \label{I:Phi4} for any $i,j\in I$ such that $H^{ij}_k\ne0$ for some
  $k\in I$, the identity morphism $\Id_{V_i\otimes V_j}$ is in the image of
  the linear map
  \begin{equation*}
    \bigoplus_{k\in I} H^{ij}_k\otimes_\FK H_{ij}^k\to \End(V_i\otimes V_j),
    \quad x\otimes y\mapsto x\circ y.
  \end{equation*}
\end{enumerate}
Consider the $\FK$-module $H=\hat H\oplus\check H$, where
\[
\hat H=\bigoplus_{i,j,k\in I}H_{ij}^k\quad {\rm {and}} \quad \check
H=\bigoplus_{i,j,k\in I}H^{ij}_k.
\]
  Let
  \[
  \pi_{ij}^k\colon H\to H_{ij}^k,\quad \pi^{ij}_k\colon H\to H^{ij}_k,\quad
  \hat\pi\colon H\to \hat H,\quad \check\pi\colon H\to \check H
  \]
  be the obvious projections.  The $\FK$-module $H$ has a symmetric bilinear
  pairing $\ang{\,,\,}$ given by
  \begin{equation}\label{E:pairing}
    \ang{x,y}=\sum_{i,j,k\in I} (\ang{\pi_{ij}^k x\,
    \pi_k^{ij}y}+\ang{\pi_{ij}^ky\, \pi_k^{ij}x} \in\FK
\end{equation} 
for any $x,y\in H$. A \emph{transpose} of $f\in\End(H)$ is a map
$f^*\in\End(H)$ such that $\ang{fx,y}=\ang{x,f^*y}$ for all $x,y\in H$.
Note that if a transpose $f^*$ of $f$ exists, then it is unique and
$(f^*)^*=f$ (see \cite{GKT}).  An operator $f\in\End (H)$ such that $f^*=f$
is called \emph{symmetric}.  An operator $f\in\End (H)$ such that
$f(H^{ij}_k) \subset H^{ij}_k$ and $f(H_{ij}^k)\subset H_{ij}^k$ for all
$i,j,k\in I$ is called \emph{grading-preserving}.

We define linear maps \( A,B:H\to H \) by
\[
A x=\sum_{i,j,k\in I}((\Id_{V_{i^*}}\otimes \,
\pi_{ij}^kx)(b_{i^*}\otimes\Id_{V_j})+
(d_{i^*}\otimes\Id_{V_j})(\Id_{V_{i^*}}\otimes \, \pi^{ij}_kx)),
\]
\[
B x=\sum_{i,j,k\in I} ((\pi_{ij}^kx\otimes\Id_{V_{j^*}})(\Id_{V_i}\otimes \,
b_{j})+(\Id_{V_i}\otimes \, d_{j})(\pi^{ij}_kx\otimes\Id_{V_{j^*}})).
\]
In \cite{GKT}, it is shown that both $A$ and $ B$ are involutive and have
transposes.  Also, from \cite{GKT} we have that the operators $$L=A^*A ,
\,\, R=B^*B, \,\, C=(AB)^3 \in \End(H)$$ are symmetric, grading-preserving
and invertible.

\begin{definition}[\cite{GKT}]\label{D:PsiHat}A \emph{$\hat \Psi$-system} in
  $\cat$ is a $\Psi$-system in $\cat$ together with a choice of invertible,
  symmetric, grading-preserving operators $\scop{}, \srop{} \in \End(H)$
  satisfying
  \begin{subequations}\label{E:sqrtCC}
    \begin{align}\label{E:sqrtC1C2}
      (\scop{})^2&=C,\quad   A \scop{}A=B \scop{}B= \scop{-},\\
      \label{E:sqrtC1C2JJ}
      (\srop{})^2&=\Rop,\quad
      B\srop{} B=\srop{-},\quad\srop{}\scop{}=\scop{}\srop{}\\
      T\scop{}_1\scop{}_2&=T\scop{}_3\scop{}_4, \quad
      \label{E:sqrtR1L2JJ}
      T\srop{}_1\slop{}_2=T\srop{}_3\slop{}_4,\\
      \label{E:sqrtR1R2}
      T\srop{}_1\srop{}_2&=T\scop{}_3\srop{}_4, \quad
      T\slop{}_1 =T\scop{}_2\slop{}_3\slop{}_4
    \end{align}
  \end{subequations}
  where $\slop{}$ is defined to be the invertible operator $\slop{}= B A
  \srop{-} A B \in \End(H)$.
\end{definition}

\section{Modified Reshetikhin--Turaev link invariants from projective
  ambidextrous objects}\label{S:RTLink}  
  In this section we show that certain projective modules lead to ambidextrous traces and re-normalized invariants as discussed in Subsection \ref{SS:ambi}. 
Let $\cat$ be a pivotal $\FK$-linear category with duality morphisms $b_{V} :
\:\: \unit \rightarrow V\otimes V^{*} , \quad d_{V}: \:\: V^*\otimes
V\rightarrow \unit , \quad b'_V: \:\: \unit \rightarrow V^{*}\otimes V\quad
{\rm {and}} \quad d'_V: \:\: V\otimes V^*\rightarrow \unit$.  For every morphism
$f:V\to W$ in $\cat$ let $f^*:W^*\rightarrow V^*$ be the dual (or transposed)
morphism given by
$$
f^*=(d_W \otimes \Id_{V^*})(\Id_{W^*} \otimes f \otimes
\Id_{V^*})(\Id_{W^*}\otimes b_V).
$$
The axioms of a pivotal category imply that for each object $V$ of $\cat$
there is a canonical functorial isomorphism $V\to V^{**}$, see \cite{GPV}.  To
simplify notation we will use this isomorphism to identify $V^{**}$ with $V$.
Therefore, for any object $V$ of $\cat$, we have $(V\otimes
V^{*})^{*}=(V^{*})^{*}\otimes V^{*}=V\otimes V^{*}$.  A simple object $V$ is
\emph{ambidextrous} if $ f\circ b_V=f^*\circ b_V$ for all $f\in\End(V\otimes
V^*).$
Note in \cite{GPV} it is shown that when $\cat$ is ribbon this definition is
equivalent to the definition of ambidextrous given in \cite{GPT1}.

\begin{theorem}\label{T:RTLink}
  Let $\catt$ be a ribbon category and let $\At$ be the set of simple
  projective objects of $\catt$.  Suppose that there is a projective
  ambidextrous object in $\catt$.  Then there is an unique (up to a scalar of
  $\FK^\times$) map $\qd:\At\to\FK^\times$ such that $(\At,\qd)$ is an
  ambidextrous pair.  Hence the map $\Gfun':\Lt\to\FK$ given in Subsection
  \ref{SS:ambi} is a well defined isotopy invariant.
\end{theorem}
\begin{proof}
  From Theorem 3.3.2 of \cite{GKP1} we have the existence and uniqueness of a
  trace on the subcategory of projective objects $\Proj$.  As in Subsection
  \ref{SS:ambi} this trace defines a function $\qd:\At\to\FK^\times$.  The
  properties of the trace imply that $\qd(V)=\qd(V^*)$ for all $V\in \At$ and
  that $(\At,\qd)$ is an ambi pair.  The uniqueness of the trace implies the
  uniqueness of $\qd$.
\end{proof}

Let $J$ be a projective ambidextrous object in a ribbon category $\catt$.  
  Following
\cite{GPT1},  for any simple objects $V,W$ of $\catt$, define
$$
S'(V,W)=\ang{\ \epsh{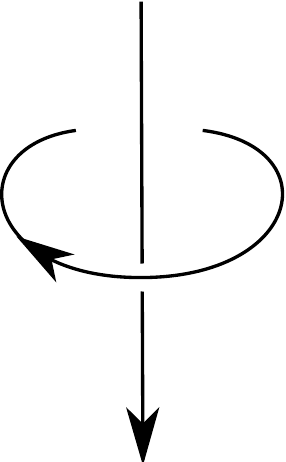}{8ex}\put(-30,-5){$V$}\put(-8,-12){$W$}\ }
$$ 
then for any $V\in \At$ with $S'(J,V)$ invertible, we have that
$$\qd(V)=\qd(J)\dfrac{S'(V,J)}{S'(J,V)}$$
because $\qd(V)S'(J,V)=\qd(J){S'(V,J)}$ is the value of $\Gfun'$ on the Hopf
link colored by $J$ and $V$.

There is a large class of examples where the category $\catt$ is not
semi-simple and the usual Reshetikhin-Turaev ribbon functor $\Gfun$ restricted
to the domain $\Lt$ of $\Gfun'$ is zero.  For example, let $\g$ be a simple
Lie algebra.  In Section \ref{S:RepQG}, for each $r$th root of unity
$\e^{2\frac{i\pi}\ro}$, we define an algebra $\UqgH$ which we call the
unrolled Drinfeld-Jimbo quantum group associated to $\g$.  Then in
Subsection~\ref{SS:weightmod}, we introduce a ribbon category $\catHr$ of
nilpotent $\UqgH$-weight modules which satisfy the hypothesis of Theorem
\ref{T:RTLink}.  In particular, we use $S'$ to give an explicit formula for
the modified quantum dimension $\qd$ in this case (see Theorem
\ref{T:catHrInv}). The corresponding invariant $\Gfun'$ restricted to framed
links colored with projective simple modules can be considered as a
generalized colored Alexander invariant which in the case of $\slt$ is the
hierarchy of invariants defined in \cite{ADO} (see Theorem 35 of \cite{GPT1}).
  
In \cite{CGP1} we will extend the link invariants of Theorem \ref{T:RTLink} to
a modified Reshetikhin--Turaev-type 3-manifold invariant.  More precisely, in
\cite{CGP1} we will give a notion of a relative $\Gr$-modular category and
show that such categories have Kirby colors which give rise to 3-manifold
invariants.

\section{Topological invariants from relative $\Gr$-spherical
  categories}\label{S:TopGSph}
In this section we will introduce a state sum \eqref{eq:TV} of 6j-symbol
associated to a relative spherical category.  This sum is a topological
invariant called the modified Turaev--Viro invariant (cf. Theorem \ref{T:inv}).
Then in Subsection \ref{SS:PsiHatRelG} we will show that a relative spherical
category induces $\hat\Psi$-structures (cf. Theorem \ref{T:PsiHatSystem}) and
that the associated Kashaev-type state sums \eqref{eq:K} are equal to the
modified Turaev--Viro invariant (cf. Theorem \ref{T:TV=K}).
\subsection{Modified $6j$-symbols}\label{S:InvOfGraphs} 
In \cite{GPT2} the authors show that certain pivotal categories give rise to
modified $6j$-symbols.  In this subsection, we show the techniques of
\cite{GPT2} can be applied to the situation of relative spherical categories.

Let $\cat$ be a $(\X,\qd)$-relative $\Gr$-spherical category with basic data
$\{V_i,w_i\}_{i\in I}$.  For any $i,j,k\in I$, recall the notation $
H^{ijk}=\Hom(\unit, V_i\otimes V_j \otimes V_k).  $ The $\FK$-modules
$H^{ijk},H^{jki},H^{kij}$ are canonically isomorphic. Indeed, let
$\sigma(i,j,k)$ be the isomorphism
$$
H^{ijk} \rightarrow H^{jki},\,\,\,
x \mapsto d_{V_i}\circ (\Id_{V_{i}^*}\otimes x \otimes \Id_{V_i})\circ
b'_{V_i}.
$$
Using the functor $\Gfun : \Graph_\cat\rightarrow \cat$, one easily proves that
$$
\sigma(k,i,j) \, \sigma(j,k,i)\, \sigma(i,j,k) =\Id_{H^{ijk}}.
$$
Identifying the modules $H^{ijk},H^{jki},H^{kij}$ along these isomorphisms we
obtain a {\it symmetrized multiplicity module} $H(i,j,k)$ depending only on
the cyclically ordered triple $(i,j,k)$.  
Remark that the $\Gr$-grading of $\cat$ implies that
\begin{equation}
  \label{eq:grad}
  \text{ for all } i,j,k\in I \text{ such that } \d i\; \d j\, \d k\neq1\in\Gr, \text{ then }
  H(i,j,k)=\{0\}. 
\end{equation}

By a labeling of a graph we mean a function assigning to every edge of the
graph an element of $I$. By a trivalent graph we mean a (finite oriented)
graph whose vertices all have valency $3$. Let $\Gamma$ be a labeled trivalent
graph in $S^2$.  Using the standard orientation of $S^2$ (induced by the
right-handed orientation of the unit ball in $ \R^3$), we cyclically order the
set $X_v$ of 3 half-edges adjacent to any given vertex $v$ of $\Gamma$.  The
labels of the edges determine a function $f_v:X_v\to I$ as follows: if a
half-edge $e$ adjacent to $v$ is oriented towards $v$, then $f_v(e)=i$ is the
label of the edge of $\Gamma$ containing $e$; if a half-edge $e$ adjacent to
$v$ is oriented away from $v$, then $f_v(e)=i^*$. Set $H_v(\Gamma)=H(f_v) $
and $H(\Gamma)=\otimes_v \, H_v(\Gamma)$ where $v$ runs over all vertices of
$\Gamma$.

Consider now a labeled trivalent graph $\Gamma\subset S^2 $ endowed
with a family of vectors $h=\{h_v \in H_v(\Gamma)\}_v$, where $v$ runs over
all vertices of $\Gamma$. We thicken $\Gamma$ into a $\cat$-colored ribbon
graph on $S^2$ as follows.  First, we insert inside each edge $e$ of $\Gamma$
a coupon with one edge outgoing from the bottom along $e$ and with one edge
outgoing from the top along $e$ in the direction opposite to the one on
$e$. If $e$ is labeled with $i\in I$, then these two new (smaller) edges are
labeled with $V_i$ and $V_{i^*}$, respectively, and the coupon is labeled with
$w_i:V_i\to V_{i^*}^*$ as in Figure \ref{F:Couponw}.

\begin{figure}[h,t]
  \centering $ \xymatrix{ \:
    \\
    *+[F]\txt{ \; $w_i$ \;} \ar[d]^{V_i} \ar[u]_{V_{i^*}}
    \\
    \: \hspace{40pt} } $
 \caption{}\label{F:Couponw}
\end{figure}

Next, we thicken each vertex $v$ of $\Gamma$ to a coupon so that the three
half-edges adjacent to $v$ yield three arrows adjacent to the top side of the
coupon and oriented towards it. If $i,j,k\in I$ are the labels of these arrows
(enumerated from the left to the right), then we color this coupon with the
image of $h_v$ under the natural isomorphism $H_v(\Gamma)\to H^{ijk}$. Denote
the resulting $\cat$-colored ribbon graph by $\Omega_{\Gamma, h}$. Then
$\mathbb G(\Gamma, h)= \Gfun (\Omega_{\Gamma, h})$ is an isotopy invariant of the
pair $(\Gamma, h)$ independent of the way in which the vertices of $\Gamma$
are thickened to coupons.

The invariant $\Gfun'$ defined in Section \ref{SS:GSpher} can be be extended to a
bigger class of $\cat$-colored ribbon graphs in $S^2$.  We say that a coupon
of a ribbon graph is straight if both its bottom and top sides are incident to
exactly one arrow.  We can remove a straight coupon and unite the incident
arrows into a (longer) edge, see Figure \ref{F:smooth}.  We call this
operation {\it straightening}. A {\it quasi-trivalent ribbon graph} is a
ribbon graph in $S^2$ such that straightening it at all straight coupons we
obtain a trivalent ribbon graph.

\begin{figure}[h,b]
  \centering \hspace{10pt} $\epsh{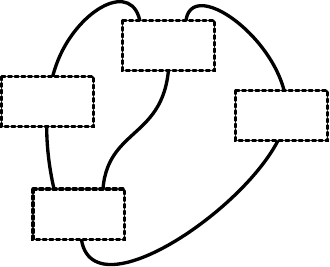}{10ex}$ \hspace{10pt}
  $\longrightarrow$ \hspace{10pt} $\epsh{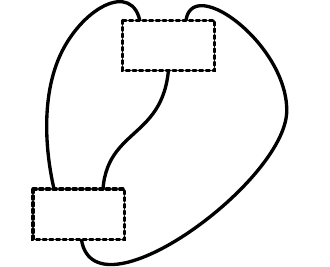}{10ex} $ \hspace{10pt}
  \caption{Straightening a quasi-trivalent  ribbon  graph}\label{F:smooth}
\end{figure}

\begin{lemma}\label{L:generalized}
  Let $\overline \tet_{I}$ be the
  class of connected quasi-trivalent ribbon graphs in $S^2$ such that the
  colors of all edges belong to the set $\{V_i \}_{i \in I}$ and all straight
  coupons are colored with isomorphisms in~$\cat$.  Then Formula \eqref{El+}
  determines a well-defined function $\Gfun':\overline \tet_{I}\rightarrow \FK$.
\end{lemma}
\begin{proof} This is proved in Lemma 2 of \cite{GPT2}.
\end{proof}

We can combine the invariant $\Gfun'$ with the thickening of trivalent graphs to
obtain invariants of trivalent graphs in $S^2$.  Suppose that $\Gamma\subset
S^2$ is a labeled connected trivalent graph.  We define
$$
\mathbb{G'}(\Gamma)\in
H(\Gamma)^\star=\Hom_\FK(H(\Gamma),\FK)
$$
as follows. Pick any family of vectors $h=\{h_v \in H_v(\Gamma)\}_v$, where
$v$ runs over all vertices of $\Gamma$.  The $\cat$-colored ribbon graph
$\Omega_{\Gamma, h}$ constructed above
belongs to the class $\overline \tet_{I}$ defined in Lemma
\ref{L:generalized}.  Set
$$
\mathbb{G}'(\Gamma)(\otimes_v h_v)=\Gfun'(\Omega_{\Gamma, h}) \in \FK.
$$
By the properties of $\Gfun'$, the vector $\mathbb{G}'(\Gamma)\in H(\Gamma)^\star$
is an isotopy invariant of~$\Gamma$.  Both $H(\Gamma)$ and
$\mathbb{G'}(\Gamma)$ are preserved under the reversion transformation
inverting the orientation of an edge of $\Gamma$ and replacing the label of
this edge, $i$, with $i^*$.  This can be easily deduced from Formula
\eqref{E:FamilyIso}.

Let $i,j,k,l,m,n$ be six elements of $I$.  Consider the labeled trivalent
graph $\Gamma=\Gamma(i,j,k,l,m,n)\subset\R^2\subset S^2$ given in Figure
\ref{F:Gamma7777}.  By definition,
$$
H(\Gamma)=  H(m,n^*, i^*) \otimes_\FK H(n,l^*, j^*)\otimes_\FK H(i,j,k^*) \otimes_\FK
H(k,l,m^*).
$$
We define the modified $6j$-symbol of the tuple $(i,j,k,l,m,n)$ to be
\begin{equation}\label{mss}
  \sjtop ijklmn=\mathbb{G}'(\Gamma)\in H(\Gamma)^\star.
\end{equation}
It follows from the definitions that the modified $6j$-symbols have the
symmetries of an oriented tetrahedron.  In particular,
$$
\sjtop ijklmn = \sjtop j{k^*}{i^*}mnl = \sjtop klm{n^*}{i}{j^*}.
$$
These equalities hold because the labeled trivalent graphs in $S^2$ defining
these $6j$-symbols are related by isotopies and reversion transformations
described above.  Also, it follows from 
Equation~\eqref{eq:grad} that
$H(i,j,k^*)=0$ if $\d i\d j\neq \d k$ and so $\sjtop ijklmn = 0$ if one of the
following equalities is not satisfied in $\Gr$:
\begin{equation}
  \label{eq:nul-6j}
  \d i\d j=\d k,\quad \d k\d l=\d m,\quad \d j\d l=\d n \quad
  \text{or}\quad \d i\d n=\d m.
\end{equation}
\begin{figure}[t]
  \centering
  \subfloat[$\Gamma(i,j,k,l,m,n)$]{\label{F:Gamma7777} \hspace{10pt}
    $\epsh{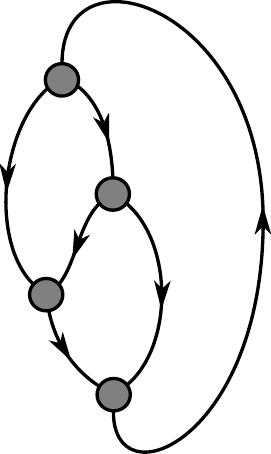}{120pt}\pic{.8pt}{\put(-83,-49){k}\put(-51,26){n}
      \put(-98,10){i}\put(-68,-3){j}\put(-46,-39){l}\put(-18,-7){m}} $
    \hspace{15ex}}
  \subfloat[$\Theta({i,j,k})$]{\label{F:Thetaijk} \hspace{10pt}
    \begin{minipage}{20ex}\ \\[5ex]
      $\epsh{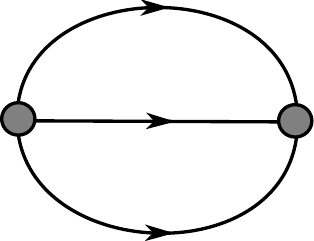}{10ex} $
      \put(-38,-19){$i$}\put(-38,7){$j$}\put(-38,30){$k$}
      \put(-74,0){$v$}\put(2,0){$u$}\\[5ex]\
    \end{minipage}}
  \caption{Elementary labeled graphes}  \label{F:el-graph}
\end{figure}
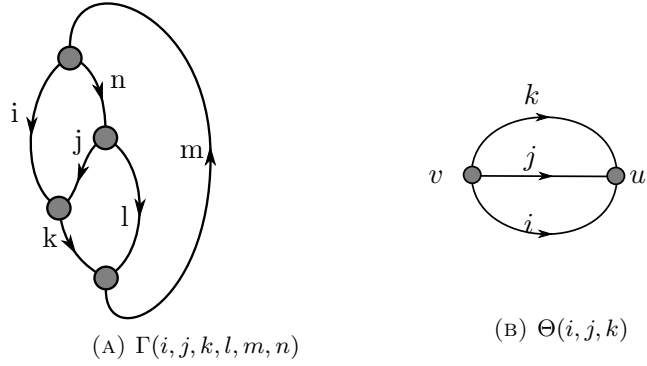
For any indices $i,j,k\in I$, we define a pairing
\begin{equation} 
  \label{E:pairing1-} (,)_{ijk}: H(i,j,k)\otimes_\FK H(k^*,j^*,i^*)\to \FK
\end{equation}
by
\begin{equation}\label{E:Pairing_ijk}
(x,y)_{ijk}=\mathbb G'(\Theta) (x\otimes y)\, ,
\end{equation}
where $x\in H(i,j,k)=H^{ijk}$ and $y\in H(k^*, j^*, i^*)=H^{k^* j^* i^*}$ and
$\Theta=\Theta_{i,j,k} $ is the theta graph with vertices $u,v$ and three
edges oriented from $v$ to $u$ and labeled with $i,j,k$, see\ Figure
\ref{F:Thetaijk}. Clearly,
$$
H_u(\Theta)=H^{ijk}=H(i,j,k)\,\, \,\, {\text {and}} \,\,\,\,
H_v(\Theta)=H^{k^*j^*i^*}=H(k^*,j^*, i^*)
$$ 
so that we can use $x$, $y$ as the colors of $u$, $v$, respectively.  It
follows from the definitions that the pairing $(,)_{ijk}$ is invariant under
cyclic permutations of $i,j,k$ and $(x,y)_{ijk}=(y,x)_{k^*j^*i^*}$ for all
$x\in H(i,j,k) $ and $y\in H(k^*, j^*, i^*)$.

Given indices $i,j,k\in I^3$ and a tensor product of several $\FK$-modules
such that among the factors there is a matched pair $H(i,j,k)$, $ H(k^*, j^*,
i^*)$, we may contract any element of this tensor product using the pairing
\eqref{E:pairing1-}. This operation is called the contraction along $H(i,j,k)$
and denoted by $*_{ijk}$.  For example, an element $x\otimes y \otimes z \in
H(i,j,k)\otimes_\FK H(k^*, j^*, i^*) \otimes_\FK H$, where $H$ is a
$\FK$-module, contracts into $(x,y)z\in H$.

The rest of this section contains properties of the modified $3j$ and
$6j$-symbols defined above.

\begin{lemma} \label{L:pairNondeg} For any elements $i,j,k$ of $I$, the
  pairing \eqref{E:pairing1-} is non-degenerate.
\end{lemma}
\begin{proof}  
  Let us assume that $H(i,j,k)\neq 0$ otherwise the lemma is trivial.  Then by
  Equation \eqref{eq:grad}, $V_i \otimes V_j$ is an object of $\cat_{\d k^{-1}}$
  which is semi-simple because $\d i\,\d j=\d k^{-1}\in\Gr'$.  Then the lemma
  follows from Lemma 3 of \cite{GPT2}.  To show that the hypothesis of the
  lemma in \cite{GPT2} are satisfied it is enough to show that the pair
  $(i,j)$ has a quality called good.  A pair of indices $(l,m)\in I^2$ is
  called good if $V_l\otimes V_m$ splits as a direct sum of some $V_n$'s
  (possibly with multiplicities) such that $n\in I$ and $\qd(V_n)$ is
  invertible in $\FK$.  Hence, by definition of a relative spherical category
   the pair $(i,j)$ is good and the proof follows
  Lemma 3 of \cite{GPT2}.
  \end{proof}

\begin{theorem}[The Biedenharn-Elliott identity]
  \label{T:BEId} Let $j_0,j_1,{\ldots},j_8 $ be elements of $ I $.  Assume
  that $\d j_2\d j_3 \notin \X$, then
  $$
  \sum_{j\in I_{\d j_2\d j_3}} \qd(V_j) \, *_{j_2j_3j^*} *_{jj_4j_7^*}*_{j_1j
    j_6^*} \left(\sjtop {j_1}{j_2}{j_5}{j_3}{j_6}{j}\otimes \sjtop
    {j_1}{j}{j_6}{j_4}{j_0}{j_7}\otimes \sjtop {j_2}{j_3}{j}{j_4}{j_7}{j_8}
  \right)
  $$
  \begin{equation}\label{E:BEId}
    = *_{j_5j_8j_0^*}\left(\sjtop {j_5}{j_3}{j_6}{j_4}{j_0}{j_8}\otimes
      \sjtop {j_1}{j_2}{j_5}{j_8}{j_0}{j_7}\right)\, .
  \end{equation}
  Here both sides lie in the tensor product of the six $\FK$-modules
  $$
  H(j_6,j_3^*,j_5^*), H(j_5,j_2^*,j_1^*), H(j_0,j_4^*,j_6^*),
  H(j_1,j_7,j_0^*), H(j_2,j_8,j_7^*), H(j_3,j_4,j_8^*).
  $$
\end{theorem}
\begin{proof}
  The proof follows from Theorem 7 of \cite{GPT2}.  We now explain why the
  hypothesis of this theorem are satisfied.  First, the definitions above
  imply that the set $J$ of \cite{GPT2} is equal to $I_{\d j_2\d j_3}$.
  Second, as in the proof of Lemma \ref{L:pairNondeg} certain pairs of
  elements of $I$ must be good.  As above, if $l,m,n\in I$ with $H(l,m,n)\neq
  0$ then the pair $(l,m)$ is good.  Moreover, if the pair $(l,m)$ is not good
  then for any $n\in I$ the space $H(l,m,n)$ is zero.  Thus, from
  \eqref{eq:nul-6j} it follows that both sides of Equation~\eqref{E:BEId} are
  zero unless all required pairs are good and so the hypothesis of \cite{GPT2}
  are satisfied.
\end{proof}

\begin{theorem}[The orthonormality relation]
  \label{T:orth}
  Let $i,j,k,l,m,p$ be elements of $I$.  Assume
  $\d j\d l\notin \X$, then
  $$
  \qd(V_k)\sum_{n\in I_{\d j\d l}}\qd(V_n) *_{inm^*}*_{jln^*}\left(\sjtop ij{p}lmn
    \otimes \sjtop k{j^*}inml \right)
  $$
  $$
  =\delta_{k,p}\Id(i,j,k^*)\Id(k,l,m^*),
  $$
  where $\delta_{k,p}$ is the Kronecker symbol and $\Id(a,b,c)$ is the
  canonical element of $H(a,b,c)\otimes_\FK H(c^*,b^*,a^*)$ determined by the
  duality pairing.
\end{theorem}
\begin{proof}
  The proof follows from Theorem 8 of \cite{GPT2}.  Again here, as in the
  proof of Theorem \ref{T:BEId} the definition of a relative $\Gr$-spherical category
  implies the hypothesis of \cite{GPT2}.
\end{proof}

The following lemma is an algebraic analog of the {\bubble} move. 
\begin{lemma}\label{L:algBubble}
  Let $g_1,g_2,g_3,g_4,g_5,g_6\in\Gr\setminus\X$ with $g_3=g_1g_2$,
  $g_6=g_2g_4$ and $g_5=g_1g_6$.  If $ i\in I_{g_1}$, $j \in I_{g_2}$, $k\in
  I_{g_3}$, then
  \begin{align}\label{E:algLbubble}
    \qd(k) \hspace{-2ex}\sum_{\tiny l \in I_{g_4}, m\in I_{g_5} ,
          n\in I_{g_6}}\hspace{-4ex} \qd(n) \bb(l)&\bb(m)
    *_{klm^*}*_{inm^*}*_{jln^*} \left( \sjtop
      ijklmn \otimes \sjtop k{j^*}inml \right) \notag \\
    &=\bb(k)\Id(i,j,k^*)
  \end{align}
\end{lemma}
\begin{proof}
The proof follows from a direct manipulation of the orthonormality relation using the properties of the map $\bb$.  In particular, the proof of   \cite[Lemma 25]{GPT2} can easily be adapted to the case when $\Gr$ is not commutative.  
\end{proof}

\subsection{Topological preliminary}
Let $M$ be an oriented compact smooth 3-manifold.  Here a graph  $\Y$ is a finite
1-dimensional CW-complex disjoint union a finite number of circles without vertices.  We
denote by $\Y_0$ the set of vertices of $\Y$ and $\partial \Y$ the set of
univalent vertices of $\Y$.  By {\em a graph in $M$}, we mean an embedding of
a graph in $M$ such that  $\Y\cap\partial M\subset \partial \Y$ and $\Y$ meets all
connected components of $M$.  We say that $\Y$ is {\em rooted} in $\Y\cap\partial M$.

By a triangulation $\T$ of $M$, we mean a smooth $\Delta$-complex structure on
$M$ (as in \cite{Ha}).  Loosely speaking, a $\Delta$-complex structure is a
quotient space of a collection of disjoint simplices obtained by identifying
certain of their faces.  In particular, the interior of the simplices of $\T$
are embedded in $M$ but their faces are not necessarily distincts and two
different simplices might meet on several faces.  We say that $\T$ is
\emph{quasi-regular} if any simplex of $\T$ is embedded in $M$.  This is
equivalent to requiring that the two endpoints of any edge of $\T$ are
distinct vertices of $\T$.

Let $\Y$ be a graph in $M$.  Let $\T$ be a quasi-regular triangulation of $M$
and $\Tb$ be a triangulation of $\partial M$.  Let $\YY$ be a set of
unoriented edges of $\T$ such that the union of these edges is the graph $\Y$
in $M$.  We say the pair $(\T,\YY)$ is an \emph{$H$-triangulation} of $(M,\Y)$
\emph{relative} to $\Tb$ if the following four conditions hold: (1) the
vertices of $\Y$ are also vertices of $\T$, (2) all the vertices of $\T$ are
contained in $\Y$ (i.e. all the vertices are incident to an edge of $\YY$),
(3) any edge of $\T$ with both endpoints in
$\Y_0$ is contained in $\partial M$ and (4) $\T$ restricts to $\Tb$.
\begin{theorem}\label{L:Toplemma-}
  Any triplet (a compact orientable $3$-manifold $M$, a quasi-regular
  triangulation $\Tb$ of $\partial M$, a graph $\Y$ in $M$ rooted in $\Tb$)
  admits an $H$-triangulation of $(M,\Y)$ relative to $\Tb$.
\end{theorem}

The proof of Theorem \ref{L:Toplemma-} will be given in Subsection
\ref{SS:proof}.

\subsection{Fundamental groupoid and space of representations in $\Gr$}\label{SS:FGRep}

A groupoid is a small category in which all morphisms are isomorphisms.  Let
us consider two examples of groupoids.  First, every group $\Gr$ can be seen
as a groupoid with only one object whose set of endomorphisms is the group
$\Gr$.  Second, let $W$ be a non empty locally path connected topological
space.  The \emph{fundamental groupoid} $\pi_1(W)$ of $W$ is the small
category whose objects are points of $W$ and whose morphisms are homotopy
classes of paths in $W$, where composition is concatenation of paths.  Let $Z$
be a subset of $W$.  If $Z$ is nonempty, let $\pi_1(W,Z)$ be the full
subcategory of $\pi_1(W)$ whose objects are points of $Z$.  When $Z$ is a
single point $w$ of $W$ then $\pi_1(W,Z)=\pi_1(W,\{w\})$ is the usual
fundamental group.  If $W_1,\ldots, W_n$ are the connected components of $W$
and $Z$ meets all of these component, then
$\pi_1(W,Z)=\bigsqcup_i\pi_1(W_i,Z\cap W_i)$.

When $W$ is connected, let
$$
\MG W=\Hom_{Groups}(\pi_1(W,\{w\}),\Gr)/\Gr
$$ 
where $w$ is any point of $W$ and $g\in\Gr$ acts on
$\rho:\pi_1(W,\{w\})\to\Gr$ by conjugation: $g.\rho:\gamma\mapsto
g\rho(\gamma)g^{-1}$.  The set $\MG W$ is called the \emph{space of
  representations} of $W$.  Remark that for two different choices of the base
point $w$, the corresponding spaces of representations of $W$ are canonically
isomorphic.  We set
$$ \MG {W,Z}=\left\{
\begin{array}[c]{l}
  \MG{W}\text{ if }Z=\emptyset\\
  \Hom_{\text{groupoid}}(\pi_1(W,Z),\Gr)\text{ if }Z\neq\emptyset
\end{array}\right. $$%
We extend this definition for non connected spaces by
$$\MG {\bigsqcup_i W_i,Z}=\prod_i\MG {W_i,Z\cap W_i}.$$
The assignment $(W,Z)\mapsto\MG{W,Z}$ extends to a contravariant functor.  In
particular, if $W'\subset W$ and $Z'\subset Z$ by functoriality we have a map
$\MG{W,Z}\to\MG{W',Z'}$ called the \emph{restriction map}.  Let $E$ be a set
of paths in $\pi_1(W,Z)$, we say that $\rr\in\MG{W,Z}$ is {\em $E$-admissible}
if $\rr(E)\cap\X=~\emptyset$.

Let $W$ be a triangulated manifold.  
A \emph{$\Gr$-coloring} of $W$ is a map $\wp$ from the set of oriented edges
of $W$ to $\Gr$ such that
\begin{enumerate}
\item $\wp(-e)=\wp(e)^{-1}$ for any oriented edge $e$ of $W$, where $-e$ is $e$
 with opposite orientation and
\item if $e_1, e_2, e_3$ are ordered edges of a face of $W$ endowed with
  orientation induced by the order, then $\wp(e_1)\, \wp(e_2)\, \wp(e_3)=1$.
\end{enumerate}
If $W_0$ is the set of vertices of $W$ then $\col\in\MG{W,W_0}$ is equivalent
to a $\Gr$-coloring $\col$ of $W$.
A $\Gr$-coloring of $W$, $\col\in\MG{W,W_0}$ is {\em admissible} if it is
$W_1$-admissible where $W_1$ is the set of oriented edges of $W$.

Let $M$ be an oriented compact smooth 3-manifold and $\Tb$ a quasi-regular
triangulation of its boundary.  Let $\Y$ be a graph in $M$ rooted in $\Tb$.
Let $\Y_0$ be the set of vertices of $\Y$ and $\Tb_1$ be the set of edges of
$\Tb$.  An \emph{admissible representation} of $(M,\Tb,\Y)$ in $\Gr$ is a
$\Tb_1$-admissible element of $\MG{M,\Y_0}$.  If $(\T,\YY)$ is a
$H$-triangulation of $(M,\Y)$ relative to $\Tb$ then a $\Gr$-coloring 
$\col\in\MG{M,\T_0}$ of $\T$ restricts to an element of $\MG{M,\Y_0}$.

In particular, this is true for a closed 3-manifold: If $L$ is a link in
a closed 3-manifold $M$, and if $(\T,\LL)$ is a $H$-triangulation of $(M,L)$, then
any $\Gr$-coloring of $\T$ restricts to a (always admissible) representation
of the fundamental group of $M$ up to conjugation, i.e. an element of $\MG M$.
This restriction map is the image by the contravariant functor $\MG{*}$ of the
inclusion $(M,\emptyset)\subset(M,\T_0)$.

\begin{theorem}\label{L:admColoringsExist}
  Let $(\T,\YY)$ be a $H$-triangulation of $(M,\Y)$ relative to $\Tb$.  Then
  for any admissible representation $\rr$ of $(M,\Tb,\Y)$, there exists an
  admissible $\Gr$-coloring of $\T$ which restricts to $\rr$.
\end{theorem}
The proof of the Theorem \ref{L:admColoringsExist} is given in Subsection
\ref{SS:proof}.

\subsection{Modified Turaev-Viro-invariants}\label{TTVV} 
We start from the algebraic data described in the Subsection \ref{S:InvOfGraphs}  and
produce a topological invariant of a quadruple $(M,\Tb,\Y,\rr)$, where
$M$ is a compact oriented 3-manifold, $\Tb$ is a quasi-regular
triangulation of its boundary, $\Y\subset M$ is a graph in $M$ rooted in
$\Tb$, and $\rr$ is an admissible representation of $(M,\Tb,\Y)$.

A {\it state} of an admissible $\Gr$-coloring $\col$ of a triangulated
manifold $W$ is a map $\st$ assigning to every oriented edge $e$ of $W$ an
element $\st (e)$ of $I_{\col(e)}$ such that $\st(-e)=\st (e)^*$ for all $e$.
The set of all states of $\col$ is denoted $\states(\col)$. The identities
$\qd( {{\st(e)}})=\qd( {{\st(-e)}})$ and $\bb( {{\st(e)}})=\bb( {{\st(-e)}})$
allow us to use the notation $\qd(\st(e))$ and $\bb(\st(e))$ for non-oriented
edges.

The representation $\rr$ restricts to an admissible $\Gr$-coloring
$\col^{\partial}$ of the triangulated surface $(\partial M,\Tb)$.  Given a
state $\st$ of the $\Gr$-coloring $\col^{\partial}$ of $\Tb$, the trivalent
graph $\Gamma_\st$ dual to the 1-skeleton of $\Tb$ becomes a
labeled graph in an oriented surface.  We adopt the convention that an
oriented edge $e$ of $\Gamma_\st$ has the same label than the oriented edge of
$\Tb$ crossing $e$ from the left to the right.  We set
\begin{equation}\label{E:DualMultSpace}H(\partial
M,\Tb,\col^{\partial})=\bigoplus_{\st\in\states(\col^{\partial})}H(\Gamma_\st)
\end{equation} 
where  $H(\Gamma_\st)$ is defined in Section~\ref{S:InvOfGraphs}.
If $\partial M=\emptyset$ then $\Tb=\Gamma=\emptyset$ and $H(\partial
M,\Tb,\col^{\partial})=\FK$.

Let $(\T,\YY)$ be an $H$-triangulation of $(M,\Y)$ relative to $\Tb$ and let
$\col$ be an admissible $\Gr$-coloring of $\T$ that restrict to $\rr$.
We now define a certain partition function (state sum) as follows.  For each
tetrahedron $t$ of $\T$, we choose its vertices $v_1$, $v_2$, $v_3$, $v_4$ so
that the (ordered) triple of oriented edges
$(\vect{v_1v_2},\vect{v_1v_3},\vect{v_1v_4})$ is positively oriented with
respect to the orientation of $M$. Here by $\vect{v_1v_2}$ we mean the edge
oriented from $v_1$ to $v_2$, etc. For each $\st\in \states(\col)$, set
$$
|t|_\st=\sjtop ijklmn \text{ where }\left\{
\begin{array}{lll}
  i=\st(\vect{v_2v_1}),& j=\st(\vect{v_3v_2}),&
  k=\st(\vect{v_3v_1}),\\
  l=\st(\vect{v_4v_3}),&
  m=\st(\vect{v_4v_1}),& n=\st(\vect{v_4v_2}).
\end{array}\right.
$$
This $6j$-symbol belongs to the tensor product of $4$ multiplicity modules
associated to the faces of $t$ and does not depend on the choice of the
numeration of the vertices of $t$ compatible with the orientation of $M$.
Note that any face of $\T\setminus\Tb$ belongs to exactly two tetrahedra of
$\T$, and the associated multiplicity modules are dual to each other. The
tensor product of the $6j$-symbols $|t|_\st$ associated to all tetrahedra $t$
of $\T$ can be contracted using this duality.  We denote by $\cntr$ the tensor
product of all these contractions.  Let $\T_1'$ be the set of unoriented edges
of $\T$ that are not in $\Tb$ and let $\T_3$ the set of tetrahedra of
$\T$. Set
\begin{equation}
  \label{eq:TV}
TV(\T,\YY,\col)=\sum_{\st\in \states ( \col)}\,\, \prod_{e\in\T_1'{\setminus
    \YY}} \qd(\st(e))\, \prod_{e\in \YY} \bb(\st(e))\, \, \cntr \left
  (\bigotimes_{t\in\T_3}|t|_{\st}\right ) \in H(\Tb,\col^{\partial}).
\end{equation}

\begin{theorem}\label{T:inv} 
  $TV(\T,\YY,\col)$ depends only on the isotopy class of $\Y$ in $M$, on the
  triangulation $\Tb$ of $\partial M$ and on the admissible 
  representation
  $\rr$ of $(M,\Tb,\Y)$.  We set
  $$TV(M,\Y,\rho)=TV(\T,\YY,\col)\in H(\partial M,\Tb,\col^{\partial}).$$
  Moreover, this invariant extends to a Relative Quantum Field Theory given in
  \eqref{E:FunQ}.
\end{theorem}

A proof of this theorem will be given in Section \ref{SS:proof}.

\subsection{Kashaev type 3-manifold invariants} \label{SS:KashInv}
In \cite{GKT}, Geer, Kashaev and Turaev introduce the notion of a $\hat
\Psi$-system in a tensor category (see Definition \ref{D:PsiHat} above).  It is
shown that such a system, plus some algebraic data, leads to the existence of
a Kashaev-type invariant of a link in a 3-manifold.  In this subsection we
briefly recall the construction of this invariant given in \cite{GKT}.

Throughout this subsection, $M$ is a closed connected orientable $3$-manifold
and $L$ is a non-empty link inside $M$.  Here we regard $L$ as a graph in $M$
with no vertices.  Let $(\T,\LL)$ be an $H$-triangulation of $(M,L)$.  Let $t$
be an tetrahedron of $\T$.  An \emph{integral charge} on $t$ is a
$\frac12\Z$-valued map $c$ defined on the edges of $t$ such that $c(e)=c(e')$
for opposite edges $e$ and $e'$, and for any edges $e_1, e_2, e_3$ of $t$
forming the boundary of a face of $t$ we have $c(e_1)+c(e_2)+c(e_3)=1/2$.

Let $E(\T)$ be the set of edges of $\T$ and let $\widehat{E}(\T)$ be the set
consisting of the edges of all the tetrahedra $\{t_i\}$ of $\T$.  Any edge $e$
of $\T$ gives rise to $n$ elements of $\widehat{E}(\T)$ where $n$ is the
number of tetrahedra of $\T$ adjacent to $e$.  Let $\epsilon_\T:
\widehat{E}(\T) \rightarrow E(\T)$ be the natural projection.  An
\emph{integral charge} on an $H$-triangulation $(\T,\LL)$ of $(M,L)$ is a map
$c: \widehat{E}(\T) \rightarrow \frac12\Z$ such that the restriction of $c$ to
any tetrahedron $t$ of $\T$ is an integral charge on $t$, and for each edge
$e$ of $\T$ we have $\sum_{e'\in \epsilon_\T^{-1}(e)}c(e')=c_e$, where $c_e=0$,
if $e$ belongs to $\LL$ and $c_e=1$, otherwise.  Each charge $c$ on $(\T,\LL)$
determines a cohomology class $[c]\in H^1(M; \Z/2\Z)$ as follows.  Let $s$ be
a simple closed curve in $M$ which lies in general position with respect to
$\T$ and such that $s$ never leaves a tetrahedron $t$ of $\T$ through the same
2-face by which it entered.  Thus, each time $s$ passes through $t$, it
determines a unique edge $e$ belonging to both the entering and departing
faces.  The sum of the residues $2 \,c \vert_t(e) (mod \, 2) \in \Z/2\Z$ over
all passages of $s$ through tetrahedra of $\T$ depends only on the homology
class of $s$ and is the value of $[c]$ on $s$.  For any non-empty link
$L\subset M$ and any $\zeta\in H^1(M; \Z/2\Z)$, each $H$-triangulation of
$(M,L)$ has a charge representing $\zeta$.  The theory of integral charges is
based on the work of Neumann \cite{Neu90,Neu98}.

We now describe the algebraic data needed to define the Kashaev-type
invariant.  Let $\cat$ be a $\FK$-linear category where $\FK$ is a field.
Recall the vector spaces $H^{ij}_{k}, H^{i}_{jk}, H,$ etc., from
Section~\ref{S:SphPsi}.  Fix a $\widehat{\Psi}$-system in $\cat$ with
distinguished simple objects $\{V_i\}_{i\in I}$.  Fix a family $\{I_g\}_{g\in
  \Gr}$ of finite subsets of the set $I$ numerated by elements of a group
$\Gr$ and satisfying the following conditions:
\begin{enumerate}
\item \label{I:comp1} for any $g\in \Gr$, if $i\in I_g$, then $i^*\in
  I_{g^{-1}}$;
\item \label{I:comp2} for any $i_1\in I_{g_1}, i_2\in I_{g_2}$, $k\in
  I\setminus I_{g_1g_2}$ with $g_1,g_2\in \Gr$, we have $H^{i_1 i_2}_k=0$;
\item \label{I:comp2+} if $i_1\in I_{g_1}, i_2\in I_{g_2}$ with $g_1,g_2\in
  \Gr$ then either $I_{g_1g_2}= \emptyset$ or there is a $k\in I_{g_1g_2}$
  such that $H^{i_1 i_2}_k\neq 0$;
\item \label{I:comp3} for any finite family $ \{g_r \in \Gr\}_r$, there is
  $g\in \Gr$ such that $I_{gg_r}\neq\emptyset$ for all $r$;
\item \label{I:comp4} we are given a map $\bb:I\to \FK$ such that
  $\bb(i)=\bb(i^*)$ for all $i\in I$, and for any $g_1,g_2 \in \Gr$, $k\in
  I_{g_1g_2}$ such that $I_{g_1} \neq \emptyset$ and $I_{g_2} \neq \emptyset$,
  $\sum_{i_1 \in I_{g_1},\, i_2\in I_{g_2}}
  \bb({i_1})\bb({i_2})\dim_\FK(H^{i_1i_2}_{k}) = \bb(k).$
\end{enumerate}
We define two linear forms \( T, \wb T, \colon H^{\otimes4}\to \FK \)  by
the following diagrammatic formulae: for any $ u,v,x,y\in H$,
\[
T(u\otimes v\otimes x\otimes y)=\sum_{i,\ldots,n\in I} \ang{
  \epsh{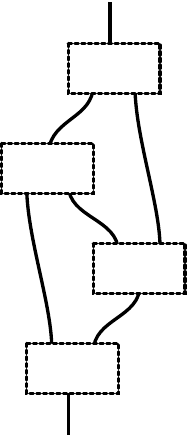}{18ex}\put(-17,29){$u$}\put(-29,10){$v$}\put(-12,-10){$x$}
  \put(-25,-28){$y$} \put(-13,40){\ms{m}}
  \put(-28,21){\ms{k}}\put(-33,-5){\ms{i}}\put(-16,4){\ms{j}}
  \put(-6,10){\ms{l}}\put(-12,-20){\ms{n}}\put(-21,-38){\ms{m}}
} ,\qquad 
\wb T(u\otimes v\otimes x\otimes y)=\sum_{i,\ldots,n\in I} \ang{
    \epsh{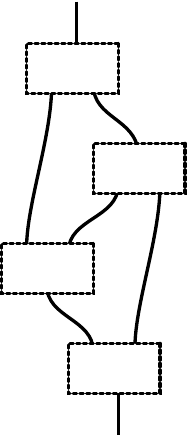}{18ex}\put(-24,29){$u$}\put(-12,10){$v$}\put(-30,-10){$x$}
  \put(-16,-28){$y$}\put(-19,40){\ms{m}}
  \put(-31,16){\ms{i}}\put(-6,-10){\ms{l}}\put(-16,0){\ms{j}}
  \put(-28,-20){\ms{k}}\put(-12,21){\ms{n}}\put(-21,-38){\ms{m}}
}.
\]
For any  $a , b\in\frac12\Z$,  let $T(a,b), \wb T(a,b)\colon
H^{\otimes4}\to \FK$ be linear maps 
\[
T(a,b)=T\sqop{4ab}{}_1\Rop^{b}_1\Rop^{-a}_2\Lop^{-a}_3\Rop^{-b}_3,\qquad
\wb T(a,b)=\wb
T\sqop{-4ab}{}_1\Lop^{-a}_2\Rop^{-b}_2\Rop^{-a}_3\Rop^{b}_4.
\]
Let $\sqop{}{8}\in \End(H)$ be the operator $\sqop{}{8}=\srop{} A \srop{-} A
\slop{-}\scop{-}$.

Let $(\T,\LL)$ be an $H$-triangulation of $(M,L)$, with integral charge $c$.
Let $\col$ be a $\Gr$-coloring of $\T$ such that $\states(\col)\neq
\emptyset$.  As remarked in the previous subsection, the $\Gr$-coloring $\col$
induces the conjugacy class of a representation $[\col]\in \MG M$.  Fix a
total ordering of the vertices of $\T$.  We recall the partition function
defined, in \cite{GKT}, from the above data.

For any tetrahedron $t$ of $\T$, let $v_1$, $v_2$, $v_3$, $v_4$ be its
vertices in increasing order (induced from the total ordering).  We say that
$t$ is {\it right oriented} if the tangent vectors $v_1v_2, v_1 v_3, v_1 v_4$
form a positive basis in the tangent space of~$M$; otherwise $t$ is {\it left
  oriented}.  For each state $\st$ of $\wp$, set
\begin{equation}\label{E:AssignmentOfLabels}
 i=\st(\vect{v_1v_2}),
j=\st(\vect{v_2v_3}),
  k=\st(\vect{v_1v_3}),
  l=\st(\vect{v_3v_4}),
  m=\st(\vect{v_1v_4}), n=\st(\vect{v_2v_4})
\end{equation}
where $\vect{v_iv_j}$ is the oriented edge of $t$ going from $v_i$ to $v_j$.
Then for each right oriented tetrahedron $t$ of $\T$ and state $\st$ the
restriction of $ T(c({v_1v_2}),c({v_2v_3})) $ to the tensor product
$H^m_{kl}\otimes_\FK H^k_{ij}\otimes_\FK H^{jl}_n \otimes_\FK H^{in}_m \subset
H^{\otimes 4}$ gives a vector in the $\FK$-vector space
\begin{equation}\label{E:ambientvectorspace} \Hom_\FK (H^m_{kl}\otimes_\FK 
  H^k_{ij}\otimes_\FK H^{jl}_n \otimes_\FK
  H^{in}_m,\FK)= H_m^{kl}\otimes_\FK H_k^{ij}\otimes_\FK H_{jl}^n \otimes_\FK
  H_{in}^m . 
\end{equation}
This vector is denoted by $|t|_\st^c$ and depends on $\st$, $c$ and the ordering
of the vertices of $\T$.  Similarly, if $t$ a left oriented tetrahedron then
one uses $\wb T(c({v_1v_2}),c({v_2v_3}))$ to assigns $t$ a vector $|t|_\st^c\in
H_m^{in}\otimes_\FK H_n^{jl}\otimes_\FK H_{ij}^k \otimes_\FK H_{kl}^m$.  The
multiplicity modules of $|t|_\st^c$ are associated to the faces of $t$.

Note that any face of $\T$ belongs to exactly two tetrahedra of $\T$, and the
associated multiplicity modules are dual to each other as in Lemma 1 of
\cite{GKT}. The tensor product of the $6j$-symbols $|t|_\st^c$ associated to
all tetrahedra $t$ of $\T$ can be contracted using this duality.  We denote by
$\cntr$ the tensor product of all these contractions.  Set
\begin{equation}
  \label{eq:K}
  \Kas(\T,\LL,\wp,c)=\sum_{\st \in \states(\wp) }\, 
  \left (\prod_{e\in \LL} \bb(\st(e))\right ) \, \cntr \left
    (\bigotimes_{t}|t|_{\st}^c\right ) \in \FK,
\end{equation}
where $t$ runs over all  tetrahedra of $\T$.

\begin{theorem}[\cite{GKT}]\label{T:MainTopInv}
  Suppose that there exists a scalar $\wsqop\in \FK$ such that $\sqop{}{8}$ is
  equal to the operator
  $$\wsqop\Id_{\hat H}\oplus \wsqop^{-1}\Id_{\check H}\in \End(H).$$
  Then, up to multiplication by integer powers of $\wsqop$, the state sum
  $\Kas(\T,\LL,\wp,c)$ depends only on $[\wp] \in \MG M$, the isotopy class of
  $L$ in $M$, and the cohomology class $[c]=\zeta\in H^1(M;\Z/2\Z)$.  We set
  $\Kas(M,L,[\wp],\zeta)=\Kas(\T,\LL,\wp,c)$.
\end{theorem}

\subsection{$\hat\Psi$-system in relative $\Gr$-spherical
  categories} \label{SS:PsiHatRelG} 

In this subsection we prove that a relative $\Gr$-spherical category with
basic data has a natural $\hat\Psi$-system for any choice of a square root of
the modified dimension $\qd$.  Moreover, for this situation, we show that
corresponding Kashaev-type and modified TV invariants are equal.

Let $\cat$ be a $(\X,\qd)$-relative $\Gr$-spherical category with basic data
$\{V_i,w_i\}_{i\in I}$ where the ground ring $\FK$ is a field.  For every pair
$(i,i^*)\in I\times I$ define the morphisms $d_i,d_{i^{*}},b_i, b_{i^{*}}$
by\begin{align*} d_i&= d_{V_{i^*}}(w_{i} \otimes \Id_{V_{i^*}}), & d_{i^*}&
  =d'_{V_{i^*}}(\Id_{V_{i^*}}\otimes w_i), & b_i&=(\Id_{V_i}\otimes
  w_{i^{*}}^{-1})\circ b_{V_i}, & b_{i^*}&=(w_{i^*}^{-1}\otimes
  \Id_{V_i})b'_{V_i}.
\end{align*}
It is easy to see that Equation \eqref{E:FamilyIso} implies $(\Id_{V_i}\otimes
w_{i^*}^{-1})b_{V_i}=(w_{i}^{-1}\otimes \Id_{V_{i^*}})b'_{V_{i^*}}$.
Therefore, if $i=i^*$ then $b_i=b_{i^*}$ and $d_i=d_{i^*}$.

Recall the definition of a $\Psi$-system given in Subsection \ref{SS:hatPsi}.

\begin{prop}\label{P:PsiSysCat}
 The collection $\{V_i,b_i,d_i\}_{i\in I}$ is a $\Psi$-system in $\cat$.
\end{prop}
\begin{proof}
  By definition $\{V_i,b_i,d_i\}_{i\in I}$ satisfies the first three
  properties in the definition of a $\Psi$-system.  Let $i,j\in I$ and suppose
  $H^{ij}_k\neq0$ for some $k$.  Then Equation \eqref{eq:grad}
  implies $\d i \d j=\d k\in \Gr'$ and so $V_i\otimes V_j$ is an object of
  $\cat_{\d k}$.  Thus, since $\cat_{\d k}$ is semi-simple with a finite
  number of isomorphism classes of simple objects given by $\{V_l:l\in I_{\d
    k}\}$ we have that $\Id_{V_i\otimes V_j}$ has the desired property.
\end{proof}

Recall the operators $A,B, L= A^{*}A, R= B^{*}B:H\to H$ defined in Subsection
\ref{SS:hatPsi}.  Since $\cat$ is pivotal we have $ABA=BAB$ and thus
$C=(AB)^3=\Id_H$.  Let $S=ABA$ then $S^{*}=S=S^{-1}$ and $SAS=B$.

Choose a function $\qee:I\to\FK$ satisfying $\qee(i^{*})=\qee(i)$ and
$\qd(i)=\qee(i)^2$ for all $i\in I$. Define the operator
$$
\qe:H\to H\text{ given by }\qe=\bigoplus_{ijk}\qee(k)\Id_{H_{ij}^{k}}\oplus
\qee(k)\Id_{H^{ij}_{k}}.
$$ 
Let $\qe^{A}=A\qe A\text{ and }\qe^{B}=B\qe B$.

\begin{lemma}\label{L:scio1}
  $\qe,\qe^{A},\qe^{B}$ are symmetric commuting invertible operators
  satisfying
  \begin{align}\label{eq:dS}
    S\qe&=\qe S&S\qe^{A}&=\qe^{B} S\\
    \label{eq:dA}
    A\qe^{B}&=\qe^{B}A&B\qe^{A}&=\qe^{A}B\\
    \label{eq:dB}
    A\qe&=\qe^{A}A&B\qe&=\qe^{B}B
    \\
    \label{eq:RLambi}
    L&=\left(\qe/{\qe^{A}}\right)^{2}&R&=\pp{\qe/{\qe^{B}}}^{2}
    \\
    \label{eq:Td1}
  T\qe_1&=T\qe_4&T\qe^{B}_1&=T\qe_2&T\qe^{A}_1&=T\qe^{A}_3\\
    \label{eq:Td2}
  T\qe^{A}_2&=T\qe^{B}_3&T\qe^{B}_2&=T\qe^{B}_4&T\qe_3&=T\qe^{A}_4
  \end{align}
\end{lemma}
\begin{proof}
  First, $\qe$, $\qe^{A}=\bigoplus_{ijk}\qee(i)\Id_{H_{ij}^{k}\oplus
    H^{ij}_{k}}$ and $\qe^{B}=\bigoplus_{ijk}\qee(j)\Id_{H_{ij}^{k}\oplus
    H^{ij}_{k}}$ are obviously commuting.
   Notice that on $H_{ij}^{k}\oplus H^{j^{*}i^{*}}_{k^{*}}$, $\qe$ acts by the
  scalar $\qee(k)=\qee(k^{*})$.  As $S$ stabilizes these summands of $H$, $S$
  and $\qe$ commute which is the first equality of \eqref{eq:dS}.
  Similarly, the other equality in \eqref{eq:dS} and Equations
  \eqref{eq:dA},\eqref{eq:dB} are easily deduced from the relation satisfied
  by $A,B,S$ and the definitions of $\qe,\qe^{A},\qe^{B}$.

  We will now show that the equalities in \eqref{eq:RLambi} are a consequence
  of the fact that $(\A,\qd)$ is a t-ambi pair.  Consider the isomomorphisms
  \begin{align}\label{E:Hiso}
    H^{ij}_k \rightarrow
    H^{ijk^*}&= H(i,j,k^*),\,\, y \mapsto (y \otimes w_{k^*}^{-1})\,
    b_{V_k},\notag \\ 
    H_{ij}^k \rightarrow H^{kj^*i^*}&=H(k,j^*,i^*), \,\, x \mapsto (x\otimes
    w_{j^*}^{-1} \otimes w_{i^*}^{-1})(\Id_{V_i}\otimes b_{V_j} \otimes
    \Id_{V^*_i})b_{V_i}.
  \end{align}
  Using these isomorphisms the pairings $(,)_{ijk}: H(i,j,k)\otimes_\FK
  H(k^*,j^*,i^*)\to \FK$ and $\ang{\, ,\,}:H\otimes H\to \FK$ given in
  Equations \eqref{E:pairing} and \eqref{E:Pairing_ijk}, respectively are
  related by:
   \begin{equation}
    \label{eq:nonAmbiPairing}
    (*,*)_{i,j,k^{*}}=\qd(V_k)\ang{*,*}_{|H^{ij}_k\otimes H_{ij}^{k}}
    =\ang{\qe(*),\qe(*)}_{|H^{ij}_k\otimes H_{ij}^{k}}.
  \end{equation}
  Remark that the isomorphisms of \eqref{E:Hiso} commute with $A$ and $B$ and
  so we have
   $$\qd(V_k)\ang{x,Ly}_{| H_{kj^{*}}^{i}\otimes H^{kj^*}_i}=
  \qd(V_k)\ang{Ax,Ay}_{|H^{ij}_k\otimes H_{ij}^{k}}=(x,y)_{i,j,k^{*}}
  =\qd(V_j)\ang{x,y}_{|H_{i^*k}^{j}\otimes H^{i^*k}_j},$$
  $$\qd(V_k)\ang{x,Ry}_{| H_{kj^{*}}^{i}\otimes H^{kj^*}_i}=
  \qd(V_k)\ang{Bx,By}_{|H^{ij}_k\otimes H_{ij}^{k}}=(x,y)_{i,j,k^{*}}
  =\qd(V_i)\ang{x,y}_{| H_{kj^{*}}^{i}\otimes H^{kj^*}_i}. $$ 
  Thus, $Ly=\qe^{2}{(\qe^{A})}^{-2}(y)$, $R=\qe^{2}{(\qe^{B})}^{-2}(y)$ and
  this implies Equation~\eqref{eq:RLambi}.

  The form $T$ on $H^{\otimes 4}$ is $0$ on most of its summands.  
  In particular, it can only be non-zero if the morphisms are composable which
  mean that the summand has the form
  \begin{equation}\label{E:HHHHsum}
  H^m_{kl}\otimes_\FK H^k_{ij}\otimes_\FK H_n^{jl}\otimes_\FK H_m^{in}
  \end{equation}
  for some $(i,j,k,l,m,n)\in I^{6}$.  Equations \eqref{eq:Td1},\eqref{eq:Td2}
  follow from the fact that on these summand the operators
  $\qe,\qe^{A},\qe^{B}$ act by easily identifiable scalars.  For example, on
  the summand of \eqref{E:HHHHsum} the operators $\qe^{B}_1$ and $\qe_2$
  act as $\qe^{B}_{|H^m_{kl}}=\qee(k)\Id$ and 
  $\qe_{|H^k_{ij}}=\qee(k)\Id$, respectively.
\end{proof}
\begin{theorem}\label{T:PsiHatSystem}
  The $\Psi$-system of Proposition \ref{P:PsiSysCat} extends to a $\hat
  \Psi$-system where $\Re=\frac\qe{\qe^{B}}$ and $\Ce=\Id$.  In this case,
  $\Le=\frac\qe{\qe^{A}}$.
\end{theorem}
\begin{proof}
  The Equations \eqref{E:sqrtC1C2},\eqref{E:sqrtC1C2JJ} are trivial except for
  $B\Re B=\Rop^{-\frac12}$ which is a consequence of the fact that conjugation
  by $B$ exchange $\qe$ and $\qe^{B}$ (see Equations \eqref{eq:dB}). Equations
  \eqref{E:sqrtR1L2JJ},\eqref{E:sqrtR1R2} follow from Equations
  \eqref{eq:Td1},\eqref{eq:Td2}.  For example, $T\Re_1\Le_2
  =T\frac{\qe_1\qe_2}{\qe^{B}_1\qe^{A}_2} =T\frac{\qe_1}{\qe^{A}_2}$ and
  $T\Re_3\Le_4= T\frac{\qe_3\qe_4}{\qe^{B}_3\qe^{A}_4} =
  T\frac{\qe_4}{\qe^{B}_3}$ are equal.  Finally,
  $$\Le=BA\Rop^{-\frac12}AB=BA\frac{\qe^{B}}{\qe^{}}AB=B\frac{\qe^{B}}{\qe^{A}}AAB
  =\frac{\qe}{\qe^{A}}BAAB=\frac\qe{\qe^{A}}.$$
\end{proof}

\begin{theorem}\label{T:TV=K}
  Let $L$ be a non-empty link in a closed connected orientable $3$-manifold
  $M$ (as above we regard $L$ as a graph with no vertices).  Let $\rr\in\MG M$
  and $\zeta\in H^1(M;\Z/2\Z)$.  Let $TV(M,L,\rr)$ be the invariant arising
  from the relative $\Gr$-spherical structure of $\cat$ and let
  $\Kas(M,L,\rr,\zeta)$ be the invariant arising from the $\hat\Psi$-system
  described above.  Then $TV(M,L,\rr)=\Kas(M,L,\rr,\zeta)$.  In particular,
  $\Kas(M,L,\rr,\zeta)$ is a well defined complex number which is independent
  of
  the choice of the cohomology class $\zeta$ and the square root $\qee$ of
  $\qd$.
\end{theorem}
\begin{proof}  
  First, using Lemma \ref{L:scio1} and Theorem \ref{T:PsiHatSystem} it is easy
  to compute the operator $\sqop{}{8}:H\to H$:
  $$\sqop{}{8}=\Re
  A\Rop^{-\frac12}A\Lop^{-\frac12}C^{-\frac12}=\frac\qe{\qe^{B}}A\frac{\qe^{B}}\qe
  A\frac{\qe^{A}}\qe=
  \frac\qe{\qe^{B}}\frac{\qe^{B}}{\qe^{A}}\frac{\qe^{A}}\qe=\Id_H.$$ 
  Remark that as $q$ is responsible for the ambiguity of $\wsqop$ in $\Kas$,
  we see that in our context, this ambiguity disappears.
  
  From duality and the equality
  $\qd(m)\Id_{H_{kl}^{m}}=\delta^{2}|_{H_{kl}^{m}}$, we have symmetrised
  $T$-forms $T_s, \wb {T}_s:H^{\otimes4}\to\FK$ given by
  $T_s=T\delta_1^2=T\delta_4^2$ and $\wb {T}_s=\wb T\delta_1^2=\wb
  T\delta_4^2$.  Let $h=H^m_{kl}\otimes_\FK H^k_{ij}\otimes_\FK H^{jl}_n
  \otimes_\FK H^{in}_m$ and $h'=H^m_{in}\otimes H^n_{jl}\otimes H^{ij}_k
  \otimes H^{kl}_m$.  After identifying the modules $H^{ij}_{k}$ and
  $H_{j^*i^*}^{k^*}$ and with the $H(i,j,k^*)$ (using the isomorphisms in
  Equation \eqref{E:Hiso}) we have that
  \begin{equation}\label{E:TandMod6j}
  \wb{T}_s|_{h'}=\qd(m)\wb{T}|_{h'}=\sjtop ijklmn, \;\;\; 
  T_s|_h=\qd(m)T|_{h}
  =\sjtop n{j^{*}}lkmi.
 \end{equation}
  Note the 6j-symbol in the last equality correspond to the same spherical
  graph with opposite orientation. 
  
  Let $(\T,\LL)$ be an $H$-triangulation of $(M,L)$.  Let $\col$ be
  an admissible $\Gr$-coloring of
  $\T$ such that $[\col]=\rho \in \MG M$.  Let $c$ be a integral charge on
  $(\T,\LL)$ such that $[c]=\zeta\in H^1(M;\Z/2\Z)$.  For each state $\st$ of
  $\col$ let us write $TV_\st$ (resp. $\Kas_\st$) for the summand of
  $TV(M,L,\rr)$ (resp. $\Kas(M,L,\rr,\zeta)$) corresponding to $\st$.  In
  other words, we have
  $$TV(M,L,\rr)= \sum_{\st\in \states(\col)} TV_\st, \;\; \;\;
  \Kas(M,L,\rr,\zeta)=\sum_{\st\in \states(\col)} \Kas_\st.$$ 
  
  Let $\st$ be a state of $\col$.  We will show that $TV_\st=\Kas_\st$.
  First, notice that the factors of $TV_\st$ and $\Kas_\st$ corresponding to
  $\bb$ are identical.  Let $t$ be a tetrahedron of $\T$ with ordered vertices
  $v_1,v_2,v_3,v_4$.  We will now compute the vector $|t|_\st^c$ which is used
  in the construction of $\Kas$.  Let $a=c({v_1v_2})$, $b=c({v_2v_3})$ and
  $i,j,k,l,m,n$ be the indices given in Equation \eqref{E:AssignmentOfLabels}.
  By definition we have
  $$T(a,b)=Tq_1^{4ab}R_1^{b}R_2^{-a}L_3^{-a}R_3^{-b}=
  T\left(\frac{\qe_{1}}{\qe^{B}_{1}}\right)^{2b}
  \left(\frac{\qe_{2}}{\qe^{B}_{2}}\right)^{-2a}
  \left(\frac{\qe_{3}}{\qe^{A}_{3}}\right)^{-2a}
  \left(\frac{\qe_{3}}{\qe^{B}_{3}}\right)^{-2b}.$$ 
  Using Lemma \ref{L:scio1} and the fact that $\qe$, $\qe^A$ and $\qe^B$ all
  act as scalars on a multiplicity space we have
  $$
  T(a,b)|_{h}=(S^t_1S^t_2)T_s|_h
  $$ 
  where $S^t_1=\left(\qd(i)\qd(l)\right)^a \left(\qd(j)\qd(m)\right)^{b}
  (\qd(k)\qd(n))^{c(v_1v_3)}$ and $S^t_2=(\qee(m)\qee(k)\qee(n)\qee(m))^{-1}$.
  Similarly, $ \wb T(a,b)|_{h'}= (S_1S_2)\wb{T}_s|_{h'}$.  If $t$ is right
  oriented then $|t|_\st^c$ is the vector in the module on right hand side of
  Equation \eqref{E:ambientvectorspace} corresponding to $T(a,b)|_{h}$.
  Similarly, if $t$ is left oriented then $|t|_\st^c$ is the vector in
  $H_m^{in}\otimes_\FK H_n^{jl}\otimes_\FK H_{ij}^k \otimes_\FK H_{kl}^m$
  corresponding to $ \wb T(a,b)|_{h'}$.  Thus, from Equation
  \eqref{E:TandMod6j} we have $|t|_\st^c=S^t_1S^t_2|t|_\st$ where $|t|_\st$ is
  the tensor used in the construction of $TV$.

  Now for each tetrahedron $t$ of $\T$ we have a scalar $S^t_1$ as described
  above.  Each of these scalars is the product $\prod_e\qd(\st(e))^{c(e)}$
  where $e$ runs over the 6 edges of $t$.  Let $e$
  be an edge of $\T$.  The portion of $\prod_{t\in \T_3} S^t_1$ corresponding
  to the edge $e$ is
  $$
  \prod_{e'\in \epsilon_\T^{-1}(e)}\qd(\st(e'))^{c(e')}=\qd(\st(e))^{\sum_{e'
    } c(e')}=\qd(\st(e))^{c_e}
  $$ 
  where $c_e=0$, if $e$ belongs to $\LL$ and $c_e=1$, otherwise (for notation
  see Subsection \ref{SS:KashInv}).  Thus, $\prod_{t\in \T_3}
  S^t_1=\prod_{e\in \T\setminus \LL}\qd(\st(e))$.

  For each $t\in \T_3$, the contribution of $S^t_2$ can be considered as the
  operator $\bigotimes_f\qe^{-1}_f$ where $f$ runs over the set of the 4 faces
  of $t$ and $\qe_f$ is $\qe$ applied to the multiplicity module $H$
  associated to the face $f$.  This factor is exactly what is needed to relate
  the pairs $(,)_{ijk}$ and $\ang{\, ,\,}$, see Equation
  \eqref{eq:nonAmbiPairing}.  Combining all the statements above we see that
  $TV_\st=\Kas_\st$ and so the invariants are equal.
\end{proof}

\section{Relative Homotopy Quantum Field Theory}\label{S:HQFT}
In this section we will extend the invariant $TV$ of 3-manifolds with
triangulated boundary to an invariant of 3-manifolds bounding surfaces with
 dots.  We will also prove the main topological theorems of the paper.

 \subsection{A $2+1$-cobordism category $\Cob$ of triangulated surfaces}
The first step in extending $TV$ is  the introduction of
an ``oscillating path'' in a triangulated surface.

{\em An oscillating path} in a triangulated surface $\SS$ is a 
map $\os:\{\text{edges of }\SS\}\to\{-1,0,1\}$
that satisfies the following conditions: Let $P=\os^{-1}(\{-1,1\})$ be the support of
$\os$ and $|P|$ be the 1-dimensional simplicial complex which is the closure
of $P$.  Then we require that any connected component of $|P|$ is 
homeomorphic to a segment and contains an even number of edges.  We also 
require 
that $\os$ takes alternatively values $+1$ and $-1$ on the sequence of
edges of any of these paths.  

We can represent an oscillating path on a
triangulation by coloring some of its edges with $+$ and $-$ (see Figure
\ref{fig:ospath}).
\begin{figure}[b]
  \centering
  $\epsh{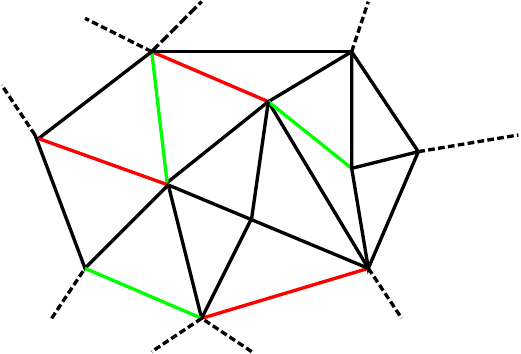}{18ex}$
  \caption{A triangulated surface with an oscillating path:
    $\os$(edge)=$\os$(color) with $\os(red)=+1$, $\os(green)=-1$,
    $\os(black)=0$}
  \label{fig:ospath}
\end{figure}
We define the number of strands of $(\SS,\os)$ to be the cardinal of the
set of vertices $\SS_0$ of $\SS$ minus the number of edges in the
support of $\os$.  This is the number of connected components of $\SS_0\cup
|P|$.

Let $\SS_1$ and $\SS_2$ be two triangulated oriented surfaces with admissible
$\Gr$-colorings and oscillating paths.  Let $\SS_2^{*}$ be the surface $\SS_2$
with opposite orientation.  An \emph{enriched cobordism} from $\SS_1$ to
$\SS_2$ is a quadruplet $(M,f,\Y,\rr)$ where:
\begin{enumerate}
\item $M$ is an oriented compact $3$-manifold,
\item $f:\SS_1\sqcup \SS_2^{*}\to\partial M$  is a diffeomorphism,  
\item $\Y\subset M$ is a graph in $M$ rooted in the triangulation $\Tb$ of
  $\partial M$ induced from $\SS_1\sqcup \SS_2^{*}$ under $f$ such that 
if  $e$ is an edge of $\SS_1$ (resp. $\SS_2$) whose value under the 
 oscillating path is $+1$ (resp. $-1$) then there exits a connected 
 component $\Y_e$ of $Y$
  and a \emph{trivializing disk} $D^{2}$ embedded in $M$ with
  $\Int(D^{2})\cap\Y=\emptyset$ and $\partial D^{2}= e\cup\Y_e$ (see Figure
  \ref{fig:cob}),
  \begin{figure}[b]
    \centering
    $\epsh{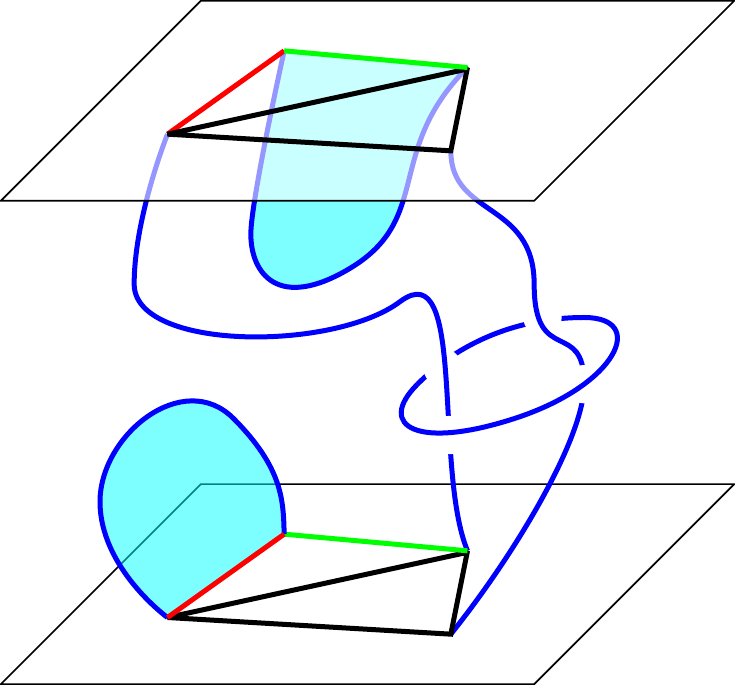}{24ex}\put(-10,-45){$\SS_1$}\put(-10,40){$\SS_2$}$
    \put(-10,0){$\Y\subset M$} 
    \caption{A cobordism between triangulated surfaces with oscillating paths,
      here the trivializing disks are the shaded regions.} 
    \label{fig:cob}
  \end{figure}
\item $\rr$ is an admissible representation of $(M,\Tb,\Y)$ that restricts to
  both of the admissible $\Gr$-colorings of $\SS_2$ and $\SS_1$.  
\end{enumerate}
The triplet $(M,\Y,\rr)$ is regarded up to diffeomorphism that is trivial
on the boundary (where we use $f$ to identify $\partial M$ and $\SS_1\sqcup
\SS_2^{*}$).

One can glue an enriched cobordism $(M,f,\Y,\rr)$ from $\SS_1$ to $\SS_2$ with
an enriched cobordism $(M',f',\Y',\rr')$ from $\SS_2$ to $\SS_3$.  The result
is an enriched cobordism $(M'',f'',\Y'',\rr'')$ from $\SS_1$ to $\SS_3$ where
$\Y''$ is the union of the graphs $\Y$ and $\Y'$.  The set of vertices of
$\Y''$ is $\Y''_0=(\Y_0\cup\Y'_0)\setminus Z$ where $\Y_0$ and  $\Y'_0$ are the
vertices of $\Y$ and  $\Y'$, respectively and $Z=\Y_0\cap\Y'_0$ is the subset of these vertices
that are on $\Sigma_2\hookrightarrow M''$.
Remark that Van Kampen's theorem for groupoids (see \cite{Br}) imply that
$\pi(M'',Y_0\cup Y'_0)$ is the pushout of the maps
$\pi(\SS_2,Z)\to\pi(M,\Y_0)$ and $\pi(\SS_2,Z)\to\pi(M',\Y'_0)$. Thus, there
exists an unique representation $\wt\rr$ of $(M'',Y_0\cup Y'_0)$ which
restrict to both $\rr$ and $\rr'$.  The representation $\rr''$ is the
restriction of $\wt\rr$ to $(M'',Y_0'')$.  Loosely speaking, a path $\gamma$
in $\pi(M'',Y_0'')$ can be cut to a finite composition of smaller paths in
$\pi(M,\Y_0)$ and $\pi(M,\Y'_0)$, and $\rr''(\gamma)$ is the product of the
images by $\rr$ and $\rr'$ of these smaller paths.

\begin{theorem}
  There is a monoidal category $\Cob$ whose objects are quadruplets
  $(\SS,\T,\os,\col)$ formed by an oriented surface $\SS$ with a quasi-regular
  triangulation $\T$, an oscillating path $\os$ and an admissible
  $\Gr$-coloring $\col$ of $\T$, and whose morphisms are enriched cobordisms.
  The tensor product is given by disjoint union and the composition by gluing
  along the boundaries.
\end{theorem}
\begin{proof} It is well known that cobordisms of surfaces form a monoidal
  category.  It is easy to adapt the proof of this result to enriched
  cobordisms.  The only non-trivial point is to prove that objects have an
  identity morphism.  Let us construct the identity morphism $M=(\SS\times
  [0,1],\Y,\rr)$ of $(\SS,\T,\os,\col)$ as follows.  The underlying 3-manifold
  is the cylinder $\SS\times [0,1]$.  The graph $\Y$ is a tangle as in the
  schematic picture of Figure \ref{fig:id}.
  \begin{figure}[b]
    \centering $\epsh{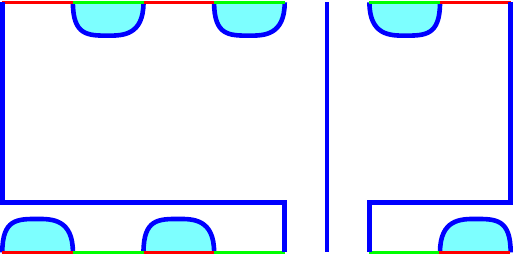}{18ex}\put(4,-45){$\SS\times\{0\}$}
    \put(4,45){$\SS\times\{1\}$}\put(4,0){$\Y\subset \SS\times[0,1]$}$
    \caption{The identity cobordism for a surface $\SS$ with 3 strands.}
    \label{fig:id}
  \end{figure}
  $\Y$ is the unique tangle (up to isotopy) which is vertical except in a
  neighborhood of $\SS\times\{0\}$ and $\SS\times\{1\}$ where it is contained
  in $|P|\times[0,1]$ (recall $|P|$ is the union of edges $e$ with
  $\os(e)=\pm1$).  The representation $\rr$ is the pull back of $\col$ by the
  projection $p:\SS\times [0,1]\to \SS$.  One can check that gluing an
  enriched cobordism $M'$ with $M$ along $\SS$ does not change the isomorphism
  class of $M'$.
\end{proof}

Let $\wt\SS=(\SS,\T,\os,\col)$ be an object of $\Cob$. 
The dual object $\wt\SS^{*}$ of $\wt\SS$ in $ \Cob$ 
is defined  as follows.  
  Let $\SS^{*}$ be the surface $\SS$
with opposite orientation.  Let $\T^{*}=\T$ be the triangulation of $\SS^{*}$
and let $\col^{*}=\col$ be the $\Gr$-coloring of $\T^{*}$.
Finally, let $\os^*=-\os$ be the oscillating path of $\T^*$. We define
$\wt\SS^{*}=(\SS^{*},\T^{*},\os^*,\col^{*})$.

Let $\Gamma$ be the trivalent graph dual to $\SS$.  As in Equation
\eqref{E:DualMultSpace} we can define
$H(\wt\SS)=\bigoplus_{\st\in\states(\col)} H(\Gamma_\st)$.  Changing the
orientation of $\SS$ changes all the indexes of $\Gamma_\st$ to their dual and
so the pairing \eqref{E:pairing1-} induces a non degenerate pairing
$\angg{\cdot,\cdot}_0:H(\wt\SS^{*})\otimes H(\wt\SS)\to\FK$.

Consider the operator 
$\qd_{\wt\SS}$ 
on $H(\wt\SS)$ given by
$$\qd_{\wt\SS}=\bigoplus_{\st\in\states(\col)}
\left(\prod_{e\in\T_1}\qd(\st(e))\right)\Id_{H(\Gamma_\st)}.$$  
Define
the non degenerate pairing
\begin{equation}
  \label{eq:pairing2}
  \angg{\cdot,\cdot}:H(\wt\SS^{*})\otimes H(\wt\SS)\to\FK \;\; \text{by} \;\;
  x\otimes y \mapsto \angg{x,\qd_{\wt\SS}(y)}_0=\angg{\qd_{\wt\SS^{*}}(x),y}_0. 
\end{equation}

Let $\wt\SS'=(\SS',\T',\os',\col')$ be another object of $\Cob$.  Since
$H(\wt\SS_i^*)\cong H(\wt\SS_i)^*$, the pairing in \eqref{eq:pairing2} induces an
isomorphism
\begin{equation}\label{E:IsoHSS}
H(\wt\SS^{*}\sqcup\wt\SS') =H(\wt\SS^{*})\otimes
H(\wt\SS') \cong\Hom_\FK(H(\wt\SS),H(\wt\SS')).
\end{equation}  
If $(M,f,\Y,h)$ is an enriched cobordism from $\wt\SS$ to $\wt\SS'$
then using \eqref{E:IsoHSS} we can identify 
 $TV(M,\Tb,\Y,h)$ with an element of
$\Hom_\FK(H(\wt\SS),H(\wt\SS'))$ where $\Tb\simeq{\T}^{*}\sqcup\T'$ is the
triangulation of $\partial M$ induced by $f$.

In \cite[Section 2.3]{TV} a mapping $\cat\to \cat'$ of categories is called a
\emph{semi-functor} if it satisfies the first condition of the definition of a
functor: namely, it sends the composition of morphisms to the composition of
their images.  
If $\cat'$ is abelian, a semi-functor leads to an honest functor:
to each object of $\cat$ assign the coimage in $\cat'$ of the image by the
semi-functor of the identity morphism of this object, then naturally extend
this assignment to morphisms.  With this in mind, we define the following
functor $\QQ$.  The definition of the state sum $TV$ implies that the mapping
$\Cob\to \Vect$ defined by $\wt\SS\mapsto H(\wt\SS)$ and
$(M,\Tb,\Y,h)\mapsto TV(M,\Tb,\Y,h)\in \Hom_\FK(H(\wt\SS),H(\wt\SS'))$
is a monoidal semi-functor.  Thus, we have a functor $\QQ:\Cob\to \Vect$ where
$\QQ(\wt\SS)$ is given by $H(\wt\SS)$ modulo the kernel of the image of
$\Id_{\wt\SS}$ under this semi-functor.

\subsection{A $2+1$-cobordism category $\Cm$ of marked surfaces}
\label{SS:notriang}
In this subsection, we remove the requirement that surfaces are triangulated.

A \emph{marked surface} is a closed oriented surface $\SS$ with a finite set
of points $\ma\subset \SS$ such that $\ma$ has a nonempty intersection with
any connected component of $\SS$.  The elements of $\ma$ are called
\emph{marks}.
We consider marked surfaces equipped with a representation in $\Gr$ that is an
element of $\MG{\SS,\ma}$.

If $\SS_1$ and $\SS_2$ are two marked surfaces with 
representations in $\Gr$,
a \emph{marked cobordism} from $\SS_1$ to $\SS_2$ is a
quadruplet $(M,f,\Y,\rr)$ where:
\begin{enumerate}
\item $M$ is an oriented compact $3$-manifold,
\item $f:\SS_1\sqcup \SS_2^{*}\to\partial M$  is a diffeomorphism,
\item $\Y\subset M$ is a graph rooted in the subset $\ma\subset\partial M$
  formed by the image under $f$ of the marks of $\SS_1$ and $\SS_2$,
\item $\rr$ is an element of $\MG{M,\Y_0}$ that restricts to the
  representations of $\SS_1$ and  $\SS_2$ in $\Gr$.
\end{enumerate}
The triplet $(M,\Y,\rr)$ is regarded up to diffeomorphism which is trivial
with respect to the boundary (where we use $f$ to identify $\SS_1\sqcup
\SS_2^{*}$ with $\partial M$).

Let $\Cm$ be the monoidal category whose objects are marked surfaces with
representation in $\Gr$ and morphisms are marked cobordisms.  There is an
obvious forgetful semi-functor $\Cob\to\Cm$ sending the set of vertices of a
triangulation to the set of marks.  We modify this semi-functor to construct
an equivalence of categories.  The main point here is that the number of marks
must be reduced.

If $(\SS,\T)$ is a triangulated surface with an oscillating path $\os$, we can
associate a set of marks $\ma$ containing one point in each connected
component of the support $|P|$ of $\os$ as follows: $\ma$ is the set of
vertices $v$ of $\T$ such that for any edge $e$ containing $v$, $\os(e)\geq0$.
Hence the number of marks of $(\SS,\T,\os)$ is equal to the number of strands
of $(\SS,\os)$. 
Restricting the representation in $\Gr$ 
of an object $\wt\SS$ of
$\Cob$, we get an object $\Eq(\wt\SS)$ of $\Cm$.  In order to extends this map
to a functor $\Eq:\Cob\to\Cm$ we need to define some special cobordisms.

Let $\wt\SS=(\SS,\T,\os,\col)$ be an object of $\Cob$ and let
$\wb\SS=\Eq(\wt\SS)$.  Consider the enriched cobordism $\SS\times[0,1]$ that
represent the identity of $\wt\SS$ (see Figure \ref{fig:id}).  Using the
forgetful semi-functor $\Cob\to\Cm$ we can view this enriched cobordism as
marked cobordism of $(\SS,\T_0,\col)$.  Cutting this marked cobordism along
$\SS\times\{\frac12\}$, we can see it as the composition of two morphisms
$\Cy_{\wt\SS}^{\wb\SS}:(\SS,\T_0,\col)\to\wb\SS $ and
$\Cy_{\wb\SS}^{\wt\SS}:\wb\SS\to(\SS,\T_0,\col)$ in $\Cm$
such that $\Cy_{\wt\SS}^{\wb\SS}\circ \Cy_{\wb\SS}^{\wt\SS}=\Id_{\wb\SS}$ and
$\Cy_{\wb\SS}^{\wt\SS}\circ \Cy_{\wt\SS}^{\wb\SS}$ is the image of
$\Id_{\wt\Sigma}$ under the forgetful semi-functor $\Cob\to\Cm$.

We define $\Eq$ on morphisms as follows.  Let $\wt\SS'$ be another object of
$\Cob$ and let $\wb\SS'=\Eq(\wt\SS')$.  For $M\in \Hom_\Cob(\wt\SS,\wt\SS')$
define $\Eq(M)=\Cy_{\wt\SS'}^{\wb\SS'}\circ M\circ \Cy_{\wb\SS}^{\wt\SS} \in
\Hom_\Cm(\wb\SS,\wb\SS').$
\begin{lemma}\label{L:equiv}
  The functor $\Eq:\Cob\to\Cm$ is an equivalence of monoidal
  categories.  
\end{lemma}
\begin{proof}
  We need to show that $\Eq$ is a full, faithful and essentially surjective.
  Consider the map $ \Hom_\Cob(\wt\SS,\wt\SS')\to
  \Hom_\Cm(\Eq(\wt\SS),\Eq(\wt\SS'))$ induced by $\Eq$.  This map is
  invertible with inverse 
  $$\Hom_\Cm(\Eq(\wt\SS),\Eq(\wt\SS'))\to
  \Hom_\Cob(\wt\SS,\wt\SS') \text{ given by } \wb M\mapsto
  \Cy_{\wb\SS'}^{\wt\SS'}\circ\wb M\circ \Cy_{\wt\SS}^{\wb\SS}.$$ Therefore, $\Eq$
  is full and faithful.  The functor $\Eq$ is essentially surjective because
  any marked surface has a quasi-regular triangulation $\T$ with oscillating
  path containing those vertices of $\T$ which are not in $\ma$.
\end{proof}

In fact, if two objects of $\Cob$ have the same underlying marked 
surface $\wb \SS$, then they are canonically isomorphic (using the cylinders
of Lemma \ref{L:equiv}) and thus we can define $\wb Q(\wb\SS)$ as the result
of the identification of these spaces along these isomorphisms.  This
construction naturally extends to a functor
\begin{equation}\label{E:FunQ}
\wb Q:\Cm\to \Vect.
\end{equation}
We call this functor a Relative Quantum Field Theory.  

\subsection{Proofs of Theorems \ref{L:Toplemma-},  \ref{L:admColoringsExist}
  and \ref{T:inv}   }\label{SS:proof} 
Throughout this section, we keep notation of Theorem \ref{T:inv}. We begin by
explaining that any two $H$-triangulations of $(M,\T',\Y)$ can be related by
elementary moves adding or removing vertices, edges, etc.  We call an
elementary move positive if it adds edges and negative if it removes edges.

The elementary moves (see \cite{GPT2}) are the so-called {\bubble} moves
(Figure~\ref{F:bubble}), the {\it {\Pachner} $2\leftrightarrow 3$ moves}
(Figure~\ref{F:pachner}) and the {\it {\lune} moves} (Figure~\ref{F:lune}).
\begin{figure}[b]
  \centering \subfloat[{\bubble} move]{\label{F:bubble} \hspace{10pt} $
    \epsh{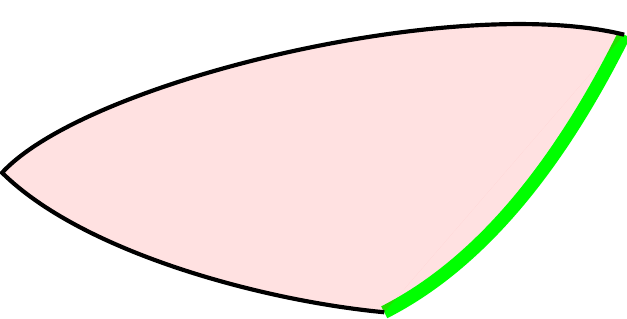}{30pt}\longrightarrow \epsh{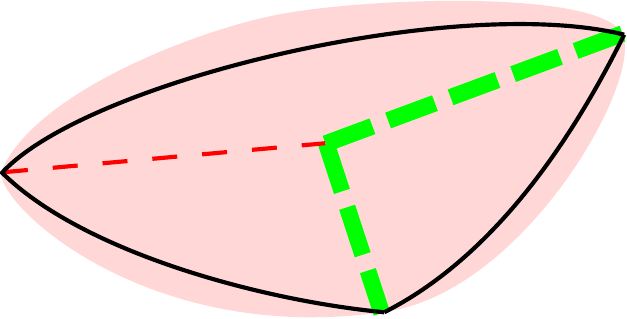}{30pt}
    \put(-23,0){\tiny v} $ \hspace{10pt} } \subfloat[{\Pachner}
  move]{\label{F:pachner} \hspace{10pt} $
    \epsh{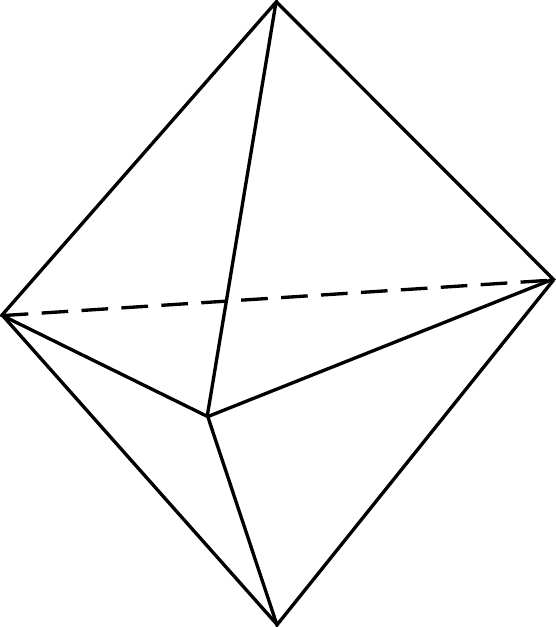}{50pt}\longleftrightarrow\epsh{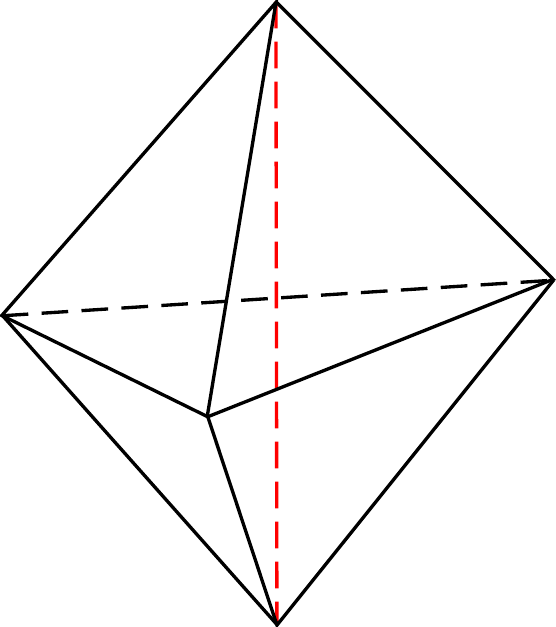}{50pt} $ \hspace{10pt}
  }\\
  \subfloat[{\lune} move]{\label{F:lune} \hspace{10pt}
    $\epsh{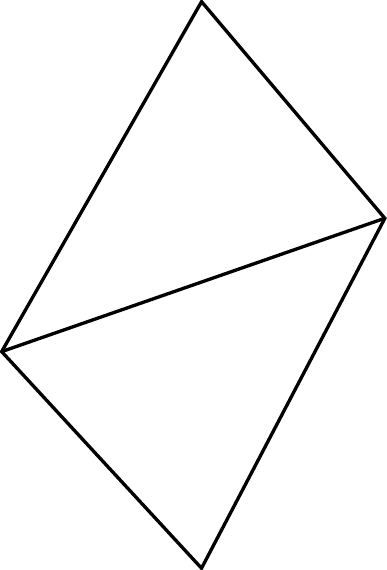}{50pt}\longleftrightarrow\epsh{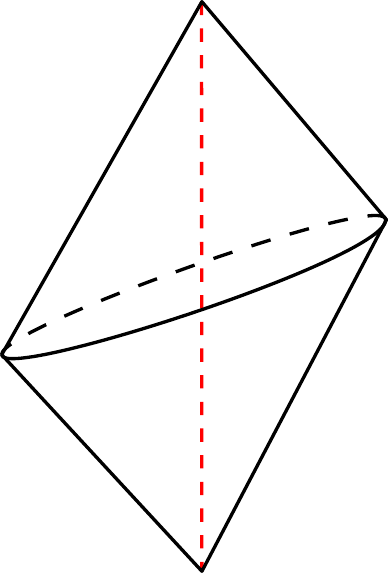}{50pt} $
    \hspace{10pt} }
  \caption{Elementary moves}
  \label{F:moves}
\end{figure}
It is understood that a negative move is only allowed if it is the inverse of
a positive move (hence these moves do not change the graph).  The {\lune} move
may be expanded as a composition of {\bubble} moves and {\Pachner} moves (see
\cite{BB}, Section 2.1), but it will be convenient for us to use the {\lune}
moves directly.

\begin{proof}[Proof of Theorem \ref{L:Toplemma-}]
  As in the statement of the theorem, let $\Tb$ be a quasi-regular
  triangulation of $\partial M$ and let $\Y$ be a graph in $M$ rooted in
  $\Tb$.

  In the proof, we shall use the language of skeletons of 3-manifolds dual to
  the language of triangulations (see, for instance \cite{Tu,BB}).  A skeleton
  of $M$ is a 2-dimensional polyhedron $P$ in $M$ such that $M\setminus P$ is
  a disjoint union of open (half) 3-balls and locally $P$ looks like a plane,
  or a union of 3 half-planes with common boundary line in $\R^3$, or a cone
  over the 1-skeleton of a tetrahedron.  A typical skeleton of $M$ is
  constructed from a triangulation $\T$ of $M$ by taking the union $P_\T$ of
  the 2-cells dual to its edges. This construction establishes a bijective
  correspondence $\T\leftrightarrow P_\T$ between the quasi-regular
  triangulations $\T$ of $M$ and ``quasi-regular'' skeletons $P$ of $M$ such
  that every 2-face of $P$ is a disk adjacent to two distinct components of
  $M-P$.  If $(\T,\YY)$ is an $H$-triangulation of $(M,Y)$, to specify the
  graph $\YY$ in dual skeleton we provide some faces of $P_\T$ with dots such
  that each component of $M-P_\T$ is adjacent to at least two (distinct)
  dotted faces.  These dots correspond to the intersections of $\YY$ with the
  2-faces.  Here the vertices $\Y_0$ of $\Y$ become a set of distinguished
  components of $M\setminus P_\T$.  The elementary moves on the
  $H$-triangulations may be translated to this dual language and give the
  well-known Matveev-Piergallini moves on skeletons adjusted to the setting of
  Hamiltonian graphs, see~\cite{BB}.

  Take a collar spine $S'$ of $(M',\Gamma)$ as in \cite[Theorems 6.4.A \&
  B]{TV} where $M'=M\setminus V(\Y)$ for an open tubular neighborhood $V(\Y)$
  of $\Y$ and $\Gamma$ is a graph in the boundary of $M'$ which is the
  disjoint union of the trivalent graph dual to the triangulation of $\partial
  M$ union one meridian circle for each edge (and circle component) of $\Y$.
  Then we fill the meridian circle of $S'$ with a doted disk and get a
  skeleton $S$ of $M$ in the complement of which a (half) 3-ball has at least
  one region in its boundary with a dot.  Taking the dual of $S$ we obtain an
  $H$-triangulation $(\T,\YY)$ of $(M,Y)$ relative to $\Tb$ which may not be
  quasi-regular.  If $\T$ is not quasi-regular then $S$ contains a region
  which has the same (half) 3-cell on both sides or a region with both sides
  in $\Y_0$.  We call such a region \emph{bad}.
 
  The skeleton $S$ can be modified by the Matveev-Piergallini moves dual to
  the elementary moves to obtain a skeleton which is dual to a quasi-regular
  $H$-triangulation.  In the case when $M$ is closed this was done by
  Baseilhac and Benedetti \cite{BB} in the proof of Proposition 4.20.  In
  particular, for a 3-cell with a bad region they do a dual {\bubble} move on
  $S$ which introduces a new 3-cell which is partially bounded by a two cell
  which they call a capping disk.  Then they use dual {\Pachner} moves to
  slide a portion of the capping disk over the bad region, thus dividing a
  (half) 3-cell with the bad region into two 3-cell (or a 3-cell and a half
  3-cell) with no bad regions.  The proof of \cite{BB} can be applied in the
  case when $M$ has a boundary and $Y$ is a graph.
\end{proof}

\begin{proposition}\label{L:Toplemma}
  Let $M$ be a compact 3-manifold, $\Tb$ a quasi-regular triangulation of
  $\partial M$ and $\Y$ a graph in $M$ rooted in $\Tb$.  Any two
  $H$-triangulations of $(M,\Y)$ relative to $\Tb$ can be related by a finite
  sequence of {\bubble} moves and {\Pachner} moves in the class of
  $H$-triangulations of $(M,\Y)$ relative to $\Tb$.
\end{proposition}
\begin{proof}
  Let $(\T_1,\YY_1)$ and $(\T_2,\YY_2)$ be $H$-triangulation of $(M,Y)$
  relative to $\Tb$.  By doing positive {\bubble} moves if necessary we can
  assume that each edge of $\Y$ is realized by the same number of edges in
  both $\YY_1$ and $\YY_2$.  Let $P_{T_i}$ be the skeleton dual to $\T_i$
  where each edge of $\YY_i$ corresponds to a dotted face as in the proof of
  Theorem \ref{L:Toplemma-}.  Replacing a neighborhood of each of these dots
  with a circle we obtain a collar spine $S_i$ of $(M\setminus
  V(\Y),\Gamma_i)$ where $\Gamma_i$ is the graph in the boundary of
  $M\setminus V(\Y)$ formed by these circles.  The assumption on the edges of
  $\YY_1$ and $\YY_2$ implies that $\Gamma_1$ and $\Gamma_2$ are isotopic.
  Then Theorem 6.4B of \cite{TV} implies that two spines $S_1$ and $S_2$ are
  related by a sequence of dual {\Pachner} and {\lune} moves.  In Proposition
  4.23 of \cite{BB}, Baseilhac and Benedetti prove that by adding and sliding
  capping disks this sequence of moves leads to a sequence of moves in the
  class of $H$-triangulations connecting $(\T,\YY)$ and $(\T',\YY')$.
\end{proof}

\begin{lemma}\label{L:AdmMove}
  Let $\col$ be an admissible $\Gr$-coloring of $\T$.  Suppose that
  $(\T',\YY')$ is an $H$-triangulation obtained from $(\T,\YY)$ by a negative
  {\Pachner}, {\lune} or {\bubble} move.  Then $\col$ restricts to an
  admissible $\Gr$-coloring $\col'$ of $\T'$ and
\begin{equation}\label{tvt}
  TV(\T,\YY,\col)=TV(\T',\YY',\col').
\end{equation}
\end{lemma}
\begin{proof} The proof is exactly the same as the one of \cite[Lemma
  18]{GPT2}.  That is we can translate the {\Pachner}, {\lune}, and {\bubble} move
  into algebraic identities: the Biedenharn-Elliott identity, the
  orthonormality relation and Equation \eqref{E:algLbubble}, respectively.
\end{proof}

Let $(\T,\YY)$ be a $H$-triangulation of $(M,\Y)$.  A \emph{$\Gr$-gauge of
  $\T$} is a map from the set of vertices $\T_0$ of $\T$ to $\Gr$.  The
$\Gr$-gauges of $\T$ form a multiplicative group which acts on the set of
$\Gr$-colorings of $\T$ as follows.  If $\delta$ is a $\Gr$-gauge of $\T$ and
$\wp$ is a $\Gr$-coloring of $\T$, then the $\Gr$-coloring $\delta \act \Phi$
is given by
$$(\delta\act\Phi)(e)=\delta(v^-_e)\,\Phi(e)\, \delta(v^+_e)^{-1},$$
where $v^-_e$ (resp.\@ $v^+_e$) is the initial (resp.\@ terminal) vertex of an
oriented edge $e$.  
Let $\T_0^{i}$ be the set of vertices of $\T$ that are not
vertices of $\YY$ so $\T_0=\Y_0\sqcup\T_0^{i}$.  Let $\Gr^{\T_0^{i}}$ be the
set of $\Gr$-gauges $\delta:\T_0\to \Gr$ such that $\delta(v)=1$ if $v\in
\Y_0$.  

\begin{proof}[Proof of Theorem \ref{L:admColoringsExist}]
  Take any $\Gr$-coloring $\col$ of $\T$ which restricts to $\rr$.
  We say that a vertex $v$ of $\T\setminus \Y_0$ is {\it bad} for $\col$ if
  there is an oriented edge $e$ in $\T$ outgoing from $v$ such that
  $\col(e)\in \X$.  Since $(\T,\YY)$ is an $H$-triangulation any edge not in
  $\partial M$ has at least one endpoint not in $\Y_0$.  Thus, it is clear
  that $\col$ is admissible if and only if $\col$ has no bad vertices. We show
  how to modify $\col$ to reduce the number of bad vertices.  Let $v $ be a
  bad vertex for $\col$ and let $E_v$ be the set of all oriented edges of $\T$
  outgoing from $v$.  Pick any $$g\in \Gr\setminus\left( \bigcup_{e\in
      E_v}(\X\col(e)^{-1})\right).$$

  Let $\delta^{v,g}$ to be the $\Gr$-valued gauge defined by
  \begin{equation}\label{E:DefDeltaGauge}
    \delta^{v,g}(v')=\left\{\begin{array}{ll}g, &  \text{if $v=v'$}\\ 1, &
        \text{else.}\end{array}\right. 
  \end{equation}
  Then $\delta^{v,g}\act\col$ takes values in $\Gr\setminus \X$ on all edges
  of $\T$ incident to $v$ and takes the same values as $\col$ on all edges of
  $\T$ not incident to $v$.  Here we use the fact that the edges of $\T$ are
  not loops which follows from the quasi-regularity of $\T$.  The
  transformation $\col \mapsto \delta^{v,g}\act\col$ decreases the number of
  bad vertices.  Repeating this argument, we find a $\Gr$-coloring without bad
  vertices which restricts to $\rr$.
\end{proof}
Remark that for $\Phi\in \MG{M,\T_0}$, and for $g\in\Gr^{\T_0^{i}}$, the
$\Gr$-colorings $\Phi$ and $\delta_g\act{\Phi}$ restrict to the same
representation in $\MG{M,\Y_0}$.
\begin{lemma}\label{L:gauge}
  If $\Phi_1,\Phi_2\in \MG{M,\T_0}$ restrict to the same representation
  $\rr\in \MG{M,\Y_0}$, then there exists an element $g\in\Gr^{\T_0^{i}}$ with
  $\delta_g\act{\Phi_1}=\Phi_2$.
\end{lemma}
\begin{proof}
  We assume in the proof, without lose of generality, that $M$ is connected.
  When $M$ is closed and $\Y_0$ is empty, the elements of $\MG M$ bijectively
  correspond to the $\Gr$-colorings of $\T$ considered up to gauge
  transformations (see for example \cite{Pe}).  Therefore, assume $\Y_0$ is
  not empty.

  We first distinguish a set of edge $\wb \YY$.  Let start with $\wb \YY=\YY$.
  Suppose that some vertex of $\T$ is not connected to $\Y_0$ by a path in
  $\wb \YY$.  We say that such a vertex is isolated.  Then there exists an
  edge of $\T$ whose ends are an isolated vertex on one side and a vertex not
  isolated in the other side.  We then add this edge to $\wb \YY$.  We repeat
  this process until there is no more isolated vertex.  One can check that he
  had added one edge of $\T$ for each circle component of $\YY$.  A path of
  length $n$ is a sequence of $n$ oriented edges $e_1,\ldots,e_n$ where the
  target of $e_i$ is equal to the source of $e_{i+1}$.  Choose an orientation
  of the edges of $\Y$.  This orientation can be extended to a compatible
  orientation of the edges of $\wb\YY$ such that for every vertex $v$ of $\T$
  there exists an unique shortest path $\gamma_v$ in $\wb\YY$ from some vertex
  of $\Y_0$ to $v$.  To be more precise, recall that every oriented edge of
  $\Y$ is realized by a path in $\YY$.  So there is an unique compatible
  orientation of the edges of $\YY$.  As for the added edges of
  $\wb\YY\setminus\YY$, we choose the orientation toward the originally
  isolated vertex.
  The union of paths $\gamma_v$ where $v$ is a vertex of $\T$ is a disjoint
  union of trees rooted in $\Y_0$.  We say a vertex $v$ of $\T$ is \emph{good}
  if $\Phi_1(e)=\Phi_2(e)$ for every edge $e$ of $\gamma_v$.  In particular,
  every vertex of $\Y_0$ is good.

  Suppose all vertices of $\T$ are good.  Then $\Phi_1=\Phi_2$ because for any
  oriented edge $e$ of $\T$ from $v_1$ to $v_2$, we
  have $$\Phi_1(e)=\Phi_1(\gamma_{v_1})^{-1}\rr(\gamma_{v_1}e\gamma_{v_2}^{-1})
  \Phi_1(\gamma_{v_2})=\Phi_2(\gamma_{v_1})^{-1}
  \rr(\gamma_{v_1}e\gamma_{v_2}^{-1})\Phi_2(\gamma_{v_2})=\Phi_2(e).$$

  Thus, it remains to show that we can modify $\Phi_1$ by a gauge
  transformation to increase the number of good vertices.  Let $e$ be an edge
  of $\wb\YY$ from $v_1$ to $v_2$ where $v_1$ is good and $v_2$ is not good.
  Then $\gamma_{v_2}=\gamma_{v_1}e$ and so
  $\delta^{v_2,g}\act{\Phi_1}(e)=\Phi_2(e)$ where $g=\Phi_2(e)^{-1}\Phi_1(e)$.
  If $v$ is a good vertex the path $\gamma_v$ can not pass through $v_2$ as
  $\Phi_1(e)\neq\Phi_2(e)$.  Therefore, the set of good vertices strictly
  increases when one replace $\Phi_1$ with $\delta^{v_2,g}\act{\Phi_1}$.
\end{proof}

\begin{lemma}\label{L:Coboundary}
  Let $v_0\in\T_0^{i}$ and $c:\T_0\to \Gr$ be a map such that $c(v)=1$ for all
  $v\neq v_0$ and $c(v_0)\notin \X$.  If $\col$ and $\delta_c\act{\col}$ are
  admissible $\Gr$-colorings of $\T$, then
  $TV(\T,\YY,\col)=TV(\T,\YY,\delta_c\act{\col})$
\end{lemma}

\begin{proof}
  A similar claim is proved in Lemma 27 of \cite{GPT2} when $\Gr$ is a
  commutative group.  Since $v_0$ is in $\T_0^i$ the vertex $v_0$ is the
  endpoint for exactly two edges $e_1,e_2$ of $\Y$.  As shown in the proof of
  Lemma 27 of \cite{GPT2} there is a sequence of elementary moves through
  quasi-regular triangulation with admissible colorings from $(\T,\YY,\col)$
  to $(\T,\YY,\delta_c\act{\col})$.  This sequence starts with a positive
  {\bubble} move at $e_1$ (where we introduce the color $c(v_0)$) and ends
  with a negative {\bubble} move at $e_2$.
\end{proof}

\begin{lemma}\label{C:ConstCobound}   If $\col$ and $\col'$
  are two admissible $\Gr$-colorings of $\T$ that restrict to the same
  representation $\rr\in\MG{M,\Y_0}$, then $TV(\T,\YY,\col)=TV(\T,\YY,\col')$.
\end{lemma}
\begin{proof}
  Since $\col$ and $\col'$ represent the same element of $\MG{M,\Y_0}$ by
  Lemma \ref{L:gauge} there exists pairs $(v_i,g_i)\in\{\text{vertices of
    $\T$}\}\times \Gr$ such that
  $$
  \Phi'=\delta^{v_n,g_n}\delta^{v_{n-1},g_{n-1}}\cdots\delta^{v_1,g_1}\act\Phi.
  $$
  where $\delta^{v_i,g_i}$ is the $\Gr$-valued gauge, defined as in
  \eqref{E:DefDeltaGauge}, which takes non-trivial values at a single vertex
  $v_i$, for all $i=1,\ldots,n$.  We prove the desired equality by induction on
  $n$.  If $n=0$ then $\wp'=\wp$ and the equality is clear.  Otherwise, let
  $E_{1}$ be the set of (oriented) edges of $\T$ beginning at $v_1$. Pick any
  $$g\in \Gr\setminus\Big[ \X\cup \bigcup_{e\in E_{1}} \big(\X\wp(e)^{-1}\big)
   \cup \bigcup_{e\in E_{1}} \big(\X \wp'(e)^{-1}g_1\big)\Big].$$

   Then $\delta^{v_1,g}\act\wp$ and
   $\delta^{v_n,g_n}\delta^{v_{n-1},g_{n-1}}\cdots\delta^{v_2,g_2}
   \delta^{v_1,g}\act\wp=\delta^{v_1,gg_1^{-1}}\act{\wp'}$ are admissible
   colorings.  Lemma~\ref{L:Coboundary} and the induction assumption imply
   that
  \begin{align*}
  TV(\T,\LL,\wp,c)&=TV(\T,\LL,\delta^{v_1,g}\act\wp,c)\\
  &=TV(\T,\LL,\delta^{v_n,g_n}\delta^{v_{n-1},g_{n-1}}
  \cdots\delta^{v_2,g_2}\delta^{v_1,g}\act\wp,c)\\
  &=TV(\T,\LL,\wp',c).
  \end{align*}

\end{proof}

\begin{theorem}\label{T:one-move}
  Let $(\T,\YY)$ and $(\T',\YY')$ be two $H$-triangulations of $(M,\Y)$
  relative to $\Tb$ such
  that $(\T',\YY')$ is obtained from $(\T,\YY)$ by a single {\Pachner} move,
  {\bubble} move, or {\lune} move.  Then for any admissible $\Gr$-colorings
  $\col$, $\col'$ of $\T$, $\T'$, respectively, restricting to the same
  element of $\MG{M,\Y_0}$, we have
  $$
  TV(\T,\YY,\col)=TV(\T',\YY',\col').
  $$
\end{theorem}
\begin{proof}
  For concreteness, assume that $\T'$ is obtained from $\T$ by a negative
  move.  The admissible $\Gr$-coloring $\col$ of $\T$ restricts to an
  admissible $\Gr$-coloring $\col''$ of $\T'$ which restrict to the same
  element of $\MG{M,\Y_0}$.  Now Lemma~\ref{L:AdmMove} implies that
  $TV(\T,\YY,\col)=TV(\T',\YY',\col'')$ and Lemma \ref{C:ConstCobound} implies
  that $TV(\T',\YY',\col'')=TV(\T',\YY',\col')$.
\end{proof}

\begin{proof}[Proof of Theorem \ref{T:inv}]
  From Proposition \ref{L:Toplemma} we know that any two $H$-triangulation of
  $(M,Y)$ are related by a finite sequence of elementary moves.  Then the
  result follows from the Theorem \ref{T:one-move} by induction on the number
  of moves.
\end{proof}

\section{Quantum groups at root of unity}\label{S:RepQG}
In this section we first recall some of the deep results established by De
Concini, Kac, Procesi, Reshetikhin and Rosso in the series of papers
\cite{DK}--\cite{DPRR}.  We then prove that the modules studied in these
papers give rise to the topological invariants described in Sections
\ref{S:RTLink}, \ref{S:TopGSph} and \ref{S:HQFT}.
In particular, in Subsections \ref{SS:Uqg(H)} and \ref{SS:PBW} we give definitions and properties of the algebras $\Uqg$ and  $\UqgH$ associated to a simple Lie algebra $\g$.   
In Subsection \ref{SS:weightmod} we discuss the category of weight modules over these algebras.  
Since the constructions in the first three subsections are some what abstract, in Subsection \ref{SS:ExSl2Mod} we give a concrete example when $\g=\sll_2$.    Then we give descriptions of the pivotal structure, braiding and modified dimensions.   The section is concluded with a theorem stating that the category of finite dimensional weight modules over $\Uqg$ is a relative $\Gr$-spherical category.  

\subsection{The quantum groups $\Uqg$ and  $\UqgH$}\label{SS:Uqg(H)}
Let $\g$ be a simple finite-dimensional complex Lie algebra of rank $\rk$ and
dimension $2N+\rk$ with a root system.  Fix a set of simple roots
$\{\alpha_1,\ldots,\alpha_\rk\}$ and let $\roots^{+}$ be the corresponding set
of positive roots.  Also, let $A=(a_{ij})_{1\leq i,j\leq \rk}$ be the Cartan
matrix corresponding to these simple roots.  There exists a diagonal matrix
$D=\diag(d_1,\ldots,d_\rk)$ such that $DA$ is symmetric and positive definite.
Let $\h$ be the Cartan subalgebra of $\g$ generated by the vectors
$H_1,\ldots,H_\rk$ where $H_j$ is determined by $\alpha_i(H_j)=a_{ji}$.  Let
$L_R$ be the root lattice which is the $\Z$-lattice generated by the simple
roots $\{\alpha_i\}$.  Let $\ang{\;,\;}$ be the form on $L_R$ given by
$\ang{\alpha_i,\alpha_j}=d_ia_{ij}$.  Let $L_W$ be the weight lattice which is
the $\Z$-lattice generated by the elements of $\h^*$ which are dual to the
elements $H_i$, $i=1\cdots\rk$.  Let
$\rho=\frac12\sum_{\alpha\in\roots^{+}}\alpha\in L_W$.

Let $\ro$ be an odd integer such that $\ro\geq 3$ (and $\ro\notin 3\Z$ if
$\g=G_2$).  Let $q=\e^{2i\pi/\ro}$ and for $i=1,\ldots,\rk$, let
$q_i=q^{d_i}$.  For $x\in \C$ and $k,l\in \N$ we use the notation:
$$
q^x=\e^{\frac{2i\pi x}\ro},\quad \qn x_q=q^x-q^{-x},\quad \qN x_q=\frac{\qn
  x_q}{\qn 1_q},\quad\qN k_q!=\qN1_q\qN2_q\cdots\qN k_q,\quad{k\brack
  l}_q=\frac{\qN{k}_q!}{\qN{l}_q!\qN{k-l}_q!}.
$$    
Remark for $x\in\C$, 
 $\qn x_q=0$ if and only if $ x\in\frac\ro{2}\Z$.

 Next we consider two quantum groups associated to $\g$.  The Drinfeld-Jimbo
 quantum group $\Uqg$ is defined as the algebra with generators $K_\beta, \,
 X_i,\, X_{-i}$ for $\beta\in \La, \,i=1,\ldots,\rk$ and relations
\begin{eqnarray}\label{eq:rel1}
 & K_0=1, \quad K_\beta K_\gamma=K_{\beta+\gamma},\, \quad K_\beta X_{\sigma
    i}K_{-\beta}=q^{\sigma \ang{\beta,\alpha_i}}X_{\sigma i}, & \\
 \label{eq:rel2} &[X_i,X_{-j}]=
 \delta_{ij}\frac{K_{\alpha_i}-K_{\alpha_i}^{-1}}{q_i-q_i^{-1}}, &  \\
 \label{eq:rel3} & \sum_{k=0}^{1-a_{ij}}(-1)^{k}{{1-a_{ij}} \brack
   k}_{q_i} X_{\sigma i}^{k} X_{\sigma j} X_{\sigma i}^{1-a_{ij}-k} =0,
 \text{ if }i\neq j &
\end{eqnarray}
where $\sigma=\pm 1$.  The algebra $\Uqg$ is a Hopf algebra with coproduct
$\Delta$, counit $\epsilon$ and antipode $S$ defined by
\begin{align*}
  \Delta(X_i)&= 1\otimes X_i + X_i\otimes K_{\alpha_i}, &
  \Delta(X_{-i})&=K_{\alpha_i}^{-1} \otimes
  X_{-i} + X_{-i}\otimes 1,\\
  \Delta(K_\beta)&=K_\beta\otimes K_\beta, & \epsilon(X_i)&=
  \epsilon(X_{-i})=0, &
  \epsilon(K_{\alpha_i})&=1,\\
  S(X_i)&=-X_iK_{\alpha_i}^{-1}, & S(X_{-i})&=-K_{\alpha_i}X_{-i}, &
  S(K_\beta)&=K_{-\beta}.
\end{align*}

The \emph{unrolled quantum group} $\UqgH$ is the algebra generated by
$K_\beta, X_i,X_{-i},H_i$ for $\beta\in \La, \,i=1,\ldots,\rk$ with Relations
\eqref{eq:rel1},\eqref{eq:rel2},\eqref{eq:rel3} plus the relations
\begin{align}\label{eq:relH}
  [H_i,X_{\epsilon j}]=\sigma a_{ij}X_{\sigma j},&&[H_i,H_j]=[H_i,K_\beta]=0
\end{align}
where $\sigma=\pm 1$.  The algebra $\UqgH$ is a Hopf algebra with coproduct
$\Delta$, counit $\epsilon$ and antipode $S$ defined as above on $K_\beta, \,
X_i,\, X_{-i}$ and defined on the elements $H_i$ for $i=1,\ldots,\rk$ by
\begin{align*}
  \Delta(H_i)= 1\otimes H_i + H_i\otimes 1, && \epsilon(H_i)=0, &&S(H_i)=-H_i.
\end{align*}

As a vector space $\UqgH$ is isomorphic to $\UH\otimes_\C\Uqg$ where
$\UH=\C[\h]$ is the symmetric algebra of $\h\otimes_\Z\C$.  The obvious map
$\Uqg\to\UqgH$ is an injective morphism of Hopf algebra.  Using this morphism
we can identify $\Uqg$ with a Hopf subalgebra of $\UqgH$.

\subsection{The PBW basis}\label{SS:PBW}
Let $\alpha_1,\ldots,\alpha_\rk$ be any ordering of the simple positive roots
and let $s_{i_1}s_{i_2}\cdots s_{i_N}$ be a reduced decomposition of the
longest element of the Weyl group.  Then
$\beta_1=\alpha_{i_1},\,\beta_2=s_{i_1}\alpha_{i_2},\ldots,\,
\beta_N=s_{i_1}s_{i_2}\cdots s_{i_{N-1}}\alpha_{i_N}$ is a total ordering of
the set of positive roots $\roots^{+}$.  We call
$\beta_*=(\beta_1,\ldots,\beta_N)$ a {\em convex order} of $\roots^{+}$.  For
each $i=1,\ldots,N$ let $X_{\pm\beta_i}$ be the \emph{positive
  (resp. negative) root vector} of $\Uqg$ (see for example Section 8.1 and 9.1
of \cite{CP}). We say that $(X_{\pm\beta})_{\beta\in\beta_*}$ is a {\em convex
  set of root vectors} which depends of the order of
$\alpha_1,\ldots,\alpha_\rk$ and $s_{i_1},s_{i_2},\ldots,s_{i_N}$.

If $y_1,\ldots,y_m$ is a set of elements in an algebra then let
$Alg\ang{y_i:i=1,\ldots,m}$ and $ Span_\C \ang{y_i:i=1,\ldots,m}$ be the
algebra generated and the complex vector space spanned by these elements,
respectively.  We will use the following notation (for $\Uqg, \UqgH, \UH$ see
Section \ref{SS:Uqg(H)}):
\begin{align*}
  \Uqn+&=Alg \ang{X_i:i=1,\ldots,n},\\
  \Uqn-&=Alg \ang{X_{-i}:i=1,\ldots,n},\\
  \Uqn{++}&=\Uqn+\cap\ker\epsilon,\\
  \Uqn{--}&=\Uqn-\cap\ker\epsilon,\\
  \Uqn{+<}&=Span_\C \ang{X_{\beta_1}^{k_1}\dots X_{\beta_N}^{k_N} : \ 0\leq
    k_i<\ro },\\ 
  \Uqn{-<}&=Span_\C \ang{X_{-\beta_1}^{k_1}\dots X_{-\beta_N}^{k_N} : 0\leq
    k_i<\ro},\\ 
  \ZUo&=Alg \ang{K_{\gamma}^\ro, X^\ro_{ \beta}:\gamma \in L_W,
    \beta\in\beta_*},\\  
  \UK&=Alg \ang{K_{\omega_i} : \{\omega_i\}_{i=1,\ldots,\rk} \text{ is the set
      of fundamental weights}}.
\end{align*}
Remark (1) the algebra $\UK$ is a commutative Hopf subalgebra of $\Uqg$,
(2) the spaces $\Uqn{+<}$ and $\Uqn{-<}$ depend on the convex order
$\beta_*$ of $\roots^{+}$ and (3) it is proven in \cite{DKP} that
$\ZUo=(\ZU\cap\Uqn-)(\ZU\cap\UK)(\ZU\cap\Uqn+)$ and thus $\ZUo$ is independent
of $\beta_*$.

We will need the following weak version of the PBW theorem, for a proof see
for example \cite{DKP2,CP}.
\begin{theorem}\label{T:PBW}
  The multiplication map defines vector space isomorphisms:
  $$
  \Uqn-\otimes \UK\otimes\Uqn+\stackrel\sim\to\Uqg
  \stackrel\sim\leftarrow \Uqn+\otimes \UK\otimes\Uqn-.
  $$
  Furthermore, for $\sigma\in\{+,-\}$, the set 
  $\{X_{\sigma\beta_1}^{k_1}X_{\sigma\beta_2}^{k_2}\cdots
  X_{\sigma\beta_N}^{k_N} : (k_1,\ldots,k_N)\in\N^N\}$ is a basis of
  $\Uqn\sigma$ and the set $\{K_\alpha:\alpha\in L_W\}$ is a basis of $\UK$.
\end{theorem}
Note that in Theorem \ref{T:PBW} one can reverse the order of the products for
the monomials of the basis of $\Uqn\pm$ and still obtain a basis.
From Proposition 5.6 of \cite{DKP} we have that $\ZUo$ is a Hopf subalgebra of
$\Uqg$ contained in the center $\Uqg$.  

Let $\wb \Gr$ be the simply connected group associated to $\g$ and let
$T=\exp\h$ be the maximal torus of $\wb \Gr$.  Let $U^+$ and $U^-$ be the
unipotent radicals of the opposite Borel subgroups $B^+$ and $B^-$,
respectively.  Recall from \cite{DKP} that, as a Hopf algebra, $\ZUo$ is the
coordinate
ring of the \emph{dual group} $\Gr$ which is the subgroup of $B^+\times B^-$
given by the kernel of the composition $B^+\times B^-\to T\times T \to T$
where $T\times T \to T$ is the multiplication and $B^+\times B^-\to T\times T$
is the quotient modulo the unipotent radical (also see \cite[Section
6.3]{DPRR}).
As a variety, $\Gr$ can be identified with $U^-\times T \times U^+$ where the
inclusion $U^-\times T \times U^+\to B^+\times B^-$ is given by
$(u,t,v)\mapsto(tu,t^{-1}v)$.  Thus, $\Gr$ can be identified with the group of
ring homomorphisms $g:\ZUo\to \C$ where the group structure is given by
$gh=(g\otimes h)\circ \Delta$.  Moreover, $T$ can be identified with the
subgroup of $\Gr$ consisting of elements $g$ 
which are zero on any element of $\ZUo\cap(\Uqn{++}\cup\Uqn{--})$.

The authors of \cite{DPRR} define a subset of $\Gr$ called the unramified locus.
Let $\X$ be the complement of the unramified locus in $\Gr$.    
The set $\Gr\setminus\X$ is a Zariski dense open subset of $\Gr$.
 
We will use the group $\Gr$ throughout the rest of the paper.  In particular,
we will show $\Gr$ gives a grading on a certain category of $\Uqg$-modules
where graded pieces corresponding to the unramified locus are semi-simple.
With this in mind we recall the following theorem.
\begin{theorem}\label{Th:Ugss} 
  \emph{(}\cite{DK,DKP,DKP2} and \cite[Prop. 5.5]{DPRR}\emph{)}.  For any
  $g\in\Gr$ the following are equivalent:
  \begin{enumerate}
  \item $\Uqg\otimes_{g:\ZUo\to\C}\C$ is a semi-simple algebra.
  \item $g\notin \X$.
  \item There are (at least) $\ro^{\rk}$ non isomorphic irreducible
    $\Uqg$-modules of dimension $\ro^{N}$ on which any element of $z\in\ZUo$
    acts by the scalar $g(z)$.
  \end{enumerate}
\end{theorem}
In Corollary \ref{C:GsNil} we will give a partial description of $\X$.  

\subsection{Weight modules} \label{SS:weightmod} 
A {\em multiplicative weight} is a $\C$-algebra homomorphism $\kappa:\UK \to
\C$.  Given a $\Uqg$-module $V$ and a multiplicative weight $\kappa$, let
$E_\kappa(V)$ be the weight space consisting of elements $v\in V$ such that
any $x\in \UK$ acts on $v$ as the scalar $\kappa(x)$.  A $\Uqg$-module $V$ is
called a {\em weight module} if $V$ splits as a direct sum of its weight 
spaces and if all elements of $\ZUo$ act diagonally on it.
 
On the other hand, an {\em additive weight} is a $\C$-algebra homomorphism
$\lambda:\UH \to \C$.  Given a $\UqgH$-module $V$ and an additive weight
$\lambda$, let $E_\lambda(V)$ be the weight space consisting of elements $v\in
V$ such that any $x\in \UH$ acts on $v$ as the scalar $\lambda(x)$.  A
$\UqgH$-module $V$ is called a {\em weight module} if $V$ splits as a direct
sum of its weight spaces, all elements of $\ZUo$ act diagonally on it, and for
any $\beta=\sum_i b_i\alpha_i\in \La$, the element $K_\beta$ acts on $V$ as
the scalar $\prod_i q_i^{b_iH_i}$.  An additive weight $\lambda$ induces a
multiplicative weight $\kappa=q^{\lambda}$ defined by the rule
$\kappa(K_\beta)=\prod_i q_i^{b_i\lambda(H_i)}=q^{\ang{\beta,\lambda}}$ for
any $\beta=\sum_i b_i\alpha_i\in \La$.  
If $\gamma$ is a multiplicative or additive weight we call a vector $v\in
E_\gamma(V)$ a {\em weight vector} with (multiplicative or additive) weight
$\gamma$.
  The vector space $\UqgH$ with the adjoint action is a weight module over
$\UqgH$.  If $x\in E_\lambda(\UqgH)$ we say $x$ has \emph{weight} $\lambda$.

Let $\catU$ and $\catH$ be the categories of finite dimensional weight modules
over $\Uqg$ and $\UqgH$, respectively.  Let $\catU_T$ be the full subcategory
of $\catU$ whose objects are the modules on which $X_{\pm\beta}^{\ro}$ acts by
$0$, for all $\beta\in\beta_*$.  For $g\in\Gr$, let $\catU_g$ be the full
subcategory of $\catU$ whose objects are the modules for which any $x\in\ZUo$
acts by the scalar $g(x)$.  Then $\catU=\bigoplus_{g\in \Gr}\catU_g$
is $\Gr$-graded.  
Also, if $t\in T\subset \Gr$ we have that $\catU_t\subset \catU_T$ and
$\catU_T=\bigoplus_{t\in T}\catU_t$.

Let $V$ be an object of $\catH$.  The first equality in Equation
\eqref{eq:relH} implies that the action of $X_j$, for $j=1,\ldots,\rk$,
translates the weight spaces of $V$.  Similarly, any
$x\in\Uqn{++}\cup\Uqn{--}$ translate non trivially the weights.  Combining
this with the fact that $V$ is finite dimensional one sees that the operator
$x:V\to V$ is nilpotent.  In particular for any $\beta\in\beta_*$, 
the operator
$X_{\pm \beta}^\ro:V\to V$ is nilpotent and thus the zero operator since
$X_{\pm \beta}^\ro\in \ZUo$.  By forgetting the
action of $\UH$ the module $V$ becomes a weight module over $\Uqg$.  Moreover,
since $X_{\pm \beta}^\ro:V\to V$ is the zero operator, $V$ is an object in
$\catU_T$.  Thus, we have a forgetful functor
\begin{equation}
\label{E:ForgetfulF}
\FF:\catH\to\catU_T
\end{equation} 
and we say that objects of
$\catU_T$ and $\catH$ are {\em nilpotent modules}.  

For any nilpotent module $V$ there exists a vector $v_{\pm}$ such that $X_{\pm i}\,v_\pm=0$ for all $i=1,\ldots,\rk$.  Such vectors $v_+$ and $v_-$ are called  {\em
  highest weight vectors}  and {\em lowest weight vectors}, respectively.  
If $V$ is simple then $v_\pm$ is unique up to a scalar.

\begin{theorem}\label{T:catU_gSS}
  Let 
  $V\in\catU_g$ and $V'\in\catU_{g'}$, for $g,g'\in\Gr$.  Then
  \begin{enumerate}
  \item $V\otimes V'\in\catU_{gg'}$,
  \item $V^{*}\in\catU_{g^{-1}}$,
  \item $\catU_g$ is semi-simple if and only if $g\notin \X$.   
  \end{enumerate}
\end{theorem}
\begin{proof}
  The first two items follows from the fact that $\ZUo$ is a Hopf subalgebra
  of $\Uqg$. The third is a consequence of Theorem \ref{Th:Ugss}.
\end{proof}

\subsection{Example: $\Uqg=\Uqg(\sll_2)$ simple weight modules}\label{SS:ExSl2Mod}
To illustrate the different type of modules we use in this paper we
give a description of simple $\Uqg$-weight modules when $\g=\sll_2$,
for more details see Chapter 11.1 of \cite{CP}.  For general $\g$ also
see \cite{CP}.  The algebra $\Uqg(\sll_2)$ has four generators
$X_{1}$, $X_{-1}$ and $k$ and $k^{-1}$.  In the notation above, the generator $k$ is $K_\alpha$ where $\alpha$ is the fundamental weight.  

The group $\Gr$ can be identified with the group of ring homomorphisms
$g:\ZUo\to \C$ where $\ZUo=\C[k^{\pm r},X_1^r,X_{-1}^r]$ is the
central sub-Hopf algebra.  As a Lie group, $\Gr$ can be embedded in
$B^+\times B^-$ via the map 
\newcommand{\mat}[4]{{\small
    \left(\begin{array}{cc} #1&#2\\#3&#4
    \end{array}\right)}}
$$g\mapsto\pp{
  \mat{g(k^{-r})}00{g(k^{r})}\times\mat1{g(X_1^r)}01,
  \mat10{g(X_{-1}^r)}1\times\mat{g(k^{r})}00{g(k^{-r})}}.$$ 
It can be shown that $\X$ is the subset of $\Gr$ defined by the equation
$$g\in\X\iff g((\qn1^{2r}X_1^rX_{-1}^r-k^{2r}-k^{-2r})^2)=4.$$
Finally, $T\subset\Gr$ is the subgroup of elements $t$ such that
$t(X_1^r)=t(X_{-1}^r)=0$ and $t\in T\cap\X$ if and only if
$t(k^{4r})=1$.

A $\Uqg(\sll_2)$-weight module is finite dimensional.  
As we explain now, there are two kinds of simple $\Uqg(\sll_2)$-weight module: highest weight modules and cyclic.  
Let $V$ be a
simple $\Uqg(\sll_2)$-weight module of degree $g$.  In \cite{CP}, $V$
is called \emph{cyclic module} if $g(X_1^r)$ and $g(X_{-1}^r)$ are non zero.
This terminology emphasis the fact that $X_+$ and $X_-$ permute the
weight spaces of $V$ in a cyclic way.  A cyclic module has dimension
$r$.  If $g(X_1^r)=g(X_{-1}^r)=0$ then $g\in T$ and $V\in\catU_T$ is a
\emph{highest weight module}.  In this case, $V$ has  a unique (up to a scalar) highest weight
vector $v_0$ such that $k.v_0=q^{\lambda/2} v_0$ ($\lambda\in\C$ is unique
modulo $2r\Z$).  A highest weight module
$V$ has dimension $r$ if and only if $2\lambda\notin
\{0,1,\ldots,r-2\}+r\Z$ (see Proposition \ref{P:typical} below).  This is in particular true if $g\in
T\setminus (\X\cap T)$ that is if $g(k^{4r})\neq1$.

Remark that Theorem \ref{T:catU_gSS} implies that for
$g\in\Gr\setminus\X$, the category $\catU_g$ has $r$ distinct isomorphic classes of
simples modules, all of dimension $r$.  Moreover, the modules in $\catU_g$ are those
isomorphic to a finite direct sum of these $r$-dimensional modules.  For $g\in\X$,
$\catU_g$ is not semi-simple.

For $t\in T$ and $V\in\catU_t$ with highest weight $\kappa\in\C^*$ as
above (i.e. $k.v_0=\kappa v_0$), a lift of $V$ in $\catH$ consists in
choosing a scalar $\lambda\in\C$  by
which $H\in\UqgH$ acts on $v_0$ such that $\kappa=q^{\lambda/2}$.  Finally, a module in $\catH_t$ has the
following alternative description: A module of $\catH$ is in $\catH_t$
if and only if all its additive weights are equal to $\lambda$ modulo
$2\Z$ (i.e. modulo the root lattice).





\subsection{Typical modules}
Let $I^{H}$ be the ideal of $\UqgH$ generated by the central elements $X_{\pm
  \beta}^{\ro}$, $\beta\in\beta_*$.  Let $I=I^{H}\cap\Uqg$ the corresponding
ideal in $\Uqg$.  The ideals $I$ and $I^H$ do not depend on the choice of
$\beta_*$ because they are generated by
$(\ZUo\cap\Uqn{++})\cup(\ZUo\cap\Uqn{--})$.  The set of monomials of the PBW
basis that contain at least one factor $X_{\pm\beta}^{k}$, with $k\geq\ro$, is
a basis of $I$.  Hence $\Uqn{\pm<}$ and $I\cap\Uqn\pm$ are complementary
vector subspaces of $\Uqn\pm$ (note $\Uqn{\pm<}$ is defined in Subsection
\ref{SS:PBW}).  Any nilpotent module is annihilated by $I$.

Let $\kappa\in \Hom_\al(\UK,\C)$ be a multiplicative weight and consider the
one dimensional $\UK\Uqn+$-module $\ang{v_0}$ whose action is determined by
$xv_0=\kappa(x)v_0$ for $x\in \UK$ and $yv_0=\epsilon(y)v_0$ for $y\in \Uqn+$.
This module induces the {\em Verma module} $\wt V_\kappa$ which is an infinite
dimensional weight module over $\Uqg$ with highest weight vector $v_0$ of
multiplicative weight $\kappa$.  The module $\wt V_\kappa$ contains an unique
maximal submodule not containing $v_0$. The quotient of this submodule is an
irreducible $\Uqg$-module $V_\kappa$ with a highest weight vector of weight
$\kappa$.  Similarly, an additive weight $\lambda\in \Hom_\al(\UH,\C)$ defines
an one dimensional module over $\UH\UK\Uqn+$ whose action is given by
$\lambda$ on $\UH$, $q^{\lambda}$ on $\UK$ and $\epsilon$ on $\Uqn+$.  The
same process gives the \emph{Verma module} $\wt V_\lambda$ with irreducible
quotient $V_\lambda$, both having highest weight vectors of weight $\lambda$.

If $\lambda\in \Hom_\al(\UH,\C)$ is an additive weight and
$\kappa=q^{\lambda}\in \Hom_\al(\UK,\C)$ is its associated multiplicative
weight we denote by $|\lambda|=|\kappa|$ the element of
$T\simeq\Hom_\al(\ZUo\cap \UK,\C)$ induced by $\kappa$.

\begin{prop}\label{P:UniqueNilpotent}
  Let $\lambda\in \Hom_\al(\UH,\C)$ be an additive weight and
  $\kappa=q^{\lambda}\in \Hom_\al(\UK,\C)$ its induced multiplicative weight.
  Then $V_\lambda\in Ob(\catH)$ and $V_\kappa\in Ob(\catU_{|\kappa|})$ are the
  unique (up to isomorphism) irreducible
  nilpotent module with highest weights $\lambda$ and $\kappa$, respectively.
  Furthermore, $\FF(V_\lambda)$ is isomorphic to $ V_\kappa$, where $\FF$ is
  the forgetful functor defined in \eqref{E:ForgetfulF}.
\end{prop}
\begin{proof}
  Let $V$ be the module $V_\lambda$ or $V_\kappa$ with height weight vector
  $v_0$.  Since the ideal $I$ annihilates $V$ we have that
  $V=\Uqn-v_0=\Uqn{-<}v_0$.  Therefore, the set $\{X_{-\beta_1}^{k_1}\cdots
  X_{-\beta_N}^{k_N}v_0 : 0\leq k_i<\ro\}$ contains a finite basis of weight
  vectors.
  The uniqueness of $V$ comes from the universal property of its Verma module
  $\wt V$: $\wt V$ can be mapped to any module $W$ with a highest weight
  vector of the same weight as $v_0$.  If $W$ is irreducible then the kernel
  of this map is a maximal proper submodule and so $W$ is isomorphic to $V$.
  
  Similarly, the map $\FF(\wt
  V_\lambda)\stackrel\sim\to\wt V_\kappa$ induces a surjective map
  $\FF(V_\lambda)\to V_\kappa$.  
  If the kernel of this map was nonzero then it would be a nilpotent
  module with a highest weight vector $w$.   But $w$ would also be
  a highest weight vector of $V_\lambda$ generating a proper submodule of
  $V_\lambda$.  This is not possible because $V_\lambda$ is irreducible.
\end{proof}
Proposition \ref{P:UniqueNilpotent} implies that any irreducible nilpotent
module in $\catU$ is isomorphic to $\FF(V_\lambda)$ for some $\lambda$.
However, it is not clear that every nilpotent $\Uqg$-module $V$ is isomorphic
to $\FF(W)$ for some $\UqgH$ weight module $W$.

\begin{prop}\label{P:typical}\cite[Th. 3.2]{DK} 
  Let $\lambda\in \Hom_\al(\UH,\C)$ be an additive weight and
  $\kappa=q^{\lambda}\in \Hom_\al(\UK,\C)$ its induced multiplicative weight.
  Let $v_+$ and $v_-$ be the highest and lowest weight vectors of $V_\kappa$,
  respectively.  Then the following are equivalent.
  \begin{enumerate}
  \item The linear map $\Uqn{-<}\to V_\kappa$ defined by $x\mapsto xv_+$ is
    bijective.
  \item The linear map $\Uqn{+<}\to V_\kappa$ defined by $x\mapsto xv_-$ is
    bijective.
  \item $q^{2\ang{\lambda+\rho,\beta}+m\ang{\beta,\beta}}\neq1$ for all
    $\beta\in\roots^{+} $ and $m\in \{1,\ldots\ro-1\}$.
\end{enumerate}
\end{prop}
\begin{proof}
  In \cite[Section 3.2]{DK} De Concini and Kac define the so called diagonal
  module $V'_\kappa=\wt V_\kappa/I\wt V_\kappa$ with highest weight vector of
  weight $\kappa$.  They prove that the third condition is equivalent to the
  module $V'_\kappa$ being 
  irreducible (remark they have a misprint with the range of $m$ in
  their Theorem 3.2).
Thus, the proposition follows from the following facts: 1) when
$V'_\kappa$ is irreducible it is isomorphic to $V_\kappa$, 2) the maps
in the proposition are surjective, and 3)
$\dim_\C(\Uqn{-<})=\dim_\C(\Uqn{+<})=\dim_\C(V'_\kappa)=\ro^{N}$.
\end{proof}
If $\lambda$ and $\kappa=q^\lambda$ are weights such that one of the
equivalent conditions of Proposition \ref{P:typical} is satisfied then we say
that $V_\lambda$ and $V_\kappa$ are {\em typical} modules.  
When $\g=\sll_2$ then the typical modules are exactly the simple $r$-dimensional $\Uqg(\sll_2)$-weight modules described above.
 The following
corollary combined with Proposition \ref{P:UniqueNilpotent} implies that any
irreducible $\Uqg$-module $V\in \catU_g$ with $g\in T\setminus \X$ is typical.

\begin{corollary}\label{C:GsNil}
  If $g\in T\subset\Gr$ such that
  $g(K_\beta^{\ro})\neq\pm1$ for all  $\beta\in\roots^{+}$, then $g\notin \X$. 
\end{corollary}
\begin{proof}
  There exist (non unique) additive weights $\lambda_i$, for
  $i=1,\ldots,\ro^{\rk}$, with induced multiplicative weights $\kappa_i$ such
  that $|\kappa_i|=g$, for all $i=1,\ldots,\ro^{\rk}$.  Then for all $\beta\in
  \roots^+ $ and $i=1,\ldots,\ro^{\rk}$ we have
  $g(K_\beta^{\ro})=\kappa_i(K_\beta^{\ro})=q^{\ro\ang{\lambda_i,\beta}}$.
  Combining this with the hypothesis $g(K_\beta^{\ro})\neq\pm1$ we have
  $q^{2\ang{\lambda_i+\rho,\beta}+m\ang{\beta,\beta}}\neq1$ for all
  $\beta\in\roots^{+} $, $m\in \{0,\ldots\ro-1\}$ and
  $i\in\{1,\ldots,\ro^{\rk}\}$.  Thus, Proposition \ref{P:typical} implies
  that for each $i=1,\ldots,\ro^{\rk}$ the module $V_{\kappa_i}$ has dimension $\ro^{N}$ and the result follows from
  Theorem~\ref{Th:Ugss}.
\end{proof}

\subsection{The pivotal structure}\label{SS:PivotalSt}
It is well known that $\ang{2\rho,\alpha_i}=\ang{\alpha_i,\alpha_i}=2d_i$
where $\alpha_i$ is a simple positive root.  Then for $\sigma=\pm1$, Equation
\eqref{eq:rel1} implies that $K_{2\rho} X_{\sigma
  i}K_{2\rho}^{-1}=K_{\alpha_i}X_{\sigma i}K_{\alpha_i}^{-1}$.  Combining this
fact with $S^{2}(X_{\sigma i})=K_{\alpha_i}X_{\sigma i}K_{\alpha_i}^{-1}$ we
have
\begin{equation}
  \label{eq:HopfPivotal}
  S^2(x)=K_{2\rho}^{1-{\ro}} x K_{2\rho}^{{\ro}-1}
\end{equation}
for all $x$ in $\UqgH$.  Equation \eqref{eq:HopfPivotal} is a modification (by the central element $K_{2\rho}^{\ro}$) of a similar well known equation.  We impose this modification to make the modified dimension $\qd$ ``spherical'' and the twist invariant by the antipode (see Subsections \ref{SS:braiding} and \ref{SS:modifiedDim}).  

It is a general
fact  that a category of modules over a Hopf
algebra in which the square of the antipode is equal to the conjugation by a
group-like element is a pivotal (or sovereign) category (see \cite[Proposition 2.9]{Bi}).  
Hence, from Equation \eqref{eq:HopfPivotal} it follows that  $\catU$ and $\catH$ are both pivotal $\FK$-linear categories with ground ring $\C$.  In particular, for any
object $V$ in $\catU$, the dual object and the duality morphisms are defined as
follows: $V^* =\Hom_\C(V,\C)$ and
\begin{align}\label{E:DualityForCat}
  b_{V} :\, & \C \rightarrow V\otimes V^{*} \text{ is given by } 1 \mapsto
  \sum v_j\otimes v_j^*,\notag
  \\
  d_{V}:\, & V^*\otimes V\rightarrow \C \text{ is given by } f\otimes w
  \mapsto f(w),\notag
  \\
  b_V':\, & \C \rightarrow V^*\otimes V \text{ is given by } 1 \mapsto \sum
  v_j^*\otimes K_{2\rho}^{{\ro}-1}v_j,
  \\
  d_{V}':\, & V\otimes V^{*}\rightarrow \C \text{ is given by } v\otimes f
  \mapsto f(K_{2\rho}^{1-{\ro}}v),\notag
\end{align}
where $\{v_j\}$ is a basis of $V$ and $\{v_j^*\}$ is the dual basis
of $V^*$.

\subsection{Ambidextrous modules} 
Recall the definition of an ambidextrous object given in
Section~\ref{S:RTLink}.  In this subsection we show that typical modules over
$\UqgH$ and $\Uqg$ are ambidextrous.  With this in mind we give a general
theorem which allows one to prove certain objects are ambidextrous.

We will now assume that $\cat$ is a category of $H$-modules for some Hopf
$\C$-algebra $H$.  If $V$ is any object of $\cat$ we have the bilinear pairing
$\ang{\;,\;}$ on $V\otimes V^*= (V\otimes V^{*})^*$ given by $
\ang{w,w'}=d_{V\otimes V^*}(w\otimes w').$ Then for any $f\in\End(V\otimes
V^*)$ we have $ \ang{f(w),w'}= \ang{w,f^*(w')}$.
\begin{theorem}\label{T:dualvectorambi}
  Let $V$ be an irreducible $H$-module.  Assume there exists $w_0,w'_0\in
  V\otimes V^*$ and $x,y\in H$ such that $\ang{w_0,w'_0}=1$ and $x.w_0$,
  $y.w'_0$ are nonzero $H$-invariant vectors (i.e. any $h\in H$ acts on
  $x.w_0$ and $y.w'_0$ by the scalar $\varepsilon(h)$).  Then the module $V$
  is ambidextrous if either of the following two conditions hold:
  \begin{enumerate}
  \item[(a)] $\C w_0$ and $\C w'_0$ are invariant lines under the
    action of $\End_\cat(V\otimes V^*)$,
  \item[(b)] $\ker x\subset\ker\ang{\cdot, w'_0}$ and $\ker
    y\subset\ker\ang{w_0,\cdot}$, here $x,y$ are operators on $ V\otimes V^*$.
  \end{enumerate}
\end{theorem}
\begin{proof}
  Let $f\in \End_\cat(V\otimes V^*)$.  Since $V$ is simple then any invariant
  vector of $V\otimes V^{*}$ is proportional to $ b_V(1)$ because $b_V$
  generates the one dimensional vector space $\Hom_\cat(\C,V\otimes
  V^*)\simeq\Hom_\cat(V,V)$.  So up to rescaling $x$ and $y$, we may assume
  that $x.w_0=y.w'_0=b_V(1)$.  Also, since $f(x.w_0)$ and $f(y.w'_0)$ are
  invariant vectors there exists $\alpha, \alpha' \in\C$ with $f(x.w_0)=\alpha
  x.w_0$ and $f(y.w'_0)=\alpha' y.w'_0$.  Therefore, there exists $(v,v')\in
  \ker x\times \ker y$ such that $f(w_0)=\alpha w_0+v$ and
  $f^{*}(w'_0)=\alpha'w'_0+v'$.  Also, $f\circ b_V(1)=f(x.w_0)=\alpha
  x.w_0=\alpha b_V(1)$ and $f^{*}\circ b_V(1)=\alpha' b_V(1)$.

  If Condition (a) holds we can assume $v=v'=$0.  Then
  $\alpha=\ang{f(w_0),w'_0}=\ang{w_0,f^{*}(w'_0)}=\alpha'$.  On the other
  hand, if Condition (b) is satisfied, $\ang{v,w'_0}=\ang{w_0,v}=0$ and so
  $\alpha=\ang{f(w_0),w'_0}=\ang{w_0,f^{*}(w'_0)}=\alpha'$.  Thus, in both
  case we have $f^{*}\circ b_V=f\circ b_V$.
\end{proof}

\begin{lemma}\label{L:hwPBW}
  Consider the two following elements of $\Uqg$:
  $$
  X_-=\prod_{\alpha\in \beta_*}X_{-\alpha}^{\ro-1}\quad\text{ and }\quad
  X_+=\prod_{\alpha\in \beta_*}X_\alpha^{\ro-1}.
  $$ 
If $\gamma_*$ is any convex order of $\roots^+$  and  $\{X_\alpha\}_{\alpha\in\gamma_*}$ is the corresponding convex set of root
  vectors then there exists $a,b\in\C^{*}$ such that
  $$
  X_-\equiv a \prod_{\alpha\in \gamma_*}X_{-\alpha}^{\ro-1}\mod I\quad\text{
    and }\quad X_+\equiv b \prod_{\alpha\in \gamma_*}X_\alpha^{\ro-1}\mod I
  $$
\end{lemma}
\begin{proof}
  This is a consequence of (PBW) Theorem \ref{T:PBW}.  Indeed, the weight space of
  $\Uqn+/I$ with weight $2(\ro-1)\rho$ is one dimensional, generated by
  $\prod_{\alpha\in \gamma_*}X_\alpha^{\ro-1}$ for any convex order
  $\gamma_*$ of $\roots^+$.  This implies the equivalence for $X_-$,
   the equivalence for $X_+$ is  obtained analogously with $\Uqn-/I$.
\end{proof}

\begin{theorem}\label{T:typAmbi}
  Any typical module $V$ over $\UqgH$ or $\Uqg$ is ambidextrous.
\end{theorem}
\begin{proof}
  Let $v_+,v_+'$ be (the unique up to a scalar) highest weight vectors of $V,
  V^{*}$, respectively.  Choose lowest weight vectors $v_-,v_-'$ of $V, V^{*}$
  so that $\ev_V(v_+'\otimes v_-)=\tev_V(v_+\otimes v_-')=1$ (i.e. choose
  scalars of the vectors).  We want to apply Theorem~\ref{T:dualvectorambi}
  where $w_0=v_+\otimes v_+', w_0'=v_-\otimes v_-'$ and $x,y$ are any elements
  of $\Uqg$ proportional to $\prod_{\alpha\in \gamma_*}X_{-\alpha}^{\ro-1},
  \prod_{\alpha\in \gamma_*}X_\alpha^{\ro-1}$, respectively, for any convex
  ordering $\gamma_*$ of $\roots^{+}$ and corresponding convex set $\{X_{\pm
    \gamma}\}_{\gamma\in \gamma^*}$.  First, we will prove that $x.w_0$ is a
  non-zero invariant vector, the proof for $y.w_0'$ is analogous.
     
  Note that $x.v_+\in\C^{*}v_-$ and the weight of $v_+'$ is the opposite of
  the weight of $v_-$.  Hence, the weight of $w_0$ is $2(\ro-1)\rho$ and
  weight of $x.w_0$ is $0$.  We write $x.w_0=\sum v_i\otimes v'_i$ where
  $v_i,v_i'$ are weight vectors.  The term of $x.w_0$ for which $v_i$ has
  smallest weight is equal to $(x\otimes1). (v_+\otimes v_+')=(x.v_+)\otimes
  v_+'\neq 0$.  Moreover, if $j\neq i$ then the weight of $v_j$ is strictly
  larger than the weight of $v_i$.  Thus, $x.w_0\neq0$.

  Next we show $x.w_0$ is invariant.  It suffices to show that for any
  $i=1,\ldots,\rk$ we have $X_{i}.(x.w_0)=0$, $X_{-i}.(x.w_0)=0$ and $
  H_i.(x.w_0)=0$.  First, the last equality is true because as mentioned above
  $x.w_0$ has weight $0$.  Second, applying Lemma~\ref{L:hwPBW} with a convex
  ordering $\gamma_*$ of $\roots^{+}$ such that $\gamma_1=\alpha_i$, we can
  write $x\equiv X_{-i}^{\ro-1}x' \mod I$ for some $x'\in \Uqn-$ (note by
  definition $X_{\gamma_1}=X_{\alpha_i}=X_i$).  Then $X_{-i}.(x.w_0)= 0$
  because $X_{-i} x\equiv X_{-i}^{\ro}x' \equiv 0 \mod I$.  Finally, notice
  that the weights of the vectors of $\Uqg.w_0=\Uqn-.w_0$ are of the form
  $\sum_{j=1}^{N}a_j\gamma_j$ for $0\leq a_j<\ro$.  But
  $X_{\alpha_i}.(x'.w_0)$ has weight $\rk\alpha_i$ so $X_{i}.(x'.w_0)=0$.
  Combining this with the fact that
  $[X_{i},X_{-i}^{\ro-1}]=\qN{\ro-1}_{q_i}\qN{\ro-2+H_i}_{q_i}
  X_{-i}^{\ro-2}$, 
  we have
   \begin{equation*}\label{E:XiAct}
   X_i.(x.w_0)= 
  \qN{\ro-1}_{q_i}\qN{\ro-2+H_i}_{q_i}
  X_{-i}^{\ro-2}x'.w_0=\qN{\ro-1}_{q_i}\qN{\ro}_{q_i}X_{-i}^{\ro-2}x'.w_0=0
  \end{equation*}
  since $H_i$ acts on $X_{-i}^{\ro-2}x'.w_0$ as $2$ and $\qN{\ro}_{q_i}=0$.
  This prove that $x.w_0$ is invariant.
  
  Now we can apply Theorem \ref{T:dualvectorambi}.  Indeed, if $V$ is an
  $\UqgH$-module, then Condition (a) is satisfied: $w_0$ and $w'_0$ are fixed
  up to a scalar by any endomorphism of $V\otimes V^{*}$ because they are
  weight vectors in a one dimensional weight space.  If $V$ is a $\Uqg$-module
  we can write $V\simeq \FF(W)$ for some $W\in Ob(\catH)$.  The vector space
  $W\otimes W^{*}=V\otimes V^{*}$ is equipped with a bilinear form and the
  weight decomposition of $W\otimes W^{*}$ gives an orthogonal decomposition:
  $V\otimes V^{*}=(\C w_0\oplus\C w'_0)\oplus V'$ for some $V'$ with
  $\ang{V',w_0}=\ang{V',w_0}=\ang{w_0,V'}=\ang{w_0', V'}=0$.  Now the weights
  in $W\otimes W^{*}$ imply that $\ker x\subset V'\oplus\C
  w'_0=\ker\ang{\cdot, w'_0}$ and $\ker y\subset V'\oplus\C
  w_0=\ker\ang{w_0,\cdot}$.  Hence Condition (b) of Theorem
  \ref{T:dualvectorambi} is satisfied.
\end{proof}

\subsection{The braiding}\label{SS:braiding}
We recall some well known facts about the $h-adic$ version of the quantum
group $\Uhg$.  The algebra $\Uhg$ is a $\C[[h]]$ topological Hopf algebra with
generators $X_i,X_{-i},H_i$ for $i=1,\ldots,\rk$ the Relation \eqref{eq:rel2},
\eqref{eq:rel3} and \eqref{eq:relH} where $q$ is replaced by $\e^{h/2}$ and
$K_{\alpha_i}=q_i^{H_i}$.  For a root $\beta\in L_R$, let
$q_\beta=q^{\ang{\beta,\beta}/2}$.  Let $\exp_q(x)=\sum_{i=0}^\infty
\frac{x^i}{[i;q]!}$ where $[i;q]=\frac{1-q^{i}}{1-q}$ and
$[i;q]!=[i;q]\cdots[1;q]$.  Consider the following elements of $\Uhg$:
\begin{equation*}
  {\Hh}=q^{\sum_{i,j}d_i(A^{-1})_{ij}H_i\otimes H_j}, \;\;
  \check{R}^h=\prod_{\beta\in 
    \beta_*}\exp_{q_\beta^{-2}}\Big((q_\beta-q_\beta^{-1})X_\beta\otimes
  X_{-\beta}\Big) 
\end{equation*}
where the product is ordered by the convex order $\beta_*$ of $\roots^+$.  It
is well known that $R^{h}={\Hh}\check{R}^h$ defines a quasi-triangular
structure on $\Uhg$ (see for example \cite{Ma}).  This mean that $R^{h}$ is
invertible and satisfy:
\begin{equation}\label{E:Rh}
  \Delta\otimes\Id(R^{h})=R^{h}_{13}R^{h}_{23},
  \;\; \Id\otimes\Delta(R^{h})=R^{h}_{13}R^{h}_{12}, \;\;
  R^{h}\Delta^{op}(x)=\Delta(x)R^{h} 
\end{equation}
for all $x\in\Uhg$.  The algebra $\Uhg$ admit a PBW basis.  Using this basis
we can write $\Uhg$ as a direct sum decomposition: $\Uhg=U^{<}\oplus I$ where
$U^{<}$ is the $\C[[h]]$-module generated by the monomials
$\prod_{i=1}^{\rk}H_i^{k_i} \prod_{\beta_i \in \beta_*}X_{\beta_i}^{l_i}
\prod_{\beta_i \in\beta_*}X_{-\beta_i}^{m_i}$ for $0\leq l_i, m_i< \ro$ and
$I$ is generated by the other monomial.\footnote{The set $I$ is not an ideal
  when $q$ is not a $\ro$\textsuperscript{th} root of unity and $U^{<}$
  depends of $(X_{\beta_*})$.}  Let $p:\Uhg\to U^{<}$ be the projection map
with kernel $I$.
We define
\begin{equation}
  \label{eq:R<}
  R^{<}=p\otimes p(R^{h})=p\otimes \Id(R^{h})=\Id\otimes p(R^{h}).
\end{equation}
\begin{prop}\label{P:R<} $R^{<}$ satisfy:
  \begin{enumerate}
  \item $(p\otimes p\otimes p)(\Delta\otimes\Id(R^{<}))=(p\otimes p\otimes
    p)(R^{<}_{13}R^{<}_{23})$,
  \item $(p\otimes p\otimes p)(\Id\otimes\Delta(R^{<}))=(p\otimes p\otimes
    p)(R^{<}_{13}R^{<}_{12})$,
  \item $(p\otimes p)(R^{<}\Delta^{op}(x))=(p\otimes
    p)(\Delta(x)R^{<})$ for all $x\in\Uhg$.
  \end{enumerate}
\end{prop}
\begin{proof}
  These equalities are obtained by projecting the equalities in \eqref{E:Rh}.
   Indeed, using $p\circ p=p$ and Equation \eqref{eq:R<} we have
  $(p\otimes p\otimes p)(\Delta\otimes\Id(R^{h}))=(p\otimes p\otimes
  p)\circ(\Delta\otimes\Id)\circ\Id\otimes p(R^{h}))=(p\otimes p\otimes
  p)(\Delta\otimes\Id(R^{<}))$.  On the other hand, $(p\otimes p\otimes
  p)(R^{h}_{13}R^{h}_{23})=(p\otimes p\otimes p)( (p\otimes \Id\otimes
  \Id)(R^{h}_{13}) (\Id\otimes p\otimes \Id)(R^{h}_{23}))=(p\otimes p\otimes
  p)(R^{<}_{13}R^{<}_{23})$.  This prove the first equality and the second is
  similar. 
  
  To prove the third equality, it suffices to assume $x$ is a generator of the
  algebra $\Uhg$.  The equality is clear for $x=H_i$ as $\Delta(H_i)$ is
  symmetric and commutes with any of the products $X_\alpha\otimes
  X_{-\alpha}$.  Now for $x=X_i$, we have $\Delta(X_i)=1\otimes X_i+X_i\otimes
  K_{\alpha_i}$.  Hence, $(p\otimes p)(R^{h}(K_{\alpha_i}\otimes X_i+X_i\otimes
  1))=(p\otimes p)((p\otimes\Id)(R^{h}K_{\alpha_i}\otimes X_i)+(\Id\otimes
  p)(R^{h}X_i\otimes 1))=(p\otimes p)(R^{<}(K_{\alpha_i}\otimes X_i+X_i\otimes
  1))$.  On the other hand, $(p\otimes p)((1\otimes X_i+X_i\otimes
  K_{\alpha_i})R^{h})=(p\otimes p)((p\otimes\Id)(1\otimes
  X_iR^{h})+(\Id\otimes p)(X_i\otimes K_{\alpha_i}R^{h}))=(p\otimes
  p)((1\otimes X_i+X_i\otimes K_{\alpha_i})R^{<})$.  The proof is similar for
  $x=X_{-i}$.
\end{proof}

Let $V=V_1\otimes\cdots\otimes V_k$ where $V_i\in Ob(\catH)$, for
$i=1,\ldots,k$.  Then each $x\in(\UqgH)^{\otimes k}$ defines a linear map
$\psi(x):V\to V$ and we have a finite dimensional representation
$\psi:(\UqgH)^{\otimes k}\to\End_\C(V)$.  For any $x\in (\UqgH)^{\otimes k}$,
let $\psi(q^{x})$ be the limit as $m$ goes to infinity of the absolutely
convergent series $\sum_{j=0}^m\frac1{j!}\psi\left(\frac{2i\pi
    x}{\ro}\right)^{j}$.
In this situation, we say that $q^{x}$ is an \emph{operator on $\catH$}.  If
$q^x,q^y$ are two operators such that $\psi(q^x)=\psi(q^y)$ for all finite
dimensional representations $\psi$ then we say the operators are equal and
write $q^x=q^y$.  One can multiply operator and take their coproduct and
antipode using the rules $\Delta q^{x}=q^{\Delta x}$ and $S(q^x)=q^{S(x)}$.
We will also call the linear map $\psi(x):V\to V$ an \emph{operator} and
denote it by $x$.

Let $\RH$ be the operator on $\catH$ defined as
$$\RH=q^{\sum_{i,j}d_i(A^{-1})_{ij}H_i\otimes H_j}.$$

\begin{lemma}\label{L:OpH}
  For any $x,y\in\UqgH$ with weights $\alpha,\beta\in L_W$, respectively, we
  have the equalities of operators on $\catH$:
\begin{equation}
    \label{eq:RHcom}
   \RH (x\otimes y) = q^{\ang{\alpha,\beta}}(xK_\beta\otimes yK_\alpha)\RH.
  \end{equation}
  Also, $\Delta\otimes\Id(\RH)=\RH_{13}\RH_{23}$ and
  $\Id\otimes\Delta(\RH)=\RH_{13}\RH_{12}$.
\end{lemma}
\begin{proof}
  The proof is a direct computation using the following three facts: 1) 
  $H_ix=x(H_i+\frac1{d_i}\ang{\alpha_i,\alpha})$, 2) as operators
  $K_{\alpha_i}=q_i^{H_i}$, and 3) 
  \begin{equation}
    \label{eq:RHact}
    \RH({v_\lambda\otimes w_\mu})=q^{\ang{\lambda,\mu}}{v_\lambda\otimes w_\mu}
  \end{equation}
for any weight vectors $v_\lambda$ and $w_\mu$  of weight
  $\lambda$ and $ \mu$, respectively.  
\end{proof}

Let $\qR$ be the truncated quasi $R$-matrix of $\UqgH$ given by
$$\qR=\prod_{i=1}^{N}\left(
  \sum_{j=0}^{\ro-1}\frac{\left((q_{\beta_i}-q_{\beta_i}^{-1})X_{\beta_i}\otimes
      X_{-\beta_i}\right)^j }{[j;q_{\beta_i}^{-2}]!}
\right)\in\Uqn+\otimes\Uqn-$$  
where $\beta_*=(\beta_1,\ldots,\beta_N)$ is a convex set of $\roots^+$.

\begin{theorem}\label{T:BraidingExists}
  The operator $\oR=\RH\qR$ leads to a braiding $\{c_{V,W}\}$ in $\catH$ where
  $c_{V,W}:V\otimes W\to W\otimes V$ is given by $v\otimes w\mapsto
  \tau(\oR(v\otimes w))$.  Here $\tau:V\otimes W\to W\otimes V$ is the trivial
  isomorphism of vector spaces given by permutation.
\end{theorem}
\begin{proof}
 It is enough to show that the operator $\oR$ satisfies 
  \begin{equation}\label{E:ReloR}
    \Delta\otimes\Id(\oR)=\oR_{13}\oR_{23},\;\;
    \Id\otimes\Delta(\oR)=\oR_{13}\oR_{12},\;\;
    \oR\Delta^{op}(x)=\Delta(x)\oR
  \end{equation}
  for all $x\in\UqgH$.  Let $\chi: \Uhg\otimes \Uhg\to \Uhg\otimes \Uhg$ be
  the map given by $x\otimes y\mapsto {\Hh}(x\otimes y)({\Hh})^{-1}$.  Viewing
  $\check R$ as element of $\Uhg$ we have that Proposition \ref{P:R<} induces
  the relations
\begin{align}\label{E:RelcheckR}
 \Delta\otimes\Id(\check{R})=\chi_{23}(\check{R}_{13})\check{R}_{23},\;\;
  \Id\otimes\Delta(\check{R})=\chi_{12}(\check{R}_{13})\check{R}_{12}, \;\;
   \check{R}\Delta^{op}(x)=\chi(\Delta(x))\check{R}
\end{align}
for all $x\in \Uhg$.  Let us prove the first equality in Equation
\eqref{E:RelcheckR}, the other two are similar.  The first equation of Proposition
\ref{P:R<} implies that $
\Delta\otimes\Id({\Hh}\check{R})={\Hh}_{13}\check{R}_{13}{\Hh}_{23}\check{R}_{23}$.
The left hand side of this equality is equal to $ \Delta\otimes\Id({\Hh})
\Delta\otimes\Id(\check{R})={\Hh}_{13}{\Hh}_{23}\Delta\otimes\Id(\check{R})$.
On the other hand,
${\Hh}_{13}\check{R}_{13}{\Hh}_{23}\check{R}_{23}={\Hh}_{13}{\Hh}_{23}\chi_{23}(\check{R}_{13})\check{R}_{23}$.
Now since ${\Hh}$ is invertible we have the desired result.
 
The element $R^{<}$ does not have a pole at $q$ when $q$ is a primitive root
of unity of order $\ro$.  Thus, the relations of Equation \eqref{E:RelcheckR}
hold when $q$ is a root of unity.  Finally, Lemma \ref{L:OpH} implies that the
operators ${\Hh}$ and $\RH$ satisfy the same commutator relations on $\Uhg$
and $\UqgH$, respectively.  Thus, the relations of Equation \eqref{E:ReloR}
hold.
\end{proof}
Consider the operator $\uH=q^{\sum_{i,j}-d_i(A^{-1})_{ij}H_iH_j}$.
Then $S(\uH)=\uH$ and if $v$ is a weight vector in a $\UqgH$-module of
weight $\lambda$ then $\uH v=q^{-\ang{\lambda,\lambda}}v$.  Also, if
$x\in\UqgH$ has weight $\alpha$, then $ \uH
x=q^{-\ang{\alpha,\alpha}}xK_{-\alpha}^{2}\uH$ and conjugation by
$\uH^{-1}$ induces a well defined automorphism
$\operatorname{Ad}_{\uH^{-1}}$ of $\UH$.  Let
$u=\mu\circ((\operatorname{Ad}_{\uH^{-1}}\circ
S)\otimes\Id)(\check{R}_{21}) \in \Uqn-\Uqn+$ where $\mu$ is the
multiplication.  In particular, since $u$ has weight $0$, the
operators $\uH$ and $\qu$ commute.

Let $\ou$ and $\theta$ be the operators $\ou=\uH\qu$ and $\theta=\ou
K_{2\rho}^{\ro-1}$.  Then
 $$ \theta S(\theta)^{-1} = \ou
 K_{2\rho}^{\ro-1}S(\ou)^{-1}K_{2\rho}^{\ro-1} = \qu
 S(\qu)^{-1}K_{2\rho}^{2(\ro-1)}.$$ We define $\varpi=\qu
 S(\qu)^{-1}K_{2\rho}^{2(\ro-1)}\in \UqgH$.
\begin{prop}\label{P:typPiId}
  Let $V_\lambda$ be a typical $\UqgH$-module then $\varpi:V_\lambda\to
  V_\lambda$ is the identity.
\end{prop}
\begin{proof}
  Let $v_+$ and $v_-$ be the highest and lowest weight vectors of $V_\lambda$
  of weights $\lambda$ and $\lambda-2(\ro-1)\rho$, respectively.  Then
  $\qu-1\in \UK\Uqn{--}\Uqn{++}$ implies that $\qu$ acts by $1$ on $v_+$.
  Similarly, $S(\qu)$ acts by $1$ on $v_-$.  In particular,
  $\theta(v_+)=q^{-\ang{\lambda,\lambda}}q^{\ang{2\rho(\ro-1),\lambda}}v_+$
  and
  $$(S(\theta))(v_-)=q^{-\ang{2(\ro-1)\rho,\lambda-2(\ro-1)\rho}}
  q^{-\ang{\lambda-2(\ro-1)\rho,\lambda-2(\ro-1)\rho}}v_-
  =q^{-\ang{\lambda,\lambda}}q^{\ang{2\rho(\ro-1),\lambda}}v_-.$$ But $\theta$
  and $S(\theta)$ are central and act as scalars on $V_\lambda$.  Thus,
  \begin{equation}
    \label{eq:twist}
    \theta_{|V_\lambda}=S(\theta)_{|V_\lambda} 
    =q^{-\ang{\lambda,\lambda-2\rho(\ro-1)}}\Id.
  \end{equation}
\end{proof}

\begin{theorem}\label{T:UHribboncat}
  Let $\catHr$ be the full subcategory of $\catH$ formed by modules on which
  $\varpi$ acts as the identity.  Then $\catHr$ is a 
  $\C$-linear ribbon category  
  with braiding $c_{V,W}:V\otimes W\to W\otimes V, \; v\otimes w\mapsto
  \tau(\oR(v\otimes w))$ and twist $\theta_V:V\to V,\; v\mapsto \theta^{-1}(v)$.
\end{theorem}
\begin{proof}
  Recall that $I^{H}$ is the ideal of $\UqgH$ generated by the central
  elements $X_{\pm\beta}^{\ro}$, $\beta\in\beta_*$.  Remark that $\varpi$ is
  central and group-like in the quotient Hopf algebra $\UqgH/I^{H}$.  Hence
  $\catHr$ is a tensor subcategory of $\catH$.  The braiding of $\catH$ given
  Theorem \ref{T:BraidingExists} restricts to a braiding on $\catHr$.
  Finally, since $\varpi$ acts as the identity on any object in $\catHr$ we
  have that operators $\theta$ and $S(\theta)$ are equal.  It follows that
  $\theta_V$ is a twist on $\catHr$.
  \end{proof}
Remark that from Proposition \ref{P:typPiId} the category $\catHr$ contains
all the typical modules of $\catH$.

\begin{conjecture}
  $\varpi=1$ and $\catH=\catHr$ is a ribbon category.
\end{conjecture}

\subsection{The modified dimension}\label{SS:modifiedDim}
As irreducible projective modules of $\catU$ and $\catH$ are ambidextrous, there
is an unique non trivial trace on the ideal of projective objects 
(cf \cite{GPV}).  On simple objects, the trace is given by the modified
dimension which we calculate for the special case of nilpotent modules.  Then
we show that the modified dimension is the same for an object and its dual,
thus leading to a relative $G$-spherical structure in $\catU$.

\begin{prop} \label{P:S'} Let $\wta,\wtb\in \Hom_\al(\UH,\C)$ be additive weights 
where $V_\wta$ is typical. Then
  $$S'(\wta,\wtb)=\ang{\ 
    \epsh{fig2}{8ex}\put(-30,-5){$V_\wta$}\put(-8,-12){$V_\wtb$}\ } =
  q^{2\ang{\wtb+(1-\ro)\rho,\wta}}\prod_{\alpha\in\roots^{-}}
  \dfrac{q^{2\ro\ang{\wtb+(1-\ro)\rho,\alpha}}-1}{q^{2\ang{\wtb+(1-\ro)\rho,\alpha}}-1}\in \C$$
  where $\roots^-=-\roots^+$ is the set of negative roots and as above $\ang{f}\in \C$ of a morphism $f:V_\wta\to V_\wta$ is the scalar defined by $f=\ang{f}\Id_{V_\wta}$. 
\end{prop}
\begin{proof}

  Let $\{v_{i}\}$ be a basis of $V_\wta$ such that $v_{i}$ is a weight vector
  with additive weight $\wta_{i}$.  Let $w_{\wtb}$ be a highest weight vector of ${V}_\wtb$.  From Equation \eqref{eq:RHact} we have
  \begin{align}
    \label{E:K1}
    \RH(w_{\wtb}\otimes v_{i}) =q^{<\wtb,\wta_{i}>}w_{\wtb}\otimes
    v_{i}&&\text{ and }&& \RH(v_{i}\otimes w_{\wtb})
    =q^{<\wtb,\wta_{i}>}v_{i}\otimes w_{\wtb}.
 \end{align}

 We now give two facts.  Let $v$ be any weight vector of ${V}_\wta$ of
 weight~$\eta$.
 \begin{description}
 \item[Fact 1] $\oR(w_{\wtb}\otimes v)=q^{<\wtb,\eta>}(w_{\wtb}\otimes
   v)$.\\
   \indent This fact follows from the definition of $\check R$, Equation \eqref{E:K1} and
   the property that $E_{\alpha}w_{\wtb}=0$ for $\alpha\in \roots^+$.
 \item[Fact 2] All the pure tensors of the element $(\check{R}-1)(v\otimes
   v_{\wtb})\in V_\wta\otimes V_\wtb$ are of the form $v'
   \otimes w'$ where $w'$ is a weight vector of $V_\wtb$ and $v'$ is a
   weight vector of $V_\wta$ whose weight is of strictly higher
   order than that of the weight of $v$.\\
  \indent Fact 2 is true because $E_{\alpha}^{n}v$ (for $\alpha\in \roots^+$ and
   $n\in \N_{>0}$) is zero or a weight vector whose weight is of strictly higher
   order than the weight of $v$.
 \end{description}

 We will now compute $S'(\wta,\wtb)$ directly.  Let $V$ be a typical module
 and recall that the duality morphisms $\tev_{V}: V\otimes V^{*}\rightarrow
 \C$ is defined as
 \begin{equation}
   \label{E:K2}
   \tev_{V_\wta} (v\otimes f)= f(q^{2(1-\ro)<\eta,\rho>}v)
 \end{equation} where $v$ is a weight vector of
 $V_\wta$ of additive weight $\eta$.  Consider the element
 $S\in \End_{\C}(V_\wtb)$ 
 given by
 $$(\Id_{V_\wtb}\otimes
 \tev_{V_\wta})\circ(c_{V_\wta,V_\wtb}\otimes
 \Id_{V_\wta^{*}})\circ ( c_{V_\wtb,V_\wta}\otimes
 \Id_{V_\wta^{*}})\circ(\Id_{V_\wtb}\otimes
  b_{V_\wta}).$$ 
 To simplify notation set $S=(X_{1})(X_{2})(X_{3})(X_{4})$ where $X_{i}$ is
 the corresponding morphism in the above formula.  The morphism $S$ is
 determined by its value on the highest weight $w_{\wtb}$.  By definition
 $S(w_{\wtb})=S'(\wta,\wtb)w_{\wtb}$, so it suffices to compute $S(w_{\wtb})$.
 \begin{align}
   \label{E:valueS}
   S(w_{\wtb})&=(X_{1})(X_{2})(X_{3})\Big(w_{\wtb}\otimes
   \sum_{i}(v_{i}\otimes v_{i}^{*})\Big)\notag\\
   &= (X_{1})(X_{2})\Big( \sum_{i} q^{<\wtb,\wta_{i}>}
   v_{i}\otimes w_{\wtb}\otimes v_{i}^{*}\Big)\notag \\
   &= (X_{1})\Big( \sum_{i} \big( q^{2<\wtb,\wta_{i}>} w_{\wtb}\otimes
   v_{i}\otimes v_{i}^{*}\big) + \sum_{k} w_{k}' \otimes
   v_{k}' \otimes z_{k} \Big) \notag\\
   &=\sum_{i} q^{2\ang{\wtb +(1-\ro)\rho, \wta_{i}}}w_{\wtb}
 \end{align}
 where $z_{k}= v_{i}^{*}$ (for some $i$), $v'_{k}$ is a weight vector of
 $V_\wta$ whose weight is of strictly higher weight than that of the
 weight of $z_{k}^{*}$ and $w'_{k}$ is a weight vector of $V_\wtb$.
 Moreover, the second equality of the above equation follows from Fact 1, the
 third from \eqref{E:K1} and Fact 2, and finally the fourth from Equation
 \eqref{E:K2} and the fact that $z_{k}(v'_{k})=0$.  The key observation in
 this proof is that Facts 1 and 2 imply that in the above computation the only
 contribution of the action of the operator $\oR$ comes from $\RH$.

 Finally, $V_\lambda$ is typical and so its weights are determined the formal
 sum (i.e. its character):
 $$
 \sum_i \e^{\wta_{i}}= \e^{\wta}\prod_{\alpha\in \roots^{-}}
 \dfrac{\e^{\ro\alpha}-1}{\e^{\alpha}-1}.
 $$
Thus, the proposition follows from Equation~\eqref{E:valueS}.
\end{proof}

We now can apply the results of Section \ref{S:RTLink} to $\catHr$.   From Theorem \ref{T:typAmbi} we know that any typical $\UqgH$-module $J$ is ambidextrous.  
Therefore, we can apply Theorem \ref{T:RTLink} to $\catHr$ and obtain an ambi pair $(\At,\qd)$.  As discussed in Section \ref{S:RTLink} the functions $\qd$ and $\qd_J$ are proportional.  
Now, let
$\lambda\in \Hom_\al(\UH,\C)$ be an additive weight.  Combining the formula in
Proposition \ref{P:S'} with the formula
  $$\prod_{\alpha\in\roots^{-}}(q^{2\ang{\lambda +(1-\ro)\rho, \alpha}} -1) =
  q^{\ang{\lambda +(1-\ro)\rho,-2\rho}}
  \prod_{\alpha\in\roots^{+}}\qn{\ang{\lambda +(1-\ro)\rho, \alpha}}$$ we can
see that the function $\qd$ can be multiplied by a scalar in $\C^\times$ so that it is 
given by the formula
  \begin{equation}\label{E:FormMD}
  \qd(V_\lambda)=\prod_{\alpha\in\roots^{+}}
  \dfrac{\qn{{\ang{\lambda-(\ro-1)\rho,\alpha}}}}
  {\qn{\ro\ang{\lambda-(\ro-1)\rho,\alpha}}}
\end{equation}
 for  $V_\lambda$ with  $|\lambda|\notin\X$. 
Thus we have proved:
 \begin{theorem}\label{T:catHrInv}
 The category $\catHr$ is a ribbon category with a projective ambidextrous object and so gives rise to an ambi pair $(\At,\qd)$ and isotopy invariant $\Gfun'$ where $\qd$ satisfies Equation \eqref{E:FormMD}.
 \end{theorem}

It is clear from Formula \eqref{E:FormMD} that $\qd(V_\lambda)$ depends only on the
multiplicative weight $\kappa=q^{\lambda}\in \Hom_\al(\UK,\C)$ induced by
$\lambda$. Thus we can define, $\qd(V_\kappa)=\qd(V_\lambda)$ whenever
$|\kappa|\notin\X$.

Let $\A$ be the set of irreducible $\Uqg$-modules $V$ such that $V\in \catU_g$
for some $ g\in \Gr\setminus\X$.  Theorem~\ref{T:catU_gSS} implies that each
module in $\A$ is projective.  
Let $\Proj$ be the full subcategory of $\catU$ consisting of projective
$\Uqg$-modules. 
Recall the notion of a trace given in Subsection \ref{SS:ambi}.
\begin{lemma}\label{L:traceProjD}
  There exist a trace $\mt = \{\mt_{V}\}$ on $\Proj$ such that
  $\mt_{V_\kappa}(\Id_{V_\kappa})=\qd(V_\kappa)$ for any $V_\kappa$ of
  multiplicative weight $\kappa=q^{\lambda}$ with $|\lambda|\notin\X$.
\end{lemma}
\begin{proof}
  Let $U\in \catU_g$ be a typical projective $\Uqg$-module with $g\notin \X$.
  Consider the linear map $t: \End_{\cat}(U) \to \C$ given by $f\mapsto
  \qd(U)\ang{f}$ where, as above, $\ang{f}$ is defined by $f=\ang{f}\Id_U$.
  From Theorem~\ref{T:typAmbi} we have that $U$ is ambidextrous.  Then from
  \cite{GPV} we have that $t$ determines a unique trace $\mt = \{\mt_{V}\}$ on
  $\Proj$ such that $t=\mt_V$.
 
  On the other hand, $\Ffun'$ defines a trace $\mt' = \{\mt'_{V}\}$ on the
  projective modules of $\catH$ given by $\mt'_{V}(f)=\Ffun'(L_f)$ where $L_f$
  is ribbon graph which is the closure of coupon of $f$.  Then by definition,
  for any additive weight $\lambda$ with $|\lambda|\notin\X$, we have
  $\mt'_{V_\lambda}(\Id_{V_\lambda})=\Ffun'(O_{V_\lambda})=\qd(V_\lambda)$
  where $O_{V_\lambda}$ is the unknot colored with $V_\lambda$.

  By construction of $\mt$, for each $\kappa=q^\lambda$ with
  $|\lambda|\notin\X$, we have $\mt_{V_{\kappa}}=\mt_{\FF(V_\lambda)}$ where
  $\FF:\catH\to \catU$ is the forgetful functor defined above.  Moreover,
  $\mt_{\FF(V_\lambda)}(\FF(f))=\mt'_{V_\lambda}(f)$ for all
  $f\in \End(V_\lambda)$.  Thus,
  $\mt_{V_{\kappa}}(\Id_{V_{\kappa}})=\mt'_{V_\lambda}(\Id_{V_\lambda})
  =\qd(V_\lambda)=\qd(V_\kappa)$.
\end{proof}

Lemma \ref{L:traceProjD} implies the assignment $\qd$ on nilpotent $\Uqg$-modules can be extended to a function
$\qd:\A\to \C$ given by $\qd(V)=\mt_V(\Id_V)$.  Since $V\in \A$ is projective it follows from \cite{GPV}  that $\qd(V)\neq 0$. 
  Let $s:\A\to \C$ be the {\em slope} given
by $s(V)=\qd(V)/\qd(V^{*})$.  It satisfies
\begin{equation}\label{E:slope} 
s(W)=s(U)s(V)
\end{equation}
whenever $W$ is a direct summand of $U\otimes V$ (see \cite{GPV}).

\begin{lemma}\label{L:slope1}
  For any $V\in\A$, we
  have $s(V)=1$ and so $\qd(V)=\qd(V^{*})$.
\end{lemma}
\begin{proof}
  First, we will show that for an irreducible $V\in \catU_g$ the slope $s(V)$
  depends only of the degree $g\in\Gr$.  In \cite{DPRR} it is shown that for
  any $g,g',gg'\in \Gr\setminus \X$ and any irreducible modules $V\in \catU_g,
  V'\in\catU_{g'}$ we have
  $$V\otimes V'\simeq\bigoplus_{W\in irr(\catU_{gg'})}W^{\oplus \ro^{N-\rk}}$$
  where $irr(\catU_{gg'})$ denote a representing set of the isomorphism
  classes of irreducible modules of $\catU_{gg'}$.  This implies that if
  $W_1,W_2\in\catU_g$ are irreducible modules then there exists projective
  irreducible modules $U,V$ such that both $W_1$ and $W_2$ are direct summands
  of $U\otimes V$.  Then by Equation \eqref{E:slope} we have
  $s(W_1)=s(W_2)=s(U)s(V)$.  Thus, the slope factor through a map on
  $\Gr\setminus\X$ also denoted by $s$ which satisfies $s(gg')=s(g)s(g')$ and
  $s(g^{-1})=1/s(g)$.

  Next we show that $s$ extend to a character on the whole group $\Gr$: If
  $x\in \X$, choose $g\in \Gr\setminus(\X\cup \X x^{-1})$ then we can define
  $s(x)$ as the ratio $s(gx)/s(g)$.  This is well defined: if $h\in
  \Gr\setminus(\X\cup \X x^{-1})$ then there exists $k\in \Gr\setminus(\X \cup
  \X x^{-1}\cup \X g\cup \X h)$ for which $k,kx,kg^{-1},kh^{-1}\notin\X$ and
  thus $s(kx)/s(gx)=s(kg^{-1})=s(k)/s(g)$, $s(kx)/s(hx)=s(kh^{-1})=s(k)/s(h)$
  which imply that $s(gx)/s(g)=s(kx)/s(k)=s(hx)/s(h)$.

  But now Equation \eqref{E:FormMD} implies that $T\subset\ker s$.  As the
  normal subgroup of the borel $B_+\subset\Gr$ generated by its Cartan
  subgroup $T$ is $B_+$ itself, this implies that $B_+\subset\ker s$.
  Similarly $B_-\subset\ker s$ thus $\ker s\supset B_+B_-=\Gr$.
\end{proof}

\begin{lemma}\label{UqgT-Ambi}
The pair $(\A,\qd)$ is t-ambi.
\end{lemma}
\begin{proof}
  From Lemmas \ref{L:traceProjD} and \ref{L:slope1} there exists a trace $\mt
  = \{\mt_{V}\}_{V\in \Proj}$ such that $\qd(V)=\qd(V^*)$ for all $V\in
  \Proj$.  Then from \cite{GPV} we have that $(\At,\qd)$ is t-ambi.
\end{proof}
  
Let us increase $\X$ by adding all $g\in \Gr$ with $g^2=1$.  Then
$\Gr\setminus\X$ is still a Zariski open dense subset of $\Gr$ and by Lemma
\ref{L:basicdata} there exists basic data for $\catU$.
\begin{theorem}
  The category $\catU$ of finite dimensional weight modules over $\Uqg$ is a
  $(\X,\qd)$-relative $\Gr$-spherical category which admit basic data.
\end{theorem}
\begin{proof}
  The category $\catU$ is a $\C$-linear pivotal category where the pivotal
  structure is given in Subsection \ref{SS:PivotalSt}.  The $\Gr$-grading of
  $\catU$ is defined in Subsection \ref{SS:weightmod}.  Item \eqref{ID:G-sph5}
  of Definition \ref{D:G-spherical} follows from Theorem \ref{T:catU_gSS}.
  Lemma \ref{UqgT-Ambi} implies Item \eqref{ID:G-sph6}.  Finally, since the
  objects of $\catU$ are finite dimensional vector spaces it follows from
  \cite{GPT2} that the map $\bb:\A\to \C$ defined in Equation \eqref{E:defbb}
  is well defined and satisfies all the properties of Item \eqref{ID:G-sph7}.
\end{proof}

\linespread{1}


\begin{thebibliography}{99}

\bibitem{ADO} Y. Akutsu, T. Deguchi, and T. Ohtsuki - {\em Invariants of
    colored links.} J. Knot Theory Ramifications \textbf{1} (1992), no. 2,
  161-184.

 \bibitem{BB} S. Baseilhac, R. Benedetti - {\em Quantum hyperbolic invariants
    of 3-manifolds with ${\rm PSL}(2,\mathbb C)$-characters. } Topology
   \textbf{43} (2004), no. 6, 1373--1423.
\bibitem{BB1} S. Baseilhac, R. Benedetti, - {\em Quantum hyperbolic geometry},
  Alg. Geom. Topol. 7 (2007) 845–917.

\bibitem{Bi} J. Bichon - {\em Cosovereign Hopf algebras}, J. Pure
  Appl. Algebra 157, No.2-3 (2001), 121-133.

\bibitem{Br} R. Brown - {\em Groupoids and van Kampen's theorem}, Proc. London
  Math. Soc. (3) 17 1967 385--401.

\bibitem{CP} V. Chari, A. Pressley, A guide to quantum groups, Cambridge
  University Press, Cambridge, 1994. 
  
\bibitem{CGP1} F.  Costantino, N. Geer, B. Patureau-Mirand - {\em In
    Preparation.}

\bibitem{DK} C. De Concini, V.G. Kac - {\em Representations of quantum groups
    at roots of $1$.} In {\em Operator algebras, unitary representations,
    enveloping algebras, and invariant theory.} (Paris, 1989), 471--506,
  Progr. Math., 92, Birkhauser Boston, 1990.

\bibitem{DKP} C. De Concini, V.G. Kac, C. Procesi - {\em Quantum coadjoint
    action.}  J. Amer. Math. Soc. 5 (1992), no. 1, 151--189.

\bibitem{DKP2} C. De Concini, V.G. Kac, C. Procesi - {\em Some remarkable
    degenerations of quantum groups.} Comm. Math. Phys. 157
  (1993), no. 2, 405--427.

\bibitem{DPRR} C. De Concini, C. Procesi, N. Reshetikhin, M. Rosso - {\em Hopf
    algebras with trace and representations.} Invent. Math. 161 (2005), no. 1,
  1--44.
 
\bibitem{GKT} N. Geer, R.M. Kashaev, V. Turaev - {\em Tetrahedral forms in
    monoidal categories and 3-manifold invariants}, Preprint arXiv:1008.3103.

\bibitem{GKP1} N. Geer, J. Kujawa, B. Patureau-Mirand - {\em Generalized trace
    and modified dimension functions on ribbon categories } Preprint
  arXiv:1001.0985.

\bibitem{GP1} N. Geer, B. Patureau-Mirand - {\em Multivariable link invariants
    arising from $\sll(2|1)$ and the Alexander polynomial.}  J. Pure
  Appl. Algebra 210 (2007), no. 1, 283--298.

\bibitem{GP2} N. Geer, B.Patureau-Mirand - {\em Multivariable link invariants
   arising from Lie superalgebras of type I.}   J. Knot Theory Ramifications
 19  (2010),  no. 1, 93--115. 

\bibitem{GPT1} N. Geer, B. Patureau-Mirand, V. Turaev - {\em Modified quantum
    dimensions and re-normalized link invariants.} Compos. Math.  \textbf{145}
  (2009), no. 1, 196--212.

\bibitem{GPT2} N. Geer, B. Patureau-Mirand, V. Turaev - {\em Modified
    $6j$-symbols and $3$-Manifold Invariants.} preprint	arXiv:0910.1624.

\bibitem{GPV} N. Geer, B. Patureau-Mirand, A. Virelizier - {\em Modified
    traces and 6j-symbols in pivotal categories.} in preparation.

\bibitem{Ha} A. Hatcher -
{\em Algebraic topology.} Cambridge University Press, Cambridge, 2002.

\bibitem{Kv} R.M. Kashaev - {\em A link invariant from quantum dilogarithm.}
  Modern Phys. Lett. A \textbf{10} (1995), no. 19, 1409--1418.

\bibitem{Kv94} Kashaev, R. - {\em Quantum dilogarithm as a $6j$-symbol.}
  \emph{Modern Phys. Lett. A} \textbf{9} (1994), no. 40, 3757--3768.


\bibitem{Ma} S. Majid - {\em Foundations Of Quantum Group Theory}, Cambridge,
  UK. Univ. Pr. (1995). 

\bibitem{MM} H. Murakami, J. Murakami - {\em The colored Jones polynomials and
    the simplicial volume of a knot.}  Acta Math. 186 (2001), no. 1, 85--104.

\bibitem{Neu90} W. Neumann, Combinatorics of triangulations and the
  Chern--Simons invariant for hyperbolic 3-manifolds, in: Topology '90
  (Columbus, OH, 1990), De Gruyter, Berlin, 1992.

\bibitem{Neu98} W. Neumann, Hilbert's 3rd problem and invariants of
  3-manifolds, \emph{Geometry and Topology Monographs,} Vol. 1, The Epstein
  Birthday Schrift, Paper No 19, 1998, pp. 383--411.

\bibitem{O} T. Ohtsuki - {\em Quantum invariants. A study of knots,
    3-manifolds, and their sets.} Series on Knots and Everything,
  \textbf{29}. World Scientific Publishing Co., Inc., River Edge, NJ, 2002.

\bibitem{Pe} J. Petit - {\em Decomposition of the Turaev-Viro TQFT} preprint
  (2009), arXiv:0903.4512v1.

\bibitem{Tu} V.G. Turaev - {\em Quantum invariants of knots and 3-manifolds.}
  de Gruyter Studies in Mathematics, 18. Walter de Gruyter \& Co., Berlin,
  (1994).
  
\bibitem{Tu2010} V.G. Turaev - {\em Homotopy Quantum Field Theory.} European
  Mathematical Society (2010).

\bibitem{TV} V. Turaev, O. Viro - {\em State sum invariants of $3$-manifolds
    and quantum $6j$-symbols.}  Topology \textbf{31} (1992), no. 4, 865--902.

\bibitem{TW} V. Turaev, H. Wenzl - {\em Quantum invariants of $3$-manifolds
    associated with classical simple Lie algebras.}  Internat. J. Math.
  \textbf{4} (1993), no. 2, 323--358.
\end{thebibliography}
\end{document}